\documentclass[10pt,letterpaper,twoside]{amsart}
\usepackage{charter, amsmath,amsbsy,amsfonts,amssymb,
stmaryrd,amsthm,mathrsfs,graphicx, amscd, cancel}

\usepackage[OT2,OT1]{fontenc}
 \newcommand\cyr
{
\renewcommand\rmdefault{wncyr} \renewcommand\sfdefault{wncyss} \renewcommand\encodingdefault{OT2} \normalfont
\selectfont
}
\DeclareTextFontCommand{\textcyr}{\cyr}    
\makeatletter
   \def\@settitle
{
\begin{center} \baselineskip14\p@\relax \LARGE\@title \end{center}
}
   \makeatother

\usepackage[normalem]{ulem}
\usepackage[mathscr]{eucal}
\numberwithin{equation}{section}

\usepackage[table]{xcolor}
\usepackage{verbatim}
\usepackage{amssymb}
\usepackage{amsbsy}
\usepackage{amscd}
\usepackage{amsmath}
\usepackage{amsthm, bm}
\usepackage[mathscr]{eucal}
\usepackage[all]{xy}
\usepackage{enumerate}
\usepackage{color, colortbl}
\usepackage{comment}

\input xyimport.tex

\newenvironment{spmatrix}{\left ( \begin{smallmatrix}} {\end{smallmatrix}\right)} 

\newenvironment{svmatrix}{\left | \begin{smallmatrix}} {\end{smallmatrix}\right|}

\title{\bf Key Varieties for 
Prime $\mQ$-Fano Threefolds\\
Related with $\mP^2\times \mP^2$-Fibrations. Part I}
\author{Hiromichi Takagi}

\address{Department of Mathematics, Gakushuin University, 1-5-1, Mejiro, Toshima, Tokyo, 171-8588 Japan}
\email{hiromici@math.gakushuin.ac.jp}

\theoremstyle{plain}
\newtheorem{thm}{Theorem}[section]

\newtheorem{prop}[thm]{Proposition}

\newtheorem{cor}[thm]{Corollary}
\newtheorem{lem}[thm]{Lemma}
\newtheorem{cla}[thm]{Claim}

\theoremstyle{definition}
\newtheorem{defn}[thm]{Definition}

\newtheorem*{ackn}{Acknowledgement}

\theoremstyle{remark}
\newtheorem*{rem}{Remark}

\newcommand{\sO}{\mathcal{O}}

\newcommand{\mA}{\mathbb{A}}
\newcommand{\mC}{\mathbb{C}}

\newcommand{\mN}{\mathbb{N}}
\newcommand{\mP}{\mathbb{P}}
\newcommand{\mQ}{\mathbb{Q}}
\newcommand{\mZ}{\mathbb{Z}}

\newcommand{\Bs}{\mathrm{Bs}\,}

\newcommand{\Hom}{\mathrm{Hom}}

\newcommand{\Sing}{\mathrm{Sing}\,}

\newcommand{\GL}{\mathrm{GL}}

\newcommand{\rank}{\mathrm{rank}\,}

\newcommand{\SigmaP}{\Sigma^{12}_{\mP}}
\newcommand{\SigmaA}{\Sigma^{13}_{\mA}}
\numberwithin{equation}{section}

\begin{document}

\begin{abstract}
We construct a $14$-dimensional affine variety $\Sigma^{14}_{\mA}$ with a $\GL_3$- and a $(\mC^*)^6$-actions. We denote by $\Sigma^{13}_{\mA}$ the affine variety obtained from $\Sigma^{14}_{\mA}$ by setting one specified variable to $1$ (we refer the precise definition to Definition \ref{defn:Key} of the paper).
We show that several weighted projectivizations of $\Sigma^{13}_{\mA}$ and $\Sigma^{14}_{\mA}$ produce, as weighted complete intersections, examples of prime $\mQ$-Fano threefolds of codimension four belonging to $24$ classes of the graded ring database \cite{grdb}. Except No.360 in the database, these prime $\mQ$-Fano threefolds have a Type I Tom projection. Moreover, they are not weighted complete intersections of the cluster variety of type $C_2$ introduced by Coughlan and Ducat.
We also show that a partial projectivization of $\Sigma^{14}_{\mA}$ has a $\mP^2\times \mP^2$-fibration over the affine space $\mA^9$. 
\end{abstract}
\dedicatory{{Dedicated to Professor Shigefumi Mori on the occasion of
  his 70th birthday}}
\maketitle

\markboth{Key varieties}{Hiromichi Takagi}
{\small{\tableofcontents}}

\section{\bf Introduction}

\subsection{Background}
In this paper, we work over $\mC$, the complex number field. 

A complex projective variety is called a {\it $\mQ$-Fano variety} if it is a normal variety with only terminial singularities, and its anti-canonical divisor is ample. The classification of $\mQ$-Fano threefolds is one of the central problems in Mori theory for projective threefolds. Although the classification is far from completion, several systematic classification results have been obtained so far.
In the online database \cite{grdb}, a huge table of candidates of $\mQ$-Fano threefolds are given. 
One kind of the classification results is that with respect to the codimension $\rm{ac}_X$ of a $\mQ$-Fano threefold $X$ in the weighted projective space of the minimal dimension determined by the anti-canonical graded ring of $X$. We call $\rm{ac}_X$ {\it the anti-canonical codimension}. 
The invariant $\rm{ac}_X$ is expected not to be unreasonably big for {\it prime} $\mQ$-Fano threefolds, namely, those whose canonical divisors generate the groups of numerical equivalence classes of $\mQ$-Cartier divisors. Hereafter, we focus on prime $\mQ$-Fano threefolds. For prime $\mQ$-Fano threefolds $X$ with $\rm{ac}_X\leq 3$, the classification was completed if $X$ are quasi-smooth, namely, the affine cones of $X$ in the affine spaces associated to the anti-canonical graded rings are smooth outside the origins; Fletcher \cite{Fl} and Chen-Chen-Chen \cite{CCC} did in the case of $\rm{ac}_X\leq 2$, and Alt{\i}nok \cite{Al} did in the case of $\rm{ac}_X=3$. 

Now let us consider prime $\mQ$-Fano threefolds $X$ with $\rm{ac}_X=4$. In this case, the classification is not completed mainly because there is no suitable general structure theorem for Gorenstein graded rings of codimension four (see \cite{Reidcodim4} for an attempt to obtain such a structure theorem by Reid) while there are such structure theorems in the case of codimension less than or equal to three; see \cite{codim2} in the case of codimension $\leq 2$ and \cite{BE} in the case of codimension $=3$.
In \cite{bkr}, Brown-Kerber-Reid constructed examples of prime $\mQ$-Fano threefolds $X$ with $\rm{ac}_X=4$ and with so called Type I projections (see Definition \ref{defn:typeI} for the definition) by the theory of unprojection developed by Papadakis-Reid \cite{WithoutC} (see also Pakadakis' papers \cite{PHD}, \cite{WithC}).
Recently, in \cite{CD, bigtables}, Coughlan and Ducat reproduced examples of prime $\mQ$-Fano threefolds $X$ with $\rm{ac}_X=4$ and with Type I projections in a different way as weighted complete intersections in weighted projectivizations of the cluster varieties of types $C_2$ and $G_2^{(4)}$, which are related with $\mP^2\times \mP^2$- and $\mP^1\times \mP^1\times \mP^1$-fibrations, respectively (see \cite{CD} for the definition of these cluster varieties, which we do not review in this paper). This type of construction is modeled after the classification of smooth prime Fano threefolds by Mukai \cite{Mu}. The idea of them is describing $\mQ$-Fano threefolds $X$ in appropriate projective varieties $\Sigma$ of larger dimensions and usually with larger symmetries; in the case of \cite{Mu}, $\Sigma$ are rational homogeneous spaces, and in the case of \cite{CD}, $\Sigma$ are cluster varieties. Such $\Sigma$ are informally called {\it key variety} as in the title of this paper. One pleasant byproduct of Coughlan and Ducat's key variety construction of prime $\mQ$-Fano threefolds $X$ is that even if $X$ themselves have no Type I projections, examples of $X$ in many cases can be constructed beyond \cite{bkr} from
the cluster varieties of types $C_2$ and $G_2^{(4)}$; the cluster varieties of these types themselves have Type I Tom and Jerry projections, respectively.

The purpose of this paper and the sequels is to present other key varieties for prime $\mQ$-Fano threefolds $X$ with $\rm{ac}_X=4$.
In particular, in this paper, we construct examples of prime $\mQ$-Fano threefolds $X$ with $\rm{ac}_X=4$ using our key varieties
which cannot be obtained as weighted complete intersections in weighted projectivizations of the cluster variety of type 
$C_2$ nor $G_2^{(4)}$.
In the next subsection, we define our key varieties.

\subsection{Key varieties $\Sigma^{14}_{\mA}$ and $\Sigma^{13}_{\mA}$}
\label{MainC}
\begin{defn}
\label{defn:Key}
We consider the $18$-dimensional affine space $\mA^{18}$ with coordinates
implemented in the following form with vectors and matrices: 
\begin{align*}
&\bm{p}:=\begin{spmatrix} p_1 \\ p_2 \\ p_3 \end{spmatrix},p_4,\,
\bm{q}:=\begin{spmatrix} q_1 \\ q_2 \\ q_3 \end{spmatrix}, r,u,S:=\begin{spmatrix}
s_{11} & s_{12} & s_{13}\\
s_{12} & s_{22} & s_{23}\\
s_{13} & s_{23} & s_{33}
\end{spmatrix}, \bm{t}:=\begin{spmatrix} t_1 \\ t_2 \\ t_3 \end{spmatrix}.
\end{align*}
We also set
\[
A:=\begin{spmatrix} 0 & -q_3 & q_2\\
                   q_3 & 0 & -q_1\\
                  -q_2 & q_1 & 0
\end{spmatrix}.
\]
It is useful to note the following formulas:
\[
AS=\begin{spmatrix}
-\begin{svmatrix} s_{12} & s_{13} \\ q_2 & q_3 \end{svmatrix}
& -\begin{svmatrix} s_{22} & s_{23} \\ q_2 & q_3 \end{svmatrix}
& -\begin{svmatrix} s_{23} & s_{33} \\ q_2 & q_3 \end{svmatrix}\\
\begin{svmatrix} s_{11} & s_{13} \\ q_1 & q_3 \end{svmatrix}
& \begin{svmatrix} s_{12} & s_{23} \\ q_1 & q_3 \end{svmatrix}
& \begin{svmatrix} s_{13} & s_{33} \\ q_1 & q_3 \end{svmatrix}\\
-\begin{svmatrix} s_{11} & s_{12} \\ q_1 & q_2 \end{svmatrix}
& -\begin{svmatrix} s_{12} & s_{22} \\ q_1 & q_2 \end{svmatrix}
& -\begin{svmatrix} s_{13} & s_{23} \\ q_1 & q_2 \end{svmatrix}
\end{spmatrix},\, 
A\bm{t}=\begin{spmatrix}
-\begin{svmatrix} q_2 & q_3\\ t_2 & t_3 \end{svmatrix}\\
\begin{svmatrix} q_1 & q_3\\ t_1 & t_3 \end{svmatrix}\\
-\begin{svmatrix} q_1 & q_2\\ t_1 & t_2 \end{svmatrix}
\end{spmatrix}.
\]
\begin{enumerate}[(1)]
\item
In the affine space $\mA^{18}$ with the above coordinates, we define an affine scheme $\Sigma^{14}_{\mA}$ by the following nine equations in 
(\ref{Key1}) and (\ref{Key2}):
\begin{align}
&\begin{cases}
{\empty^t \bm{p}}{\bm{q}}=0,\\
\left(rI+AS\right)\bm{p}=-p_4 A\bm{t},\text{where $I$ is the $3\times 3$ identity matrix}, \\
{\empty^t \bm{p}}S\bm{p}+p_4 {\empty^t \bm{p}}\bm{t}=0,
\end{cases}\label{Key1}\\
&\begin{cases}
u\bm{p}=\left(rI-AS\right)A\bm{t}, \\ 
up_4=-\left(r^2+{\empty^t \bm{q}}S^{\dagger} \bm{q}\right),
\end{cases}\label{Key2}
\end{align}
where $S^{\dagger}$ is the adjoint matrix of $S$.
We define the following four polynomials $F_1,\dots, F_4$:
\begin{equation}
\begin{spmatrix}
F_1\\ F_2 \\F_3
\end{spmatrix}:=\left(rI-AS\right)A\bm{t},\, F_4:=-\left(r^2+{\empty^t \bm{q}}S^{\dagger} \bm{q}\right),\label{F1F4}
\end{equation}

\item
We denote by $\Sigma^{13}_{\mA}$ the affine scheme obtained from
$\Sigma^{14}_{\mA}$ by setting $s_{33}=1$.
\end{enumerate}
\end{defn}
\vspace{3pt}
It is straightforward to see that
the five equations
in (\ref{Key1})
are the vanishing of
the five $4\times 4$ Pfaffians of the following skew-symmetric matrix:

\begin{equation}\label{eq:Pf}
\begin{spmatrix}
0 & q_3 & -q_2 & q_1 & r\\
  & 0 & p_1 & p_2 & a_{25}\\
  &   & 0 & p_3 & a_{35}\\
  &   &   & 0         & a_{45}\\
  &   &   &           & 0 
\end{spmatrix},
\end{equation}
where
\[
\begin{spmatrix}
a_{25}\\
a_{35}\\
a_{45}
\end{spmatrix}
=\begin{spmatrix}
s_{13} & s_{23} & s_{33}\\
-s_{12} & -s_{22} & -s_{23}\\
s_{11} & s_{12} & s_{13}\\
\end{spmatrix}
\begin{spmatrix}
p_1\\
p_2\\
p_3
\end{spmatrix}
+p_4
\begin{spmatrix}
t_3\\
-t_2\\
t_1
\end{spmatrix}.
\]
Thus the affine scheme $\Sigma_{\rm{Tom}}$ defined by these five $4\times 4$ Pfaffians
is a closed subscheme of 
the generic integral Tom defined as in \cite[\S 5.1]{WithC},
where the variables
$x_1,x_2,x_3,x_4$ (resp.~$z_1,z_2,z_3,z_4$) there correspond to $q_3, -q_2,q_1, r$ (resp.~$p_1, p_2,p_3, p_4$) in our format, and
$a^k_{ij}$'s there correspond to $0$, $1$, $s_{kl}$'s, and $t_m$'s in our format with suitable orders. 
Moreover, 
by straightforward calculations following the recipe as in \cite[\S 5.3]{WithC},
we see that the polynomials $g_1,\dots, g_4$ there correspond to
$F_1,\dots, F_4$ in our format after changing their signs and orders. 
Using these descriptions of $\Sigma^{14}_{\mA}$ and 
following the proof of \cite[Thm.~5.6]{WithC},
we see in the subsection \ref{sub:Tom14}
(Proposition \ref{prop:Tom14}) that $\Sigma^{14}_{\mA}$ is a Gorenstein variety of codimension four in $\mA^{18}$.
Moreover, we show that $\Sigma^{14}_{\mA}$ has only terminal singularities (Proposition \ref{prop:SigmaTerm}).

We also see that $\Sigma^{13}_{\mA}$ has similar properties since the open subset $\Sigma^{14}_{\mA}\cap \{s_{33}\not =0\}$ of $\Sigma^{14}_{\mA}$ is isomorphic to $\Sigma^{13}_{\mA}\times \mA^{1*}$ (Lemma \ref{lem:1314}).

\subsection{Main theorem}

The main result of this paper is the following theorem, for which we use 
for the classes of $\mQ$-Fano threefolds
the assigned numbers in the database \cite{grdb}. We prepare one convention;
by a {\it numerical data} for a prime $\mQ$-Fano threefold $X$ of No.$*$, we mean the quadruplet consisting of the Fano index of $X$, the genus $g(X):=h^0(-K_X)-2$, the basket of singularities of $X$, and the Hilbert numerator of $X$ with respect to $-K_X$, which are given in the database \cite{grdb} (note that the degree of $X$ with respect to $-K_X$ can be read off from the Hilbert numerator of $X$ once we know the ambient weighted projective space of $X$).

\begin{thm}
\label{thm:main}
The following assertions hold$:$
\begin{enumerate}[$(1)$]
\item \begin{enumerate}[$({1}\text{-}1)$]
\item For each of $23$ numerical data except No.24078 (resp.~for No.24078) with the numbers as in Tables \ref{Table 1} and \ref{Table 2}, there exists a quasi-smooth prime $\mQ$-Fano threefold $X$ with $\rm{ac}_X=4$ obtained as a weighted complete intersection from the weighted projectivization 
$\Sigma^{12}_{\mP}$ (resp.$\Sigma^{13}_{\mP}$) of $\Sigma^{13}_{\mA}$ (resp.~$\Sigma^{14}_{\mA}$) as in Table \ref{Table 2} with the weights of coordinates given in Table \ref{Table 1}. 
\item Any such an $X$ not of No.360, No.1218 nor No.24078 (resp. for No.24078) has the Type I Tom projection induced by the projection of $\Sigma^{12}_{\mP}$ (resp.$\Sigma^{13}_{\mP}$) from the $u$-point, where the $u$-point means the coordinate point of $\Sigma^{12}_{\mP}$ or $\Sigma^{13}_{\mP}$ such that only the coordinate $u$ is nonzero (we use this convention for other coordinates).
\item Any such an $X$ of No.1218 has the Type I Tom projection from each of the two $1/5(1,2,3)$-singularities.
\item Any such an $X$ of No.360 has the projection to a codimension two weighted complete intersection $(14,16)\subset \mP(1,4,5,6,7,8)$ with the unique $1/7 (1,2,5)$-singularity being the center; the $1/7 (1,2,5)$-singularity coincides with the $p_1$-point and the projection is induced by the projection of $\Sigma^{12}_{\mP}$ with the $p_1$-point being the center. 
\end{enumerate}
\item 
A general member $T$ of $|-K_X|$ for a general $X$ as in (1) is a quasi-smooth $K3$ surface with only Du Val singularities of type $A$.
Moreover, $T$ has only singularities at points where $X$ is singular and if $X$ has a $1/\alpha\, (\beta,-\beta,1)$-singularity at a point for some $\alpha,\beta\in \mN$ with $(\alpha,\beta)=1$, then $T$ has a $1/\alpha\, (\beta,-\beta)$-singularity there.
\end{enumerate}

\renewcommand{\arraystretch}{2}

{{
\begin{table}[p]
\rowcolors{2}{}{yellow!20}
\caption{Weights of coordinates}
\label{Table 1}
\begin{tabular}{|c|c|c|c|c|c|c|c|} \hline
No.& 
$w(\bm{p})$ & 
$w(p_4)$ & 
$w(\bm{q})$ &
$w(r)$ &
$w(u)$ &
$w(S)$ &
$w(\bm{t})$
\\ \hline \hline
360 &
$\begin{spmatrix}
7 \\ 8 \\ 9 
\end{spmatrix}$ & 
$6$ & 
$\begin{spmatrix}
6 \\5  \\ 4 
\end{spmatrix}$ &
$7$ &
$8$ &
$\begin{spmatrix}
4 & 3 & 2 \\
         &2  &1 \\
         &           & 0 
\end{spmatrix}$ &
$\begin{spmatrix}
 5 \\4  \\3 
\end{spmatrix}$\\\hline
393 &
$\begin{spmatrix}
6 \\ 7 \\ 8 
\end{spmatrix}$ & 
$5$ & 
$\begin{spmatrix}
6 \\5  \\ 4 
\end{spmatrix}$ &
$7$ &
$9$ &
$\begin{spmatrix}
4 & 3 & 2 \\
         &2  &1 \\
         &           & 0 
\end{spmatrix}$ &
$\begin{spmatrix}
 5 \\4  \\3 
\end{spmatrix}$\\\hline
569 &
$\begin{spmatrix}
5 \\ 6 \\ 7 
\end{spmatrix}$ & 
$3$ & 
$\begin{spmatrix}
5 \\4  \\ 3 
\end{spmatrix}$ &
$6$ &
$9$ &
$\begin{spmatrix}
4 & 3 & 2 \\
         &2  &1 \\
         &           & 0 
\end{spmatrix}$ &
$\begin{spmatrix}
 6 \\5  \\4 
\end{spmatrix}$\\\hline
574&
$\begin{spmatrix}
5 \\6  \\ 7 
\end{spmatrix}$ & 
$5$ & 
$\begin{spmatrix}
 5 \\ 4 \\3 
\end{spmatrix}$ &
$6$ &
$7$ &
$\begin{spmatrix}
4 & 3 & 2\\
        & 2 &1 \\
         &           &0 
\end{spmatrix}$ &
$\begin{spmatrix}
 4\\ 3 \\ 2 
\end{spmatrix}$\\\hline
642 &
$\begin{spmatrix}
4 \\5  \\6 
\end{spmatrix}$ & 
$3$ & 
$\begin{spmatrix}
6 \\ 5 \\ 4
\end{spmatrix}$ &
$7$ &
$11$ &
$\begin{spmatrix}
4 & 3 & 2 \\
         & 2 &1 \\
         &   & 0
\end{spmatrix}$ &
$\begin{spmatrix}
5 \\4  \\3 
\end{spmatrix}$\\\hline
644 &
$\begin{spmatrix}
4 \\5  \\6 
\end{spmatrix}$ & 
$4$ & 
$\begin{spmatrix}
6 \\5  \\4 
\end{spmatrix}$ &
$7$ &
$10$ &
$\begin{spmatrix}
4 & 3 & 2\\
         & 2 &1 \\
         &  & 0
\end{spmatrix}$ &
$\begin{spmatrix}
4 \\3  \\2 
\end{spmatrix}$\\\hline
1091 &
$\begin{spmatrix}
6 \\ 7 \\ 8 
\end{spmatrix}$ & 
$5$ & 
$\begin{spmatrix}
6  \\5  \\ 4 
\end{spmatrix}$ &
$7$ &
$9$ &
$\begin{spmatrix}
4 & 3 & 2\\
         & 2 & 1\\
         &  &0 
\end{spmatrix}$ &
$\begin{spmatrix}
5 \\4  \\3 
\end{spmatrix}$\\\hline
1181 &
$\begin{spmatrix}
3 \\4  \\5 
\end{spmatrix}$ & 
$2$ & 
$\begin{spmatrix}
6 \\5  \\4 
\end{spmatrix}$ &
$7$ &
$12$ &
$\begin{spmatrix}
4& 3 & 2\\
         & 2 & 1\\
         &  & 0
\end{spmatrix}$ &
$\begin{spmatrix}
5 \\4  \\3 
\end{spmatrix}$\\\hline
1185 &
$\begin{spmatrix}
4 \\5  \\6 
\end{spmatrix}$ & 
$2$ & 
$\begin{spmatrix}
4 \\ 3 \\2 
\end{spmatrix}$ &
$5$ &
$8$ &
$\begin{spmatrix}
4 & 3 & 2 \\
         & 2 & 1 \\
         &  &0 
\end{spmatrix}$ &
$\begin{spmatrix}
6 \\5  \\4 
\end{spmatrix}$\\\hline

1186 &
$\begin{spmatrix}
4 \\5  \\6 
\end{spmatrix}$ & 
$3$ & 
$\begin{spmatrix}
4 \\3  \\2 
\end{spmatrix}$ &
$5$ &
$7$ &
$\begin{spmatrix}
4& 3 & 2\\
         & 2 & 1\\
         &  & 0
\end{spmatrix}$ &
$\begin{spmatrix}
5 \\4  \\3 
\end{spmatrix}$\\\hline
1218 &
$\begin{spmatrix}
4 \\ 5 \\6 
\end{spmatrix}$ & 
$5$ & 
$\begin{spmatrix}
4 \\3  \\2 
\end{spmatrix}$ &
$5$ &
$5$ &
$\begin{spmatrix}
4& 3 & 2\\
         & 2 & 1\\
         &           &0 
\end{spmatrix}$ &
$\begin{spmatrix}
3 \\2  \\1 
\end{spmatrix}$\\\hline
1253 &
$\begin{spmatrix}
4 \\4  \\5 
\end{spmatrix}$ & 
$3$ & 
$\begin{spmatrix}
 4\\ 4  \\ 3 
\end{spmatrix}$ &
$5$ &
$7$ &
$\begin{spmatrix}
2& 2 & 1\\
        & 2 & 1\\
         &           &0 
\end{spmatrix}$ &
$\begin{spmatrix}
3 \\3  \\2 
\end{spmatrix}$\\\hline
1413 &
$\begin{spmatrix}
4 \\4  \\5 
\end{spmatrix}$ & 
$3$ & 
$\begin{spmatrix}
3 \\3  \\2 
\end{spmatrix}$ &
$4$&
$5$ &
$\begin{spmatrix}
2& 2 & 1\\
         &2  &1 \\
         &           &0 
\end{spmatrix}$ &
$\begin{spmatrix}
 3\\3  \\2 
\end{spmatrix}$\\\hline
2422
&
$\begin{spmatrix}
2 \\3  \\4 
\end{spmatrix}$ & 
$3$ & 
$\begin{spmatrix}
4 \\3  \\2 
\end{spmatrix}$ &
$5$ &
$7$ &
$\begin{spmatrix}
4& 3 & 2\\
         &2  & 1\\
         &           &0 
\end{spmatrix}$ &
$\begin{spmatrix}
3 \\2  \\1 
\end{spmatrix}$\\\hline
4850&
$\begin{spmatrix}
4 \\5  \\6 
\end{spmatrix}$ & 
$1$ & 
$\begin{spmatrix}
6 \\5  \\4 
\end{spmatrix}$ &
$7$ &
$13$ &
$\begin{spmatrix}
4& 3 & 2\\
         & 2 & 1\\
         &           &0 
\end{spmatrix}$ &
$\begin{spmatrix}
7 \\6  \\5 
\end{spmatrix}$\\\hline
4938&
$\begin{spmatrix}
3 \\4  \\5 
\end{spmatrix}$ & 
$1$ & 
$\begin{spmatrix}
5 \\4  \\3 
\end{spmatrix}$ &
$6$ &
$11$ &
$\begin{spmatrix}
4& 3 & 2\\
         & 2 & 1\\
         &           &0 
\end{spmatrix}$ &
$\begin{spmatrix}
6 \\5  \\4 
\end{spmatrix}$\\\hline
5202&
$\begin{spmatrix}
2 \\3  \\4 
\end{spmatrix}$ & 
$1$ & 
$\begin{spmatrix}
4 \\3  \\2 
\end{spmatrix}$ &
$5$&
$9$&
$\begin{spmatrix}
4& 3 & 2\\
         & 2 & 1\\
         &           &0 
\end{spmatrix}$ &
$\begin{spmatrix}
5 \\4  \\3 
\end{spmatrix}$\\\hline
5859&
$\begin{spmatrix}
1 \\2  \\3 
\end{spmatrix}$ & 
$2$ & 
$\begin{spmatrix}
4 \\3  \\2 
\end{spmatrix}$ &
$5$&
$8$ &
$\begin{spmatrix}
4& 3 & 2\\
         & 2 &1 \\
         &           &0 
\end{spmatrix}$ &
$\begin{spmatrix}
3 \\2  \\1 
\end{spmatrix}$\\\hline
5866&
$\begin{spmatrix}
2 \\2  \\3 
\end{spmatrix}$ & 
$1$ & 
$\begin{spmatrix}
 3\\3  \\2 
\end{spmatrix}$ &
$4$&
$7$&
$\begin{spmatrix}
2& 2 &1 \\
         &2  &1 \\
         &           &0 
\end{spmatrix}$ &
$\begin{spmatrix}
3 \\3  \\2 
\end{spmatrix}$\\\hline
6860&
$\begin{spmatrix}
2 \\2  \\3 
\end{spmatrix}$ & 
$1$ & 
$\begin{spmatrix}
 2\\2  \\1 
\end{spmatrix}$ &
$3$&
$5$&
$\begin{spmatrix}
2& 2 & 1\\
         &2  &1 \\
         &           &0 
\end{spmatrix}$ &
$\begin{spmatrix}
3 \\3  \\2 
\end{spmatrix}$\\\hline
6865&
$\begin{spmatrix}
2 \\ 2 \\3 
\end{spmatrix}$ & 
$2$ & 
$\begin{spmatrix}
2 \\2  \\1 
\end{spmatrix}$ &
$3$&
$4$ &
$\begin{spmatrix}
2& 2 &1 \\
         &2  &1 \\
         &           &0 
\end{spmatrix}$ &
$\begin{spmatrix}
2 \\2  \\1 
\end{spmatrix}$\\\hline
11004&
$\begin{spmatrix}
1 \\2  \\3 
\end{spmatrix}$ & 
$1$ & 
$\begin{spmatrix}
3 \\2  \\1 
\end{spmatrix}$ &
$4$ &
$7$ &
$\begin{spmatrix}
4& 3 & 2\\
         & 2 & 1\\
         &           &0 
\end{spmatrix}$ &
$\begin{spmatrix}
4 \\3  \\2 
\end{spmatrix}$\\\hline
16227&
$\begin{spmatrix}
1 \\1  \\2 
\end{spmatrix}$ & 
$1$& 
$\begin{spmatrix}
2 \\2  \\1 
\end{spmatrix}$ &
$3$&
$5$&
$\begin{spmatrix}
2& 2 & 1\\
         & 2 & 1\\
         &           &0 
\end{spmatrix}$ &
$\begin{spmatrix}
2 \\2  \\1 
\end{spmatrix}$\\\hline
24078&
$\begin{spmatrix}
1 \\1  \\1 
\end{spmatrix}$ & 
$1$ & 
$\begin{spmatrix}
1 \\1  \\1 
\end{spmatrix}$ &
$2$&
$3$ &
$\begin{spmatrix}
1& 1 & 1\\
         & 1 & 1\\
         &           &1 
\end{spmatrix}$ &
$\begin{spmatrix}
1 \\1  \\1 
\end{spmatrix}$\\\hline
\end{tabular}
\label{tab:keyvarwt}
\end{table}
}}

\vspace{3pt}

\noindent {\bf Notation for Tables $\ref{Table 1}$ and $\ref{Table 2}$.}
\begin{align*}
&w(*):=\text{the weight of the coordinate $*$},\\
&w(\bm{p}):=\begin{spmatrix}
w(p_1) \\ w(p_2) \\ w(p_3)
\end{spmatrix},
w(\bm{q}):=\begin{spmatrix}
w(q_1) \\ w(q_2) \\ w(q_3)
\end{spmatrix},\\
&w(S):=\begin{spmatrix}
w(s_{11})& w(s_{12}) & w(s_{13})\\
         & w(s_{22}) & w(s_{23})\\
         &           & w(s_{33})
\end{spmatrix},
w(\bm{t}):=\begin{spmatrix}
w(t_1) \\ w(t_2)\\ w(t_3)
\end{spmatrix},\\
& \mP_X:=\text{the ambient weighted projective space of $X$},\\
& \Sigma \ \text{in the 4th column of Table $\ref{Table 2}:=$ the abbreviation of $\Sigma^{12}_{\mP}$ or $\Sigma^{13}_{\mP}$}.
\end{align*}
{\renewcommand{\arraystretch}{2}
\begin{table}[b]
\rowcolors{2}{}{yellow!20}
\caption{Descriptions of $\mQ$-Fano threefolds $X$}\label{}
\label{Table 2}
\centering
\scalebox{0.8}
{
\begin{tabular}{|c|c|c|c|} \hline
No.& 
$\mP_X$  & 
Baskets of singularities & 
$X\subset \Sigma$ 
\\ \hline \hline
360& $\mP(1,4,5,6,7^2,8,9)$ & $\{2\times 1/4(1,1,3),1/6(1,1,5),1/7(1,2,5)\}$ & $\Sigma\cap (2)^2\cap (3)^2\cap (4)^2\cap (5)\cap (6)\cap (8)$ \\
\hline
393& $\mP(1,4,5^2,6,7,8,9)$ & $\{1/2(1,1,1),1/5(1,1,4),1/5(1,2,3),1/9(1,4,5)\}$ & $\Sigma\cap (2)^2\cap (3)^2\cap (4)^2\cap (5)\cap (6)\cap (7)$  \\
\hline
569& $\mP(1,3,4,5^2,6,7,9)$ & $\{2\times 1/3(1,1,2),1/5(1,2,3),1/9(1,4,5)\}$ & $\Sigma\cap (2)^2\cap (3)^2\cap (4)^2\cap (5)\cap (6)^2$ \\
\hline
574& $\mP(1,3,4,5^2,6,7^2)$ & $\{1/3(1,1,2),1/5(1,1,4),1/5(1,2,3),1/7(1,3,4)\}$ & $\Sigma\cap (2)^3\cap (3)^2\cap (4)^2\cap (5)\cap (6)$  \\
\hline
642& $\mP(1,3,4^2,5,6,7,11)$ & $\{1/2(1,1,1),1/3(1,1,2),1/4(1,1,3),1/11(1,4,7)\}$ & $\Sigma\cap (2)^2\cap (3)^2\cap (4)^2\cap (5)^2\cap (6)$  \\
\hline
644& $\mP(1,3,4^2,5,6,7,10)$ & $\{1/2(1,1,1),2\times 1/4(1,1,3),1/10(1,3,7)\}$ & $\Sigma\cap (2)^3\cap (3)\cap (4)^3\cap (5)\cap (6)$  \\
\hline
1091& $\mP(1,2,5,6,7^2,8,9)$ & $\{2\times 1/2(1,1,1),1/7(1,1,6),1/9(1,2,7)\}$ & $\Sigma\cap (2)\cap (3)^2\cap (4)^3\cap (5)^2\cap (6)$  \\
\hline
1181& $\mP(1,2,3,4,5^2,7,12)$ & $\{1/2(1,1,1),1/4(1,1,3),1/12(1,5,7)\}$ & $\Sigma\cap (2)^2\cap (3)^2\cap (4)^3\cap (5)\cap (6)$  \\
\hline
1185& $\mP(1,2,3,4,5^2,6,8)$ & $\{3\times 1/2(1,1,1),1/5(1,1,4),1/8(1,3,5)\}$ & $\Sigma\cap (2)^3\cap (3)\cap (4)^3\cap (5)\cap (6)$  \\
\hline
1186& $\mP(1,2,3,4,5^2,6,7)$ & $\{2\times 1/2(1,1,1),1/3(1,1,2),1/5(1,1,4),1/7(1,2,5)\}$ & $\Sigma\cap (2)^2\cap (3)^3\cap (4)^3\cap (5)$  \\
\hline
1218& $\mP(1,2,3,4,5^3,6)$ & $\{2\times 1/2(1,1,1),1/5(1,1,4),2\times 1/5(1,2,3)\}$ & $\Sigma\cap (1)\cap (2)^3\cap (3)^2\cap (4)^2\cap (5)$  \\
\hline
1253& $\mP(1,2,3,4^2,5^2,7)$ & $\{2\times 1/2(1,1,1),2\times 1/4(1,1,3),1/7(1,2,5)\}$ & $\Sigma\cap (1)\cap (2)^3\cap (3)^3\cap (4)^2$  \\
\hline
1413& $\mP(1,2,3^2,4^2,5^2)$ & $\{2\times 1/2(1,1,1),2\times 1/3(1,1,2),1/4(1,1,3), 1/5(1,2,3)\}$ & $\Sigma\cap (1)\cap (2)^4\cap (3)^3\cap (4)$  \\
\hline
2422& $\mP(1,2^2,3^2,4,5,7)$ & $\{4\times 1/2(1,1,1),1/3(1,1,2),1/7(1,2,5)\}$ & $\Sigma\cap (1)\cap (2)^3\cap (3)^3\cap (4)^2$  \\
\hline\hline
4850& $\mP(1^2,4,5,6^2,7,13)$ & $\{1/2(1,1,1),1/13(1,6,7)\}$ & 
$\Sigma\cap (2)^2\cap (3)\cap (4)^2\cap (5)^2\cap (6)\cap (7)$  \\
\hline
4938& $\mP(1^2,3,4,5^2,6,11)$ & $\{1/3(1,1,2),1/11(1,5,6)\}$ & 
$\Sigma\cap (2)^2\cap (3)^2\cap (4)^3\cap (5)\cap (6)$  \\
\hline
5202& $\mP(1^2,2,3,4^2,5,9)$ & $\{2\times 1/2(1,1,1),1/9(1,4,5)\}$ & $\Sigma\cap (2)^3\cap (3)^3\cap (4)^2\cap (5)$  \\
\hline
5859& $\mP(1^2,2^2,3^2,5,8)$ & $\{2\times 1/2(1,1,1),1/8(1,3,5)\}$ & $\Sigma\cap (1)\cap (2)^4\cap (3)^2\cap (4)^2$  \\
\hline
5866& $\mP(1^2,2^2,3^2,4,7)$ & $\{3\times 1/2(1,1,1),1/7(1,3,4)\}$ & $\Sigma\cap (1)\cap (2)^5\cap (3)^3$  \\
\hline
6860& $\mP(1^2,2^3,3^2,5)$ & $\{4\times 1/2(1,1,1), 1/5(1,2,3)\}$ & $\Sigma\cap (1)^2\cap (2)^5\cap (3)^2$  \\
\hline
6865& $\mP(1^2,2^3,3^2,4)$ & $\{5\times 1/2(1,1,1),1/4(1,1,3)\}$ & $\Sigma\cap (1)^2\cap (2)^7$  \\
\hline\hline
11004& $\mP(1^3,2,3^2,4,7)$ & $\{1/7(1,3,4)\}$ & $\Sigma\cap (1)\cap (2)^4\cap (3)^2\cap (4)^2$  \\
\hline
16227& $\mP(1^4,2^2,3,5)$ & $\{1/5(1,2,3)\}$ & $\Sigma\cap (1)^3\cap (2)^6$  \\
\hline
24078& $\mP(1^6,2,3)$ & $\{1/3(1,1,2)\}$ & $\Sigma\cap (1)^{10}$  \\
\hline
\end{tabular}
}
\end{table}
}
\end{thm}

\begin{rem}
\begin{enumerate}
\item
In \cite{grdb}, the candidates of the projection and the unprojection for each $\mQ$-Fano threefold are also given. It is an interesting problem to check that each $X$ as constructed in Theorem \ref{thm:main} (1) has all of them.
\item
Possibly a general $T$ as in Theorem \ref{thm:main} (2) has Picard number one. To settle this problem, we need more careful analysis of $T$.
\end{enumerate}
\end{rem}

In the forthcoming paper \cite{PartII}, we construct a 15-dimensional affine variety $\Pi_{\mathbb{A}}^{15}$ related with $\mathbb{P}^{2}\times\mathbb{P}^{2}$-fibration and produce prime $\mathbb{Q}$-Fano threefolds of anti-canonical codimension four belonging to the eight classes No.308, 501, 512, 550, 577, 872, 878, and 1766 of [GRDB] from weighted projectivizations of $\Pi_{\mathbb{A}}^{15}$. The construction of $\Pi_{\mathbb{A}}^{15}$ is based on a certain type of unprojection and is inspired by R.Taylor's thesis \cite[\S 5.2]{Tay} submitted to University of Warwick.  In that paper, to study $\Pi^{15}_{\mA}$, we also introduce another $13$-dimensional affine variety $H_{\mathbb{A}}^{13}$ related with $\mathbb{P}^{2}\times\mathbb{P}^{2}$-fibration. In a further work, we show that $H_{\mathbb{A}}^{13}$ contains the cluster variety of type $C_2$ as a closed subvariety, and show that several weighted projectivizations of closed subvarieties of $H_{\mathbb{A}}^{13}$ produce more general prime $\mQ$-Fano threefolds than those obtained from the cluster variety of type $C_2$. Moreover, a weighted projectivizations of $H^{13}_{\mA}$ itself produces also a prime $\mQ$-Fano threefold of No.20652, which was not obtained from the cluster variety of type $C_2$ (note that it was constructed in \cite{bkr} since it has Type I projections) .

We also plan to construct prime $\mQ$-Fano threefolds $X$ with $\rm{ac}_X=4$ using key varieties related with $\mP^1\times \mP^1\times\mP^1$-fibration.

\subsection{Structure of the paper}
The aim of this paper is twofold; the constructions of $\mQ$-Fano threefolds as in Theorem \ref{thm:main}, and the descriptions of the key varieties $\Sigma^{13}_{\mA}$ and $\Sigma^{14}_{\mA}$. The sections \ref{14}--\ref{sec:proj} and  the sections \ref{sec:Up}--\ref{More14} are devoted to the first and  the second aim respectively. 

In the sections \ref{14} and \ref{13}, we give descriptions of $\Sigma^{14}_{\mA}$ and $\Sigma^{13}_{\mA}$ respectively which are needed for the constructions of $\mQ$-Fano threefolds as in Theorem \ref{thm:main}.
In the subsection \ref{Charts}, we describe several open subset of $\Sigma^{14}_{\mA}$ called charts and determine the singular locus of $\Sigma^{14}_{\mA}$. The results play important roles to determine singularities of prime $\mQ$-Fano threefolds and $\Sigma^{14}_{\mA}$  in the section \ref{sec:Example} and the subsection \ref{sub:More} respectively. In the subsection \ref{sub:Tom14}, using unprojection technique, we show that
$\Sigma^{14}_{\mA}$ is irreducible, reduced, Gorenstein and of codimension four in the ambient affine space, and the affine coordinate ring of $\Sigma^{14}_{\mA}$ is a UFD (Proposition \ref{prop:Tom14}). In this subsection, we also show by Proposition \ref{prop:Tom14} that the weighted projectivization $\Sigma^{13}_{\mP}$ of $\Sigma^{14}_{\mA}$ with weights of coordinates is $\mQ$-factorial and has Picard number one (Corollary \ref{SigmaDiv}) and give a sufficient condition for $\mQ$-Fano threefolds we will construct from $\Sigma^{14}_{\mA}$ to be $\mQ$-factorial and to have Picard number one (Corollary \ref{QFanoDiv}). In the section \ref{13}, we obtain similar results for $\Sigma^{13}_{\mA}$ and for $\mQ$-Fano threefolds we will construct from $\Sigma^{13}_{\mA}$ (Proposition \ref{SigmaQFanoDiv13}), which easily follows since 
the open subset $\Sigma^{14}_{\mA}\cap \{s_{33}\not =0\}$ of $\Sigma^{14}_{\mA}$ is isomorphic to $\Sigma^{13}_{\mA}\times \mA^{1*}$
(Lemma \ref{lem:1314}). In the section \ref{sec:weight}, we give a method to search weights of coordinates of $\Sigma^{13}_{\mA}$ and $\Sigma^{14}_{\mA}$ for which the corresponding weighted projectivizations of them will produce $\mQ$-Fano threefolds as in Theorem \ref{thm:main} (1) (the subsection \ref{sub:posweight}). In the subsection \ref{sub:916}, we obtain a so-called $9\times 16$ free resolutions of the affine coordinate ring of $\Sigma^{13}_{\mA}$ and $\Sigma^{14}_{\mA}$ (Proposition \ref{prop:916}). We can read off the Hilbert series of the affine coordinate rings of $\Sigma^{13}_{\mA}$ and $\Sigma^{14}_{\mA}$ from these free resolutions. From these data, we prove a part of Theorem \ref{thm:main} (1) (Corollaries \ref{cor:candivdeg} and \ref{cor:ac4}). In the sections \ref{sec:Example} and \ref{sec:proj}, we complete the proof of Theorem \ref{thm:main}. In the section \ref{sec:Example}, we check that the candidates of $\mQ$-Fano threefolds as in Theorem \ref{thm:main} (1) which are weighted complete intersections of 
the weighted projectivizations $\Sigma^{12}_{\mP}$ and $\Sigma^{13}_{\mP}$ of $\Sigma^{13}_{\mA}$ and $\Sigma^{14}_{\mA}$ with suitable weights of coordinates have prescribed singularities and Picard number one. We also prove the assertion (2) of Theorem \ref{thm:main} as for singularities of anti-canonical members of $\mQ$-Fano threefolds. In the subsection \ref{sub:strategy}, we explain in detail our strategy to prove these. The proof is undertaken in the subsections \ref{Case1}--\ref{Case3} though repeated arguments are often omitted and only the results are written down. Finally, in the section \ref{sec:proj}, we describe some projections of $\mQ$-Fano threefolds as in Theorem \ref{thm:main} (1-2)--(1-4), which complete our proof of Theorem \ref{thm:main}.

In the section \ref{sec:Up}, we introduce another $14$-dimensional affine variety $\Upsilon^{14}_{\mA}$ pushing forward the consideration in the subsection \ref{sub:TypeII1}. The affine variety $\Upsilon^{14}_{\mA}$ was constructed independently in R.~Taylor's thesis \cite{Tay} and was studied more in detail. We refer to the beginning of the section \ref{sec:Up} for detailed background. In the section \ref{More14}, we prove the existence of a $\GL_3$- and a $(\mC^*)^6$-actions on $\Sigma^{14}_{\mA}$ (the subsection \ref{sub:sl3}), the existence of a $\mP^2\times \mP^2$-fibration for a partial projectivization of $\Sigma^{14}_{\mA}$ (the subsection \ref{fib13}), and give descriptions of singularities of $\Sigma^{13}_{\mA}$ and $\Sigma^{14}_{\mA}$ (the subsection \ref{sub:More}). In the subsection \ref{sub:dual}, we amplify the properties of the above mentioned $\mP^2\times \mP^2$-fibration by explaining its relation with the classical duality of conics.

\begin{ackn}
The author owes many important calculations in the paper to Professor Shinobu Hosono. The author wishes to thank him for his generous cooperations. The author is grateful to Professor Stavros Papadakis for answering questions about Type $\rm{II}_1$ projections. The author wishes to thank Professor Gavin Brown for clarifying my understanding of \cite{grdb}. The author is very grateful to Professor Shigeru Mukai for telling his inference about the dimension of key varieties, which helped the author to find out the key varieties in this paper.  This work is supported in part by Grant-in Aid for Scientific Research (C) 16K05090.
\end{ackn}

\section{\bf Affine variety $\Sigma_{\mA}^{14}$}~
\label{14}
\subsection{Charts and singular locus}
\label{Charts}
For a coordinate $*$, we call the open subset of $\Sigma^{14}_{\mA}$ with 
$*\not =0$ {\it the $*$-chart}.
We describe the $*$-chart such that $*$ is one of the entries of $\bm{p}$, $\bm{q}$, $p_4$, $r$, $u$.

\vspace{3pt}

\noindent \underline{\bf $p_1$-chart}: 
We note that $\Sigma^{14}_{\mA}\cap \{p_1\not=0\}$ is isomorphic to $(\Sigma^{14}_{\mA}\cap \{p_1=1\})\times \mA^{1*}$
by the map $(\bm{p},p_4, \bm{q},r,u,S,\bm{t})\mapsto  \left((p_1^{-1}\bm{p},p_1^{-1}p_4, p_1^{-1}\bm{q},p_1^{-1}r,p_1^{-1}u,S,\bm{t}), p_1\right)$. This is because all the equations of $\Sigma^{14}_{\mA}$ are quadratic when we consider the entries of $S$ and $\bm{t}$ are constants.
Therefore it suffices to describe $\Sigma^{14}_{\mA}\cap \{p_1=1\}$.
Solving regularly the equations of $\Sigma^{14}_{\mA}$ (\ref{Key1}) and (\ref{Key2}) and setting $p_1=1$, we have the following:
\begin{align*}
&q_1=-p_2 q_2 - p_3 q_3, \\
&r=q_3 s_{12} - q_2 s_{13} + p_2 q_3 s_{22} - p_2 q_2 s_{23} + p_3 q_3 s_{23} - p_3 q_2 s_{33} + p_4 q_3 t_2 - p_4 q_2 t_3, \\
&u=q_3^2 s_{22} t_1 - 2 q_2 q_3 s_{23} t_1 + q_2^2 s_{33} t_1 - 2 q_3^2 s_{12} t_2 + 2 q_2 q_3 s_{13} t_2 - p_2 q_3^2 s_{22} t_2 -\\ 
&\qquad 2 p_3 q_3^2 s_{23} t_2 + p_2 q_2^2 s_{33} t_2 + 2 p_3 q_2 q_3 s_{33} t_2 - p_4 q_3^2 t_2^2 + 2 q_2 q_3 s_{12} t_3 - 2 q_2^2 s_{13} t_3 +\\
&\qquad  2 p_2 q_2 q_3 s_{22} t_3 + p_3 q_3^2 s_{22} t_3 - 2 p_2 q_2^2 s_{23} t_3 - p_3 q_2^2 s_{33} t_3 + 
    2 p_4 q_2 q_3 t_2 t_3 - p_4 q_2^2 t_3^2,\\
&s_{11}=-2 p_2 s_{12} - 2 p_3 s_{13} - p_2^2 s_{22} - 2 p_2 p_3 s_{23} - p_3^2 s_{33} - p_4 t_1 - p_2 p_4 t_2 - p_3 p_4 t_3.
\end{align*}
This implies that $\Sigma^{14}_{\mA}$ on this chart is isomorphic to $\mA^{13}\times \mA^{1*}$ with the coordinates of the original $\mA^{18}$ except $q_1,r,u,s_{11}$,
which are written as above by the other coordinates.

\vspace{3pt}

\noindent \underline{\bf $p_2$-chart}: 
Similarly to the $p_1$-chart, we have only to describe $\Sigma^{14}_{\mA}\cap \{p_2=1\}$.
Solving regularly the equations of $\Sigma^{14}_{\mA}$ (\ref{Key1}) and (\ref{Key2}) and setting $p_2=1$, we have the following:
\begin{align*}
&q_2= -p_1 q_1 - p_3 q_3,\\ 
&r= p_3 q_1s_{33} - p_1 q_3 s_{11} - q_3 s_{12} + p_1 q_1 s_{13} - p_3 q_3 s_{13} + q_1 s_{23} - p_4 q_3 t_1 + p_4 q_1 t_3,\\ 
&u= p_1 q_1^2s_{33} t_1 + 2 p_3 q_1 q_3 s_{33}t_1 - p_1 q_3^2 s_{11} t_1 - 2 q_3^2 s_{12} t_1 - 2 p_3 q_3^2 s_{13} t_1 + 2 q_1 q_3 s_{23} t_1- \\
&\qquad  p_4 q_3^2 t_1^2 + q_1^2s_{33} t_2 + q_3^2 s_{11} t_2 - 2 q_1 q_3 s_{13} t_2 - p_3 q_1^2s_{33} t_3 + 2 p_1 q_1 q_3 s_{11} t_3 + \\
&\qquad   p_3 q_3^2 s_{11} t_3 + 2 q_1 q_3 s_{12} t_3 - 2 p_1 q_1^2 s_{13} t_3 -   2 q_1^2 s_{23} t_3 + 2 p_4 q_1 q_3 t_1 t_3 - p_4 q_1^2 t_3^2, \\
& s_{22}= -p_3^2 s_{33}- p_1^2 s_{11} - 2 p_1 s_{12} - 2 p_1 p_3 s_{13} - 2 p_3 s_{23} - p_1 p_4 t_1 - p_4 t_2 - p_3 p_4 t_3.
\end{align*}
This implies that $\Sigma^{14}_{\mA}$ on this chart is isomorphic to $\mA^{13}\times \mA^{1*}$ with the coordinates of the original $\mA^{18}$ except $q_2,r,u,s_{22}$,
which are written as above by the other coordinates.

\vspace{3pt}

The description of the $p_3$-chart is similar to the above two charts, so we only mention that $\Sigma^{14}_{\mA}$ on this chart is isomorphic to $\mA^{13}\times \mA^{1*}$ with the coordinates of the original $\mA^{18}$ except $q_3,r,u,s_{33}$ (the description of the $p_2$-chart is also similar to that of the $p_1$-chart, but we need full descriptions of both of them as above in the section \ref{sec:Example}).

\noindent \underline{\bf $u$-chart}:
Similarly to the $p_1$-chart, we have only to describe $\Sigma^{14}_{\mA}\cap \{u=1\}$.
Solving regularly the equations of $\Sigma^{14}_{\mA}$ (\ref{Key1}) and (\ref{Key2}) and setting $u=1$, we have the following:
\begin{align*}
&p_1= q_3^2 s_{22} t_1 - 2 q_2 q_3 s_{23} t_1 + q_2^2 s_{33} t_1 - q_3 r t_2 - 
    q_3^2 s_{12} t_2 + q_2 q_3 s_{13} t_2 + q_1 q_3 s_{23} t_2 - \\ 
&\ \qquad q_1 q_2 s_{33} t_2 +  q_2 r t_3 + q_2 q_3 s_{12} t_3 - q_2^2 s_{13} t_3 - q_1 q_3 s_{22} t_3 + 
    q_1 q_2 s_{23} t_3,\\ 
& p_2 = q_3 r t_1 - q_3^2 s_{12} t_1 + q_2 q_3 s_{13} t_1 + q_1 q_3 s_{23} t_1 - 
    q_1 q_2 s_{33} t_1 + q_3^2 s_{11} t_2 - 2 q_1 q_3 s_{13} t_2 +\\ 
& \ \qquad q_1^2 s_{33} t_2 - 
  q_1 r t_3 - q_2 q_3 s_{11} t_3 + q_1 q_3 s_{12} t_3 + q_1 q_2 s_{13} t_3 - 
    q_1^2 s_{23} t_3,\\ 
&  p_3= -q_2 r t_1 + q_2 q_3 s_{12} t_1 - q_2^2 s_{13} t_1 - q_1 q_3 s_{22} t_1 + 
    q_1 q_2 s_{23} t_1 + q_1 r t_2 - q_2 q_3 s_{11} t_2 +\\ 
&\ \qquad q_1 q_3 s_{12} t_2 + 
    q_1 q_2 s_{13} t_2 - q_1^2 s_{23} t_2 + q_2^2 s_{11} t_3 - 2 q_1 q_2 s_{12} t_3 + 
    q_1^2 s_{22} t_3,\\ 
&  p_4= -r^2 + q_3^2 s_{12}^2 - 2 q_2 q_3 s_{12} s_{13} + q_2^2 s_{13}^2 - 
    q_3^2 s_{11} s_{22} + 2 q_1 q_3 s_{13} s_{22} + 2 q_2 q_3 s_{11} s_{23} - \\
& \qquad   2 q_1 q_3 s_{12} s_{23} - 2 q_1 q_2 s_{13} s_{23} + q_1^2 s_{23}^2 - q_2^2 s_{11} s_{33} + 
    2 q_1 q_2 s_{12} s_{33} - q_1^2 s_{22} s_{33}.
\end{align*}
This implies that $\Sigma^{14}_{\mA}$ on this chart is isomorphic to $\mA^{13}\times \mA^{1*}$ with the coordinates of the original $\mA^{18}$ except $p_1,p_2,p_3,p_4$,
which are written as above by the other coordinates.

\vspace{3pt}

\noindent \underline{\bf $p_4$-chart}:
Similarly to the $p_1$-chart, it suffices to describe $\Sigma^{14}_{\mA}\cap \{p_4=1\}$. When $p_4=1$, we may  eliminate the coordinate $u$ by the equation $u p_4=F_4$ as in (\ref{Key2}). Then we see that $\Sigma^{14}_{\mA}\cap \{p_4=1\}$  is defined by
the five $4\times 4$ Pfaffians of the skew-symmetric matrix
(\ref{eq:Pf}) with $p_4=1$.
By this description, it holds that
\begin{equation}
\label{eq:S4}
{\sf S}_4:=\Sing \Sigma^{14}_{\mA}\cap \{p_4\not =0\}=
\{p_4\not =0, \bm{p}=\bm{q}=\bm{t}=\bm{0},\,r=u=0\},
\end{equation}
which is $7$-dimensional.

\vspace{3pt}

\noindent \underline{\bf $q_1$-chart}:
Similarly to the $p_1$-chart, it suffices to describe $\Sigma^{14}_{\mA}\cap \{q_1=1\}$. When $q_1=1$, we may eliminate the coordinate $p_1$ by the equation ${\empty^t\bm{p}}\bm{q}=0$ as in (\ref{Key1}).
Then we may verify that $\Sigma^{14}_{\mA}\cap \{q_1=1\}$ is defined by
the five $4\times 4$ Pfaffians of 
the following skew-symmetric matrix:
{{\begin{equation}
\begin{spmatrix}\label{q1Pf}
0 & -u &  r-(s_{23}- q_2 s_{13}-q_3 s_{12}+q_2 q_3 s_{11}) & -s_{33}+2q_3 s_{13}-q_3^2 s_{11} & -t_3+q_3 t_1\\
& 0  & s_{22}-2q_2 s_{12}+q_2^2s_{11} & r+(s_{23}-q_2 s_{13}-q_3 s_{12}+q_2 q_3 s_{11}) & t_2-q_2 t_1\\
& & 0 &  p_4 & -p_3\\
& & & 0 & p_2\\
& & & & 0
\end{spmatrix}
\end{equation}
}}
By this description, it holds that
\begin{align*}
{\sf S}_1&:=\Sing \Sigma^{14}_{\mA}\cap \{q_1\not =0\}=\\
&\{q_1\not =0,\quad \bm{p}=\bm{0},\quad p_4=r=u=0,\\
&q_1^2 s_{23}=q_1 q_2 s_{13}+q_1 q_3 s_{12}-q_2 q_3 s_{11},\\
&q_1^2s_{33}=2q_1q_3 s_{13}-q_3^2 s_{11},\\ 
&q_1^2s_{22}=2q_1q_2 s_{12}-q_2^2s_{11},\\
&q_1t_2=q_2 t_1,\quad q_1t_3=q_3 t_1\},
\end{align*}
which is $7$-dimensional.

\vspace{3pt}

For the $q_2$- and $q_3$-charts, detailed descriptions are similar to that of the $q_1$-chart. We only write down their singular loci below.

\noindent \underline{\bf $q_2$-chart}:
Similarly to the case of the $q_1$-chart, it holds that
\begin{align*}
{\sf S}_2&:=\Sing \Sigma^{14}_{\mA}\cap \{q_2\not =0\}=\\
&\{q_2\not =0,\quad \bm{p}=\bm{0},\quad p_4=r=u=0,\\
&q_2^2 s_{13}=q_1 q_2 s_{23}-q_1 q_3 s_{22}+q_2q_3 s_{12},\\
&q_2^2 s_{33}=2q_2 q_3 s_{23}-q_3^2 s_{22},\\ 
&q_2^2 s_{11}=2q_1q_2 s_{12}-q_1^2s_{22},\\
&q_2 t_1=q_1 t_2,\quad q_2 t_3=q_3 t_2\},
\end{align*}
which is $7$-dimensional.

\noindent \underline{\bf $q_3$-chart}:
Similarly to the case of the $q_1$-chart, it holds that
\begin{align*}
{\sf S}_3&:=\Sing \Sigma^{14}_{\mA}\cap \{q_3\not =0\}=\\
&\{q_3\not =0,\quad \bm{p}=\bm{0},\quad p_4=r=u=0,\\
&q_3^2 s_{12}=q_3 q_2 s_{13}+q_3 q_1 s_{23}-q_1 q_2 s_{33},\\
&q_3^2 s_{11}=2q_1 q_3 s_{13}-q_1^2 s_{33},\\
&q_3^2 s_{22}=2q_2 q_3 s_{23}-q_2^2 s_{33},\\ 
&q_3 t_1=q_1t_3,\quad q_3 t_2=q_2t_3\},
\end{align*}
which is $7$-dimensional.

\vspace{3pt}

Since ${\sf S}_i$ $(i=1,2,3)$ are isomorphic to $\mA^6\times \mA^{1*}$ and we may check that ${\sf S}_i\cap {\sf S}_j$ ($1\leq i<j\leq 3$) is a nonempty open subset both in ${\sf S}_i$ and ${\sf S}_j$, we see that there exists a $7$-dimensional irreducible closed subset $\overline{\sf S}$ of $\SigmaA$ such that $\overline{\sf S}\cap \{q_i\not =0\}={\sf S}_i$ ($i=1,2,3$).

\vspace{3pt}

\noindent \underline{$r$-chart}:
We see that this chart is covered by the other charts. Indeed,
by the fourth equation $up_4=-(r^2+{\empty^t \bm{q}}S^{\dagger} \bm{q})$ of (\ref{Key2}),
we have $u\not =0$ and $p_4\not =0$, or at least one $q_i\not =0$ on this chart.

\vspace{3pt}

As a consequence of the above descriptions of charts, we have the following:

\begin{prop}
\label{prop:SinglociSigma}
The singular locus of $\Sigma^{14}_{\mA}$ is equal to $\mA(p_4,S)\cup \mA(S,\bm{t})\cup \overline{\sf S}$,
where $\dim \mA(p_4,S)=\dim \overline{\sf S}=7$ and $\dim \mA(S,\bm{t})=9$.
\end{prop}

\begin{proof}
Note that $\mA(p_4,S)$ is the closure of ${\sf S}_4$. Therefore, by the above descriptions of charts, we have only to show that $\mA(S,\bm{t})\subset \Sing \Sigma^{14}_{\mA}$. This certainly holds since every equation in (\ref{Key1}) and (\ref{Key2}) is quadratic if we consider the entires of $S$ and $\bm{t}$ are constants.
\end{proof}

\subsection{Unprojection from the Tom format}
\label{sub:Tom14}

\begin{defn}[{\cite[\S 5.5, 5.7]{abr}}]
\label{defn:typeI}
For $a_1,\dots,a_N\in \mN$ such that ${\rm{gcd}}(a_1,\dots,a_N)=1$, we consider the weighted projective space $\mP(a_1,\dots,a_N)$ and we denote by $x_i$ the homogeneous coordinate with weight $a_i$ ($1\leq i\leq N$). Let $V$ be an $n$-dimensional projective variety containing the $x_k$-point for some $k$. We set ${\sf p}:=$ the $x_k$-point. 
\begin{enumerate}[(1)]
\item Assume that there exist positive integers $b_1,\dots,b_n$ with ${\rm{gcd}}(b_1,\dots,b_n)=1$ and ${\rm{gcd}}(b_i,a_k)=1$ $(1\leq i\leq n)$ and global coordinates $x_{i_j}$  ($1\leq j\leq n$) with weight $b_j(=a_{i_j})$ such that $V$ has a $1/a_k (b_1,\dots, b_n)$-singularity at ${\sf p}$ and $x_{i_j}$  ($1\leq j\leq n$) provide the local orbifold coordinates of $V$ at ${\sf p}$. In this situation, we call ${\sf p}$ {\em a Type I center}, and  the restriction 
$V\dashrightarrow \overline{V}$ of the projection $\mP(a_1,\dots,a_N)\dashrightarrow\mP(a_1,\dots, \check{a}_k,\dots,a_N)$ from ${\sf p}$ with the image $\overline{V}$ {\em a Type I projection from ${\sf p}$}.  

It turns out that $\overline{V}$ contains the $\mP(b_1,\dots,b_n)$ which is the common zero of $N-n-1$ coordinates of $\mP(a_1,\dots, \check{a}_k,\dots,a_N)$ and which is the image of the exceptional divisor of the weighted blow-up of $V$ at ${\sf p}$ with weight $1/a_k\, (b_1,\dots, b_n)$ for the above local orbifold coordinates. We set $\Pi:=\mP(b_1,\dots,b_n)$.
\item
In the situation of (1), we also assume that $\overline{V}$ is of codimension three in $\mP(a_1,\dots, \check{a}_k,\dots,a_N)$, and that,
for simplicity of notation, $\Pi$ is equal to $\{y_1=y_2=y_3=y_4=0\}$ and $k\geq 5$, where we denote by $y_l$ the coordinates of  
$\mP(a_1,\dots, \check{a}_k,\dots,a_N)$ with weight $a_l$. Assume moreover that
$\overline{V}$ is the common zero of five $4\times 4$ Pfafffians of a $5\times 5$ skew-symmetric matrix
$\begin{spmatrix} 0 & m_{12} & m_{13} & m_{14} & m_{15}\\
& 0 & m_{23} & m_{24} & m_{25}\\
& & 0 & m_{34} & m_{35}\\
& & & 0 & m_{45}\\
& & & & 0 \end{spmatrix}$
such that all the $m_{ij}$ are homogeneous with respect to the above weights and $m_{ij}\, (2\leq i<j \leq 5)$ are contained in the ideal $(y_1,y_2,y_3,y_4)$.  
Then we call the Type I projection $V\dashrightarrow \overline{V}$ from ${\sf p}$ {\em a Type I Tom projection}.
\end{enumerate}
\end{defn}

\begin{rem}
A Type I Jerry projection is also defined but we do not quote this here.
\end{rem}

In this subsection, we use freely
what we have already seen in the subsection \ref{MainC}.

\begin{lem}
\label{lem:Tom}
Let $S_{\rm{Tom}}$ be the polynomial ring over $\mC$ with variables $\bm{p},p_4,\bm{q},r,S,\bm{t}$.
Let $I_{\rm{Tom}}$ be the ideal of
$S_{\rm{Tom}}$ generated by
the five $4\times 4$ Pfaffians of the skew-symmetric matrix
$(\ref{eq:Pf})$.
The following assertions hold\,$:$
\begin{enumerate}[$(1)$]
\item
The variables $p_1,p_2,p_3,p_4$ form a regular sequence in $S_{\rm{Tom}}$.
\item
$I_{\rm{Tom}}$ is a Gorenstein ideal of codimension three.
\item
$I_{\rm{Tom}}$ is a prime ideal.
\end{enumerate}
\end{lem}

\begin{proof}
The first assertion is obvious. 

We show the second assertion
following \cite[\S 3.2.4--3.2.6]{PHD} closely.
In the affine $10$-space $\mA^{10}$ with coordinates 
$\xi_{ij}\, (1\leq i<j\leq 5)$,
let $G$ be the closed subscheme defined by
the five $4\times 4$ Pfaffians
of the skew-symmetric matrix
$\begin{spmatrix}
0 & \xi_{12} & \xi_{13} & \xi_{14} & \xi_{15} \\ 
& 0 & \xi_{23} & \xi_{24} & \xi_{25} \\
& & 0 & \xi_{34} & \xi_{35} \\
& & & 0 & \xi_{45} \\
&&&&0
\end{spmatrix}$.
Note that $G$ is nothing but 
the affine cone over the Grassmannian $G(2,5)$, thus is irreducible and reduced.
In the affine space $\mA(\bm{p},p_4,\bm{q},r,S,\bm{t})$, let $\Sigma_{\rm{Tom}}$ be the closed subscheme defined by the ideal $I_{\rm{Tom}}$.
We define a morphism $p\colon \Sigma_{\rm{Tom}}\to G$ by setting
$\begin{spmatrix}
0 & \xi_{12} & \xi_{13} & \xi_{14} & \xi_{15} \\ 
& 0 & \xi_{23} & \xi_{24} & \xi_{25} \\
& & 0 & \xi_{34} & \xi_{35} \\
& & & 0 & \xi_{45} \\
&&&&0
\end{spmatrix}=
\begin{spmatrix}
0 & q_3 & -q_2 & q_1 & r\\
  & 0 & p_1 & p_2 & a_{25}\\
  &   & 0 & p_3 & a_{35}\\
  &   &   & 0         & a_{45}\\
  &   &   &           & 0 
\end{spmatrix}$,
where we recall
$\begin{spmatrix}
a_{25}\\
a_{35}\\
a_{45}
\end{spmatrix}
=\begin{spmatrix}
s_{13} & s_{23} & s_{33}\\
-s_{12} & -s_{22} & -s_{23}\\
s_{11} & s_{12} & s_{13}\\
\end{spmatrix}
\begin{spmatrix}
p_1\\
p_2\\
p_3
\end{spmatrix}
+{\scriptsize p_4}
\begin{spmatrix}
t_3\\
-t_2\\
t_1
\end{spmatrix}$.
Then it is easy to see
that
$\Sigma_{\rm{Tom}}\setminus \{p_4=0\}\to G$ is
a $(\mA^6 \times \mA^*)$-bundle. Therefore
$\Sigma_{\rm{Tom}}\setminus \{p_4=0\}$ is a $14$-dimensional variety.
Moreover, investigating $\Sigma_{\rm{Tom}}\cap \{p_4=0\}\to G$,
it is easy to see that $\Sigma_{\rm{Tom}}\cap \{p_4=0\}$ is $13$-dimensional.
Therefore $\Sigma_{\rm{Tom}}$ is $14$-dimensional, hence
is of codimension three in $\mA^{17}$, and then
$I_{\rm{Tom}}$ is Gorenstein and of codimension three by \cite{BE}.

We prove the third assertion showing
that $R_{\rm{Tom}}:=S_{\rm{Tom}}/I_{\rm{Tom}}$ is a domain. 
In the proof of the second assertion,
we have seen that 
$\Sigma_{\rm{Tom}}\setminus \{p_4=0\}$ is a $14$-dimensional variety, and
$\Sigma_{\rm{Tom}}\cap \{p_4=0\}$ is $13$-dimensional.
Therefore, since $I_{\rm{Tom}}$ is Gorenstein,
$\Sigma_{\rm{Tom}}$ is irreducible and reduced by unmixedness.
Consequently,
$R_{\rm{Tom}}$ is a domain. 
\end{proof}

\begin{prop}
\label{prop:Tom14}
Let $S_{\Sigma}$ be the polynomial ring over $\mC$ with variables $\bm{p},p_4,\bm{q},r,u,S,\bm{t}$.
Let $I_{\Sigma}$ be the ideal of the polynomial ring $S_{\Sigma}$ generated by the nine polynomials appearing in the equations $(\ref{Key1})$ and $(\ref{Key2})$.
Set $R_{\Sigma}:=S_{\Sigma}/I_{\Sigma}$.
The following assertions hold\,$:$
\begin{enumerate}[$(1)$]
\item
$I_{\Sigma}$ is a Gorenstein ideal of codimension four.
\item
$I_{\Sigma}$ is a prime ideal, thus
$\Sigma^{14}_{\mA}$ is irreducible and reduced.
\item
$\Sigma^{14}_{\mA}$ and $\Sigma^{14}_{\mA}\cap \{p_1=0\}$ are normal.
\item
$R_{\Sigma}$ is a UFD.
\end{enumerate}
\end{prop}

\begin{proof}
(1) We show the first assertion following closely the proof of \cite[Thm.~5.6]{WithC}.
Since $S_{\Sigma}$ is a polynomial ring over $\mC$ and $I_{\Sigma}$ is positively graded 
with suitable weights of variables (for example, the weights of No.24078 as in Table \ref{Table 1}),
we have only to show
the localization $I_{\Sigma,o}$ of
$I_{\Sigma}$ by the irrelevant maximal ideal of $S_{\Sigma}$ is a Gorenstein ideal of codimension four (see for example \cite[Exercises 3.6.20]{BH}).
Abusing notation, we denote by the same letter the objects obtaining by 
this localization of some objects.
We denote by $J_{\rm{Tom}}$ the ideal of $S_{\rm{Tom}}$ generated by the regular sequence $p_1,p_2,p_3,p_4$.
We define a homomorphism $\varphi\colon J_{\rm{Tom}}/I_{\rm{Tom}}\to S_{\rm{Tom}}/I_{\rm{Tom}}$  by $p_i\mapsto (F_i\ \text{mod} \ I_{\rm{Tom}})$ $(i=1,\dots,4)$, where $F_i$ is defined as in (\ref{F1F4}).
We intend to apply \cite[Thm.~1.5]{WithoutC} to our setting with this homomorphism $\varphi$. For this, we need to check the two conditions for $\varphi$ in \cite[Lem.~1.1]{WithoutC}.
Note that $F_4$ does not vanish
at the $r$-point of $\Sigma_{\rm{Tom}}$.
Therefore $\varphi$ is a nonzero map. Since $I_{\rm{Tom}}$ is prime by
Lemma \ref{lem:Tom} (3),
we see that $\varphi$ is injective, namely,  
the condition (i) of \cite[Lem.~1.1]{WithoutC} 
holds for $\varphi$.
We check the condition (ii) of \cite[Lem.~1.1]{WithoutC} 
holds for $\varphi$.
As we have observed in the subsection \ref{MainC},
the polynomials $F_1,\dots,F_4$ in (\ref{F1F4})
correspond to $g_1,\dots,g_4$ as in \cite[\S 5.3]{WithC}.
 Therefore the assertion corresponding to \cite[Thm.~5.5]{WithC} holds in our case by specialization from the generic Tom ideal.
Then, using the polynomials $F_1,\dots,F_4$, we obtain a map of complexes
from the Pfaffian resolution of $I_{\rm{Tom}}$ to
the Koszul resolution of $J_{\rm{Tom}}$.
It is straightforward to check that
the element 
$\varphi \in \Hom (J_{\rm{Tom}}/I_{\rm{Tom}}, S_{\rm{Tom}}/I_{\rm{Tom}})$
is $k_1$ in \cite[Thm.~3.2]{WithC} obtained from this map of complexes
(remember that $\omega_{S_{\rm{Tom}}/I_{\rm{Tom}}}\simeq S_{\rm{Tom}}/I_{\rm{Tom}}$).
Therefore, by \cite[Thm.~3.2]{WithC},
the condition (ii) of \cite[Lem.~1.1]{WithoutC} 
holds for $\varphi$.
Consequently, the first assertion follows from 
\cite[Thm.~1.5]{WithoutC}.
\vspace{3pt}

\noindent (2) To see the second assertion, we show that $R_{\Sigma}:=S_{\Sigma}/I_{\Sigma}$ is a domain.
By \cite[Thm.~1.5]{WithoutC},
$u$ is not a zero divisor after the localization of the irrelevant 
maximal ideal. Noting $R_{\Sigma}$ is positively graded with 
suitable weights of variables (for example, the weights of No.24078 as in Table \ref{Table 1}) , we see that
$u$ is not a zero divisor in $R_{\Sigma}$ by a standard argument of commutative algebra.
By the study of the $u$-chart as in the subsection \ref{Charts},
we see that 
$(R_{\Sigma})_u$ is a localization of
a polynomial ring, thus is a domain.
Then $R_{\Sigma}$ is a domain since $u$ is not a zero divisor.
\vspace{3pt}

\noindent (3)
We only show that $\Sigma^{14}_{\mA}\cap \{p_1=0\}$ is normal since we can show the normality of $\Sigma^{14}_{\mA}$ vabatim.

Since $R_{\Sigma}$ is a domain, $p_1$ is not a zero divisor. 
Therefore $\Sigma^{14}_{\mA}\cap \{p_1=0\}$ is Gorenstein since so is $\Sigma^{14}_{\mA}$. In particular, the affine coordinate ring of $\Sigma^{14}_{\mA}\cap \{p_1=0\}$ satisfies the $S_2$ condition.

It remains to show that $\Sing (\Sigma^{14}_{\mA}\cap \{p_1=0\})$ has codimension $\geq 2$ in $\Sigma^{14}_{\mA}\cap \{p_1=0\}$.
This can be shown separately on the $p_2$- and $p_3$-charts, and the locus $\{p_1=p_2=p_3=0\}$.
By the coordinate descriptions of the $p_2$- and $p_3$-charts as in the subsection \ref{Charts}, we immediately check that
$\Sigma^{14}_{\mA}\cap \{p_1=0\}$ is smooth there. Now it suffices to show that $\Sigma^{14}_{\mA}\cap \{p_1=p_2=p_3=0\}$ is $11$-dimensional. By the equations in (\ref{Key1}) and (\ref{Key2}), we have $\Sigma^{14}_{\mA}\cap \{p_1=p_2=p_3=0\}$ is the union of the following two closed subschemes (\ref{p1=0a}) and (\ref{p1=0b}):
\begin{align}
&\{\bm{p}=\bm{0}, p_4=0, F_1=F_2=F_3=F_4=0\}.\label{p1=0a}\\
&\left\{\bm{p}=\bm{0}, u=0, \rank \begin{spmatrix} q_1 & q_2 & q_3 \\ t_1 & t_2 & t_3 \end{spmatrix}\leq 1, F_4=0\right\}.\label{p1=0b}
\end{align}
As for (\ref{p1=0b}), it is the intersection between the affine cone over $\mP^2\times \mP^1$ with the vertex $\mA^8(p_4,r,S)$ and the hypersurface $\{F_4=0\}$. Therefore it is $11$-dimensional as desired.
As for (\ref{p1=0a}), we consider the projection to $\mA^{12}(\bm{q},S,\bm{t})$ (the elimination of $r$).
Then we see that the image is contained in the hypersurface $\{{\empty^t \bm{v}}S\bm{v}=0\}$, where ${\empty^t \bm{v}}:=\begin{spmatrix}
\begin{svmatrix} q_2 & q_3\\ t_2 & t_3 \end{svmatrix}&
-\begin{svmatrix} q_1 & q_3\\ t_1 & t_3 \end{svmatrix}&
\begin{svmatrix} q_1 & q_2\\ t_1 & t_2 \end{svmatrix}
\end{spmatrix}$.  
Moreover, by the equation $F_4=0$, the inverse image of any point by the projection consists of at most two points.
Therefore (\ref{p1=0a}) is at most $11$-dimensional as desired.

\vspace{3pt}

\noindent (4)  
Note that we have already shown that $R_{\Sigma}$ is a domain in (2).
Therefore, by Nagata's theorem \cite[Thm.~20.2]{Matsumura},
it suffices to show that 
the ring $(R_{\Sigma})_{p_1}$ is a UFD and
$p_1$ is a prime element of $R_{\Sigma}$.

By the description of the $p_1$-chart as in the subsection \ref{Charts},
we see that 
$(R_{\Sigma})_{p_1}$ is a localization of
a polynomial ring, thus is a UFD.

Note that the polynomial ring $S_{\Sigma}$ is positively graded with some weights of variables such that $I_{\Sigma}$ is homogeneous
(for example, the weights of No.24078 as in Table \ref{Table 1}).
Then the ring $R_{\Sigma}$ is also graded with this weights. In this situation, $p_1$ defines an ample divisor on the corresponding weighted projectivization $\Sigma^{13}_{\mP}$ of $\Sigma^{14}_{\mA}$. Note that
$\Sigma^{13}_{\mP}$ is irreducible and normal since so is $\Sigma^{14}_{\mA}$ and $\Sigma^{13}_{\mP}$ is its geometric $\mC^*$-quotient.
By a general property of an ample divisor on an irreducible normal variety, the divisor $\Sigma^{13}_{\mP}\cap \{p_1=0\}$ is connected and  hence 
so is $\Sigma^{14}_{\mA}\cap \{p_1=0\}$. Since $\Sigma^{14}_{\mA}\cap \{p_1=0\}$ is normal by (3), this is irreducible.
Therefore $p_1$ is a prime element of $R_{\Sigma}$.
\end{proof}

Though the following two corollaries can be stated in a more general setting, 
we restrict them to weighted projective settings since they are easy to apply in the section \ref{sec:Example} for the constructions of $\mQ$-Fano threefolds as in Theorem \ref{thm:main} (1).

\begin{cor}
\label{SigmaDiv}
Let $\Sigma^{13}_{\mP}$ be the weighted projectivization of $\Sigma^{14}_{\mA}$ with some positive weights of coordinates.
Any prime Weil divisor on $\Sigma^{13}_{\mP}$
is the intersection between $\Sigma^{13}_{\mP}$ and a weighted hypersurface.
In particular, $\Sigma^{13}_{\mP}$ is $\mQ$-factorial and has Picard number one.
\end{cor}

\begin{proof}
Let $\frak{p}\subset R_{\Sigma}$ be the homogeneous ideal of a prime divisor $D$. Then $\frak{p}$ is an ideal of 
height one.
Since $R_{\Sigma}$ is a UFD by Proposition \ref{prop:Tom14} (3),
its ideal of height one is principal by \cite[Thm.~20.1]{Matsumura}.
Thus the assertion follows.
\end{proof}

\begin{cor}
\label{QFanoDiv}
Let $\Sigma^{13}_{\mP}$ be the weighted projectivization of $\Sigma^{14}_{\mA}$ with some positive weights of coordinates.
Let $X$ be a quasi-smooth threefold
such that $X$ is a codimension $10$ weighted complete intersection
in $\Sigma^{13}_{\mP}$, i.e., there exist ten weighted homogeneous
polynomials 
$F_1,\dots, F_{10}$ such that $X=\Sigma^{13}_{\mP}\cap \{F_1=0\}\cap\dots\cap\{F_{10}=0\}$. Assume moreover that $\{p_1=0\}\cap X$ is a prime divisor. 
Then any prime Weil divisor on $X$
is the intersection between $X$ and a weighted hypersurface. In particular, $X$ is $\mQ$-factorial and has Picard number one.
\end{cor}

\begin{proof}
Once we check that 
$R_X:=R_{\Sigma}/\left(F_1,\dots,F_{10}\right)$
is a UFD, the proof of Corollary \ref{SigmaDiv} works verbatim here.
Let $R_{X,o}$ be the localization of $R_X$ by the maximal
irrelevant ideal.  
Note that, by \cite[Prop.~7.4]{UFD},
$R_X$ is a UFD if and only if 
so is $R_{X,o}$. Thus we have only to check the latter.
To check $R_{X,o}$ is a UFD, we apply Nagata's theorem as in the proof of
Proposition \ref{prop:Tom14} (3). For this, we need to check
that $p_1$ is a prime element of $R_{X,o}$, which follows from our assumption, and
$(R_{X,o})_{p_1}$ is a UFD, which we will check below.
We denote by $R_{\Sigma,o}$
the localization of $R_{\Sigma}$ by the maximal
irrelevant ideal. 
Since $\left(R_{\Sigma,o}\right)_{p_1}$ is a localization of a polynomial ring by the description of the $p_1$-chart as in the subsection \ref{Charts},
we see that $(R_{X,o})_{p_1}$ is a localization of a complete intersection ring. 
Therefore $(R_{X,o})_{p_1}$ is a complete intersection local ring of dimension four. Since $X$ is quasi-smooth, we see that $(R_{X,o})_{p_1}$ is a UFD by 
\cite[Cor.~3.10 and Thm.~3.13]{sga2}. 
\end{proof}

\section{\bf Affine variety $\Sigma_{\mA}^{13}$}
\label{13}

Recall that we define $\SigmaA:=\Sigma_{\mA}^{14}\cap\{s_{33}=1\}$.

\begin{lem}
\label{lem:1314}
There exists 
an isomorphism from $\Sigma^{14}_{\mA}\cap \{s_{33}\not =0\}$ to
$\Sigma^{13}_{\mA}\times (\mA^1)^{*}$.
Moreover,
this induces an isomorphism from $($the $p_1$-chart of $\Sigma^{14}_{\mA})\cap \{s_{33}\not =0\}$ to
$($the $p_1$-chart of $\Sigma^{13}_{\mA})\times(\mA^1)^{*}$.
\end{lem}

\begin{proof}
We can define a morphism $\Sigma^{14}_{\mA}\cap \{s_{33}\not =0\}\to \Sigma^{13}_{\mA}\times (\mA^1)^{*}$ by
\begin{align*}
&\Sigma^{14}_{\mA}\cap \{s_{33}\not =0\}
\ni
(\bm{p},p_4,\bm{q}, r,u, S,\bm{t})
\mapsto\\
&\quad\left((s_{33}^{-1}\bm{p},s_{33}^{-1}p_4,\bm{q}, s_{33}^{-1}r,s_{33}^{-1}u,s_{33}^{-1}S,s_{33}^{-1}\bm{t}),s_{33}\right)\in \Sigma^{13}_{\mA}\times (\mA^1)^{*}.
\end{align*}
We can show easily that this is an isomorphism and the latter assertion holds. 
\end{proof}

We set ${\sf S}'_i:={\sf S}_i\cap \Sigma^{13}_{\mA}$ $(i=1,\dots,4)$ and $\overline{\sf S}':=\overline{\sf S}\cap \Sigma^{13}_{\mA}$, where we recall that ${\sf S}_i$ and $\overline{\sf S}$ are defined as in the subsection \ref{Charts}. We easily see that ${\sf S}'_1\cup {\sf S}'_2\subset {\sf S}'_3$ from the definition of ${\sf S}_i$. Therefore  $\overline{\sf S}'$ is the closure of ${\sf S}'_3$.
We also denote by $S'$ the matrix obtained from $S$ by setting $s_{33}=1$, and let $\mA(S',\bm{t}):=\mA(S,\bm{t})_{|s_{33}=1}$ and
$\mA(p_4,S'):=\mA(p_4,S)_{|s_{33}=1}$.

\vspace{3pt}

By Proposition \ref{prop:SinglociSigma} and Lemma \ref{lem:1314}, we have the following:
\begin{prop}
\label{prop:Sing13}
The singular locus of $\Sigma^{13}_{\mA}$ is equal to $\mA(p_4,S')\cup \mA(S',\bm{t})\cup \overline{\sf S}'$,
where $\dim \mA(p_4,S')=\dim \overline{\sf S}'=6$ and $\dim \mA(S',\bm{t})=8$.
\end{prop}

Moreover,
by Lemma \ref{lem:1314} and Proposition \ref{prop:Tom14}, 
we obtain the following in the same way as we obtain Corollaries \ref{SigmaDiv} and \ref{QFanoDiv}:

\begin{prop}
\label{SigmaQFanoDiv13}
Let $\Sigma^{12}_{\mP}$ be the weighted projectivization of $\Sigma^{13}_{\mA}$ with some positive weights of coordinates.
It holds that
\begin{enumerate}[(1)]
\item any prime Weil divisor on $\Sigma^{12}_{\mP}$
is the intersection between $\Sigma^{12}_{\mP}$ and a weighted hypersurface.
In particular, $\Sigma^{12}_{\mP}$ is $\mQ$-factorial and has Picard number one, and
\item
let $X$ be a quasi-smooth threefold
such that $X$ is a codimension $9$ weighted complete intersection
in $\Sigma^{12}_{\mP}$.
Assume moreover that $\{p_1=0\}\cap X$ is a prime divisor. 
Then any prime Weil divisor on $X$
is the intersection between $X$ and a weighted hypersurface. In particular, $X$ is $\mQ$-factorial and has Picard number one.
\end{enumerate}
\end{prop}

\section{\bf Searching weights}
\label{sec:weight}
\subsection{Weights for variables and equations}
\label{sub:Weights}
We assign weights for variables of the polynomial ring $S_{\Sigma}$ such that all the equations as in (\ref{Key1}) and (\ref{Key2}) are homogeneous.
Moreover, we assume that all the variables are not zero allowing some of them are constants (note that we have $s_{33}=1$ for the weighted projectivization $\SigmaP$ of $\Sigma^{13}_{\mA}$ with some weights of coordinates).
Then it is easy to derive the following relations between the weights of variables:

\vspace{3pt}

\noindent {\bf Relations between the weights of variables:}

The first equation of (\ref{Key1}) implies
\begin{equation}
w(p_i)+w(q_i) \ \text{are independent of $i$},
\end{equation}
which we denote by $d_0$.

The second matrix equation of (\ref{Key1}) implies
\begin{equation}
\label{eq:wsij}
\begin{cases}
w(s_{12})=w(r)-w(q_3),\\
w(s_{13})=w(r)-w(q_2),\\
w(s_{23})=w(r)-w(q_1),
\end{cases}
\end{equation}
and
\begin{equation}
\label{eq:wsii}
\begin{cases}
w(s_{11})=w(r)+w(p_2)-w(p_1)-w(q_3),\\
w(s_{22})=w(r)+w(p_3)-w(p_2)-w(q_1),\\
w(s_{33})=w(r)+w(p_1)-w(p_3)-w(q_2).
\end{cases}
\end{equation}

The fourth entry of (\ref{Key2}) implies 
\begin{equation}
\label{w(p4)}
w(p_4)=2w(r)-w(u).
\end{equation}

The second matrix equation of (\ref{Key1}) and the equality (\ref{w(p4)}) imply
\begin{equation}
\label{eq:wti}
\begin{cases}
w(t_1)=w(u)+w(p_2)-w(r)-w(q_3),\\
w(t_2)=w(u)+w(p_3)-w(r)-w(q_1),\\
w(t_3)=w(u)+w(p_1)-w(r)-w(q_2).
\end{cases}
\end{equation}

We may check that the above equalities of the weights of the variables generate all the relations of the weights of the variables. 

By symmetry, we may also assume that 

\begin{equation}
w(p_1)\leq w(p_2)\leq w(p_3).
\end{equation}

We list up in order the weights of the equations as in (\ref{Key1}) and (\ref{Key2}) :

\vspace{3pt}

\noindent {\bf Weights of equations:}
\begin{align}
&d_0:=w(p_1)+w(q_1) (=w(p_2)+w(q_2)=w(p_3)+w(q_3)),\label{d0}\\
&d_1:=w(r)+w(p_1),\, d_2:=w(r)+w(p_2),\, d_3:=w(r)+w(p_3),\label{d1}\\
&d_4:=-d_0+w(p_1p_2p_3)+w(r),\label{d4}\\
&d_5:=w(u)+w(p_1),\, d_6:=w(u)+w(p_2),\, d_7:=w(u)+w(p_3),\label{d5}\\
&d_8:=2w(r).\label{d8}
\end{align}

We also set
\begin{equation}
\label{eq:delta0}
\delta:=2w(r)+w(u)+w(p_1 p_2 p_3).
\end{equation}
Then we observe that the summation of all the weights of the equations is equal to 
\begin{equation}
\label{eq:delta}
\sum_{i=0}^{8} {d_i}=3\delta.
\end{equation}

\subsection{Possibilities of weights}
\label{sub:posweight}

Hereafter we fix a number $*$ as in the database \cite{grdb} for a prime $\mQ$-Fano threefold of codimension four. Then, in many cases, we can read off the weights of the nine equations for a prime $\mQ$-Fano threefold belonging to No.$*$. For example, we can do for No.360. If we cannot, then we can find a possibility of the weights of the nine equations in the following way:
Recall that for each No.$*$ as in Theorem \ref{thm:main} except No.360, a prime $\mQ$-Fano threefold belonging to this class is expected to have a Type I projection by \cite{grdb}. Then as a data for No.$*$, the type of the center of this Type I projection is given. We denote by $1/v (1,a,v-a)$ the type of the center of the projection.
Assume that the prime Fano threefold $X_{\rm{Tom}}$ of codimension three which is the target of the projection belongs to one of $69$ families as in \cite{Al} (see \cite{grdb} for cascades of projections). Then the set of the weights for the five equations of $X$ is given as a data for No.$*$.
By the definition of Type I projection (Definition \ref{defn:typeI}), the set of the weights of the coordinates of the ambient $7$-dimensional weighted projective space $\mP_{X_{\rm{Tom}}}$ for $X_{\rm{Tom}}$ contains the three values $1,a,v-a$. Let $w_1,\dots,w_4$ be the remaining four weights of the coordinates of $\mP_{X_{\rm{Tom}}}$. Then $v+w_1,\dots, v+w_4$ coincides with the weights for the four remaining equations of $X$.

Therefore, in any case, we are supposed to be given the set $\{d_0, d_1,\dots,d_8\}$ of the weights of the nine equations as in (\ref{Key1}) and (\ref{Key2}). Then we have the following program to obtain possibilities of weights of the variables of $S_{\Sigma}$ based on the subsection \ref{sub:Weights}. Note that we do not exclude possibilities of negative weights of coordinates in the program.

\vspace{3pt}

\noindent{\bf Program to restrict possibilities of weights}
\begin{itemize}
\item
Input the set ${\sf S}:=\{d_0, d_1,\dots,d_8\}$, which is given as a data for No.$*$.
\item
Set $\delta:=1/3 \sum_{i=0}^{8} d_i$ (cf. (\ref{eq:delta})).
\item
Choose an even element of ${\sf S}$ and set it as $d_8$. Set $w(r):=d_8/2$ (cf. (\ref{d8})).
\item
Choose an element of ${\sf S}$ except $d_8$ and set it as $d_0$. 
\item
Moreover, choose an element of ${\sf S}$ except $d_0,d_8$ and set it as $d_4$.
\item
Set $w(u):=\delta-w(r)-d_0-d_4$ (cf. (\ref{d4}) and (\ref{eq:delta0})).
\item
If the set ${\sf S}\setminus \{d_0,d_4,d_8\}$ can be decompose as $\{d_1,d_2,d_3\}\sqcup \{d_5,d_6,d_7\}$
such that $d_1\leq d_2\leq d_3$, $d_5\leq d_6\leq d_7$, and $d_1-d_5=d_2-d_6=d_3-d_7=w(r)-w(u)$ (cf. (\ref{d1}) and (\ref{d5})),
then output the following data:

$w(r)$, $w(u)$, $\delta$, $\begin{spmatrix} w(p_1) \\ w(p_2) \\ w(p_3) \end{spmatrix}$ defined by (\ref{d1}),
$\begin{spmatrix} w(q_1) \\ w(q_2) \\ w(q_3) \end{spmatrix}$ defined by (\ref{d0}),
$\begin{spmatrix} w(s_{11}) & w(s_{12}) & w(s_{13}) \\
 & w(s_{22}) & w(s_{23}) \\
 & & w(s_{33}) \end{spmatrix}$ defined by (\ref{eq:wsij}) and (\ref{eq:wsii}), and
$\begin{spmatrix} w(t_1) \\ w(t_2) \\ w(t_3) \end{spmatrix}$ defined by (\ref{eq:wti}).
\end{itemize}

Among several solutions, we choose those as in Table \ref{Table 1}.
\begin{rem}
We have also searched weights of variables for the classes of prime $\mQ$-Fano threefolds of codimension four in \cite{grdb} other than those as in Theorem \ref{thm:main}. Although we do not claim we are exhaustive, there are likely no other appropriate weights except as in Table \ref{Table 1}.
\end{rem}

\subsection{Graded $9\times 16$ free resolution}\label{sub:916}
It is easy to obtain the following proposition from the equations in (\ref{Key1}) and (\ref{Key2}): 
\begin{prop}
\label{prop:916}
The following assertions hold:
\begin{enumerate}[(1)]
\item The affine coordinate ring $R_{\Sigma}$ of $\Sigma_{\mA}^{14}$ has the following so-called $9\times 16$ $S_{\Sigma}$-free resolution which is graded with respect to the weights given as in the subsection \ref{sub:Weights}:
\begin{equation}
\label{eq:freeresol}
0\leftarrow R_{\Sigma} \leftarrow P_0 \leftarrow P_1\leftarrow P_2\leftarrow P_3\leftarrow P_4\leftarrow 0,
\end{equation}
where
\begin{align*}
P_0=&S_{\Sigma},\, P_1=\oplus_{i=0}^{8} S_{\Sigma}(-d_i),\\
P_2=&S_{\Sigma}(-(w(r)+d_0))\oplus S_{\Sigma}(-(w(u)+d_0))\oplus\\ 
&S_{\Sigma}(-w(rup_1))\oplus S_{\Sigma}(-w(rup_2))\oplus S_{\Sigma}(-w(rup_3))\oplus\\
&S_{\Sigma}(-w(r^2p_1))\oplus S_{\Sigma}(-w(r^2p_2))\oplus S_{\Sigma}(-w(r^2p_3))\oplus\\
&S_{\Sigma}(-\delta+w(r^2p_1))\oplus S_{\Sigma}(-\delta+w(r^2p_2))\oplus S_{\Sigma}(-\delta+w(r^2p_3))\oplus\\
&S_{\Sigma}(-\delta+w(rup_1))\oplus S_{\Sigma}(-\delta+w(rup_2))\oplus S_{\Sigma}(-\delta+w(rup_3))\oplus\\&S_{\Sigma}(-\delta+w(r)+d_0)\oplus S_{\Sigma}(-\delta+w(u)+d_0),\\
P_3=&\oplus_{i=0}^{8} S_{\Sigma}(d_i-\delta),\,P_4=S_{\Sigma}(-\delta).
\end{align*}

Moreover we have
\begin{equation}
\label{eq:KSigma}
\text{$K_{\Sigma_{\mP}^{13}}=\sO(-k)$ with $k=3\delta-2w(r)-6d_0-w(u)$.}
\end{equation}
\item
The same statements hold as in (1) for $\Sigma^{12}_{\mP}$ and for the ideal of $\Sigma_{\mA}^{13}$ in the polynomial ring obtained from $S_{\Sigma}$ by setting $s_{33}=1$ (hence $w(s_{33})=1$).
\end{enumerate}

\end{prop}

Now we prove a part of Theorem \ref{thm:main}.
Note that we can read off the Hilbert numerator of a weighted complete intersection of $\Sigma^{12}_{\mP}$ or $\Sigma^{13}_{\mP}$
from the free resolution (\ref{eq:freeresol}) (cf. \cite[Lem.4.1.13]{BH}, \cite[Rem.~3.6]{kino}). 
By (\ref{eq:KSigma}) as in Proposition \ref{prop:916}, we can immediately compute the canonical divisor of a weighted complete intersection of $\Sigma^{12}_{\mP}$ or $\Sigma^{13}_{\mP}$ respectively. Therefore we obtain the following assertion:

\begin{cor}
\label{cor:candivdeg}
For any class in Table \ref{Table 1}, we choose the weights of coordinates as in there
and a weighted complete intersection $X$ of $\Sigma^{12}_{\mP}$ or $\Sigma^{13}_{\mP}$ as in Table \ref{Table 2}. Then $K_X$ is equal to $\sO_X(-1)$, and the genus of $X$ and  the Hilbert numerator of $X$ with respect to $-K_X$ coincide with those given in \cite{grdb}.
\end{cor}
Moreover, we have the following:
\begin{cor}
\label{cor:ac4}
Let $X$ be as in Corollary \ref{cor:candivdeg}. Assume that $X$ is quasi-smooth, has Picard number one, and has cyclic quotient singularities assigned for the class of $X$. Then ${\rm{ac}}_X=4$.
\end{cor}

\begin{proof}
Note that $\Sigma^{12}_{\mP}$ or $\Sigma^{13}_{\mP}$ corresponding to $X$ is of codimension $4$ in the ambient weighted projective space. Therefore so is $X$ in the ambient weighted projective space obtained from the ambient weighted projective space of the corresponding $\Sigma^{12}_{\mP}$ or $\Sigma^{13}_{\mP}$. By Corollary \ref{cor:candivdeg} and the assumption of this corollary, $X$ is a prime $\mQ$-Fano threefold. By the classification of prime $\mQ$-Fano threefolds with anti-canonical codimension $\leq 3$, the numerical data of $X$ does not coincide with that of any prime $\mQ$-Fano threefold with anti-canonical codimension $\leq 3$. This implies that ${\rm ac}_X=4$ as desired.  
\end{proof}

\begin{rem}
In Corollary \ref{cor:candivdeg}, the assumption that the Picard number of $X$ is one is crutial. In \cite[Thm.~1.2]{GKQ}, examples of two $\mQ$-Fano threefolds with the same numerical data and different codimensions--one has codimension 3 and another has codimension 4--are presented. The $\mQ$-Fano threefold of codimension 3 has Picard number one while that of codimension 4 has Picard number two.
\end{rem}

Let $X$ be as in Corollary \ref{cor:candivdeg}.
To complete the proof of Theorem \ref{thm:main} for $X$, it remains to show the following:
\begin{enumerate}
\item
$X$ is quasi-smooth and has only cyclic quotient singularities assigned for the class of $X$, and the Picard number of such an $X$ is one.
\item
$X$ has a projection described as in Theorem \ref{thm:main} (1).
\end{enumerate}

We will show (1) and Theorem \ref{thm:main} (2) in the section \ref{sec:Example} and will show (2) in the section \ref{sec:proj}.

\begin{rem}
It is a bit surprise that there are only two cases for $w(S)$;
\[
w(S)=\begin{spmatrix} 4 & 3 & 2\\ & 2 & 1 \\& &  0 \end{spmatrix}\ \text{or}\ \begin{spmatrix} 2 & 2 & 1\\ & 2 & 1 \\& &  0 \end{spmatrix}.
\]
In the former case, we have 
\begin{align*}
&w(\bm{p})=\begin{spmatrix} a \\ a+1 \\ a+2 \end{spmatrix},  
w(\bm{q})=\begin{spmatrix} b+2 \\ b+1 \\ b \end{spmatrix},
w(\bm{t})=\begin{spmatrix} c+2 \\ c+1 \\ c \end{spmatrix},\\
&w(p_4)=2-c+a,w(r)=b+3, w(u)=c+2b-a+4,\\
&K_{\Sigma^{12}_{\mP}}=\sO(-k)\ \text{with} \ k=a+2b+2c+17
\end{align*}    
for some $a,b,c\in \mN$ by the constraints of the weights of variables as in the subsection \ref{sub:Weights}.    

In the latter case, we have 
\begin{align*}
&w(\bm{p})=\begin{spmatrix} a \\ a \\ a+1 \end{spmatrix},  
w(\bm{q})=\begin{spmatrix} b+1 \\ b+1 \\ b \end{spmatrix},
w(\bm{t})=\begin{spmatrix} c+1 \\ c+1 \\ c \end{spmatrix},\\
&w(p_4)=1-c+a,w(r)=b+2, w(u)=c+2b-a+3,\\
&K_{\Sigma^{12}_{\mP}}=\sO(-k)\ \text{with} \ k=a+2b+2c+11
\end{align*}    
for some $a,b,c\in \mN$ by the constraints of the weights of variables as in the subsection \ref{sub:Weights}.    
\end{rem}

\section{\bf Proof of Theorem \ref{thm:main}}
\label{sec:Example}

In this section, we show Theorem \ref{thm:main} (1-1) and (2). 
Only for such $X$ and $T$ in case of No.24078, examples are constructed from $\Sigma^{14}_{\mA}$. Our construction for this case is similar to those in the other cases where examples are constructed from $\Sigma^{13}_{\mA}$. Therefore we only treat the cases exept No.24078 below.


\subsection{Strategy}
\label{sub:strategy}
In this subsection, we explain in detail our strategy how to verify a weighted complete intersection of $\Sigma^{12}_{\mP}$ in each case is a desired example of the prime $\mQ$-Fano threefold. In the subsections \ref{Case1}--\ref{Case3}, we carry out the verification following the strategy; in those subsections, we often omit the common explanation in each case with reference to explanation in this subsection.

\vspace{3pt}

Fix a class No.$*$ except No.24078 as in Table \ref{Table 1}.
Hereafter in this section, we denote $\SigmaP$ by $\Sigma$ and $\SigmaA$ by $\Sigma_{\mA}$ for simplicity of notation.
In each case, we set 
\begin{equation}
\label{eq:Xgen}
X=\Sigma\cap (a_1)^{m_1}\cap \dots \cap (a_k)^{m_k},
\end{equation}
where $(a_i)$ for each $i$ means a general weighted hypersurface of weight $a_i$, and $a_1,\dots,a_k$, $m_1,\dots,m_k\in \mN$ are assigned for No.$*$ as in Table \ref{Table 2}, and we may assume that they satisfy  
$a_1<\cdots <a_k$ and $m_1+\cdots+m_k=9$.

Let $T$ be the intersection of $X$ and a general hypersurface of weight one. We write $T$ as follows:
\begin{equation}
\label{eq:Tgen}
T=\Sigma\cap (b_1)^{n_1}\cap \dots \cap (b_l)^{n_l},
\end{equation}
where
\[
\begin{cases}
\text{$l=k$, $b_i=a_i$ for any $i\geq 1$, $n_1=m_1+1$, $n_i=m_i$ for any $i\geq 2$ if $a_1=1$,}\\
\text{$l=k+1$, $b_1=1$, $b_i=a_{i-1}$ for any $i\geq 2$, $n_1=1$, $n_i=m_{i-1}$ for any $i\geq 2$ if $a_1>1$}.
\end{cases}
\]

We use the following notation for intermediate varieties appearing in the process to obtain $T$ from $\Sigma$:
\[
\text{$\Sigma(b_i):=\Sigma\cap (b_1)^{n_1}\cap \dots \cap (b_i)^{n_i}$ for any $i$}.
\]
\vspace{3pt}

\noindent {\bf List of notation}:

We use the following notation for a weighted projective variety $V$ considered below (for example, $V=X,T,\Sigma(b_i)$):
\begin{itemize}
\item
$\mP_V:=$ the ambient weighted projective space for $V$. 
\item
$V_{\mA}:=$ the affine cone of $V$.
\item
$V_{\mA}^o:=$ the complement of the origin in $V_{\mA}$.
\item
$\mA_V:=$ the ambient affine space for $V_{\mA}$.
\item
$\mA_{V}^*:=$ the affine space obtained from $\mA_V$ by setting the coordinate $*=1$.
\item
$V^*_{\mA}:=$ the restriction of $V_{\mA}$ to $\mA_{V}^*$.
\item
$|\sO_{\mP_V}(i)|_{\mA}:=$ the set of the affine cones of members of $|\sO_{\mP_V}(i)|$.
\item
$\Bs |\sO_{\mP_V}(i)|_{\mA}:=$ the intersection of all the members of $|\sO_{\mP_V}(i)|_{\mA}$.
\end{itemize}

\vspace{3pt}

\noindent {\bf Equations of $X$ and $T$}:

In each case, it is important to observe that the number of the coordinates of $\mP_{\Sigma}$ with weight $a_i$ 
is greater than or equal to $m_i$ for any $i$ (we refer to Tables \ref{Table 1} and \ref{Table 2} for this).
Therefore, for $X$, we may choose each of the sections $(a_i)$ as defined by the equation of the form
\begin{align}
\label{sectioneq}
&\text{(a coordinate of $\mP_{\Sigma}$ with weight $a_i)=$}\\
&\quad \text{(a polynomial with weight $a_i$ of other coordinates)}.\nonumber
\end{align}
We refer for explicit forms of (\ref{sectioneq}) to the examples given below in the subsections \ref{Case1}--\ref{Case3}. 
Moreover, for $T$, we just add to the sections defining $X$
one general section of weight one as defined by the equation of the form (\ref{sectioneq}) since the number of the coordinates of $\mP_{\Sigma}$ with weight one
is greater than or equal to $n_1$.

Substituting the r.h.s.\,of the equations of the sections of the form (\ref{sectioneq}) for the corresponding variables of the nine equations of $\Sigma$ (\ref{Key1}) and (\ref{Key2}), we obtain the equations of $X$ and $T$ in their ambient weighted projective spaces $\mP_X$ and $\mP_T$ respectively.

As for writing down the equations of them, we perform the case division according to the value of $h^0(\sO_{\mP_X}(1))$; $h^0(\sO_{\mP_X}(1))\geq 3$ ($2$ classes: No.11004, 16227), $h^0(\sO_{\mP_X}(1))=2$ (7 classes: No.4850--6865), or $h^0(\sO_{\mP_X}(1))=1$ (14 classes: No.360--No.2422). We often perform this case division below.  
In the case that $h^0(\sO_{\mP_X}(1))\geq 3$, it turns out that we do not need explicit equations of $X$ and $T$.
In the case that $h^0(\sO_{\mP_X}(1))=1$, the equations of $X$ are complicated since they have many parameters. So we write down the equations of $T$ instead of those of $X$ since they are reasonably simple. Explicitly, we refer to the subsection \ref{Case1}.
In the case that $h^0(\sO_{\mP_X}(1))=2$, the full equations of $T$ are also complicated.
Hence we take two sections $T,T'$ of $X$ by general members of $|\sO_{\mP_X}(1)|$ and set $C:=T\cap T'$.
Then we give the full equations of $C$, which are reasonably simple.
Explicitly, we refer to the subsection \ref{Case2}.

\vspace{5pt}

\noindent {\bf Checkpoints (A), (B), (C)}:
For each case, we will check the following claims (A)--(C) for $X$ and $T$:
\begin{enumerate}[(A)]
\item $X$ and $T$ are quasi-smooth, namely, $X_{\mA}^o$ and $T_{\mA}^o$ are smooth.
\item $X$ and $T$ have only cyclic quotient singularities assigned for No.$*$.
\item $X\cap \{p_1=0\}$ is a prime divisor.
\end{enumerate}

By these claims, we can finish a part of the proof of Theorem \ref{thm:main}. Indeed,
by (A), (C) and Corollary \ref{QFanoDiv}, $X$ has Picard number one. By this fact, and (A), (B) and Corollaries \ref{cor:candivdeg} and \ref{cor:ac4},
we see that $X$ is an example of a $\mQ$-Fano threefold of No.$*$, which shows Theorem \ref{thm:main} (1-1). By  (A) and (B), $T$ is a quasi-smooth $K3$ surface of No.$*$ since $T\in |-K_X|$, which shows Theorem \ref{thm:main} (2).

\vspace{3pt}

\noindent {\bf Claims (A) and (B)}:
By the Bertini singularity theorem (cf.\cite{Kl}), we see that
$\Sing X_{\mA}^o$ is contained in $(\bigcup_i \Bs |\sO_{\mP_{\Sigma}}(a_i)|_{\mA})\cup \Sing {\SigmaA}^o$, and
$\Sing T_{\mA}^o$ is contained in $(\bigcup_i \Bs |\sO_{\mP_{\Sigma}}(b_i)|_{\mA})\cup \Sing {\SigmaA}^o$.
Note that the dimension of the singular locus of $\SigmaA$ is less than the codimension of $X$ in $\Sigma^{13}_{\mA}$ by Proposition \ref{prop:Sing13}. Therefore we see that 
$(\Sing {\SigmaA}^o)\setminus (\bigcup_i \Bs |\sO_{\mP_{\Sigma}}(a_i)|_{\mA})$ is disjoint from $X_{\mA}^o$, and $(\Sing {\SigmaA}^o)\setminus (\bigcup_i \Bs |\sO_{\mP_{\Sigma}}(b_i)|_{\mA})$ is disjoint from $T_{\mA}^o$
for general $X$ and $T$.
Therefore it holds that
\begin{equation}
\label{eq:SingXT}
\Sing X_{\mA}^o\subset \bigcup_i \Bs |\sO_{\mP_{\Sigma}}(a_i)|_{\mA}, \ \text{and}\
\Sing T_{\mA}^o\subset \bigcup_i \Bs |\sO_{\mP_{\Sigma}}(b_i)|_{\mA}.
\end{equation}

\vspace{3pt}

\noindent {\bf Reduction of (A) and (B) to $T$}:
We see that $\Bs |\sO_{\mP_{\Sigma}}(a_i)|_{\mA}, \Bs |\sO_{\mP_{\Sigma}}(b_i)|_{\mA}$ are contained in $\Bs |\sO_{\mP_{\Sigma}}(1)|_{\mA}$ for any $i$ since there exists a coordinate of $\mA_{\Sigma}$ with weight one in each case.
Therefore
$\Sing X_{\mA}^o$ and $\Sing T_{\mA}^o$ are contained in $\Bs |\sO_{\mP_{\Sigma}}(1)|_{\mA}$,
which implies that, to prove (A) for $X,T$, it suffices to show  
$X_{\mA}^o, T_{\mA}^o$ is smooth along $\Bs |\sO_{\mP_{\Sigma}}(1)|_{\mA}$.
Moreover,
we see that 
the smoothness of $T_{\mA}^o$ implies that of $X_{\mA}^o$
since $\Sing X_{\mA}^o\subset \Bs |\sO_{\mP_{\Sigma}}(1)|_{\mA}$ and $T_{\mA}^o$ is the intersection of a general hypersurface of weight one and $X_{\mA}^o$, and
hence  
$T_{\mA}^o$ is a Cartier divisor of $X_{\mA}^o$.
Further, if we show that $(B)$ holds for $T$, i.e., $T$ has the singularity of the desired type $1/\alpha\, (\beta,\alpha-\beta)$ at each point ${\sf t}$,
then we immediately see that $(B)$ holds for $X$, i.e.,
$X$ has the desired $1/\alpha (1, \beta,\alpha-\beta)$-singularity at ${\sf t}$ since $T=X\cap (1)$.

In summary, we have shown the following claim:
\begin{cla}
\label{cla:RedAB}
As for $(A)$ and $(B)$, it suffices to show them for $T$.
\end{cla}

\vspace{3pt}

\noindent{\bf Base loci}:
It is important to calculate $\Bs |\sO_{\mP_{\Sigma}}(b_i)|_{|T}$ for any $i\geq 1$.
We write down the results of calculations in the item {\bf ($*$: Base Loci)} for each No.$*$.
Since the calculations are easy, we omit them and only write down the results.

In summary, we have the following: 
\begin{itemize}
\item 
In the case that $h^0(\sO_{\mP_X}(1))\geq 3$, 
\begin{equation}
\label{eq:BsCase1}
\text{$\Bs |\sO_{\mP_{\Sigma}}(b_i)|_{|T}$ consists of the $u$-point for any $i$.}
\end{equation} 

\item 
In the case that $h^0(\sO_{\mP_X}(1))=2$, we write down $\Bs |\sO_{\mP_{\Sigma}}(b_i)|_{|C}$ instead, where $C=T\cap T'$ as in the item {\bf Equations of $X$ and $T$} above.
We also write down the locus $\{p_1=p_2=0\}_{|C}$.
\item
In the case that $h^0(\sO_{\mP_X}(1))=1$, we write down $\Bs |\sO_{\mP_{\Sigma}}(b_i)|_{|T}$ and $\{p_1=p_2=0\}_{|T}$.
\end{itemize}

\vspace{3pt}

\noindent {\bf Claims (A) and (B) for $T$}:

\vspace{1pt}

\noindent \underline{Case $h^0(\sO_{\mP_X}(1))=1$}: First we consider the case that $h^0(\sO_{\mP_X}(1))=1$. In this case, we show the claims (A) and (B) for $T$ separately on the $p_1$-chart, the $p_2$-chart, and the locus $\{p_1=p_2=0\}$.

\vspace{3pt}

\noindent {\bf $p_1$- and $p_2$-charts}: 
We see that the $p_1$- and $p_2$-charts of $\Sigma_{\mA}$ are open subsets of $\mA^{13}$ by the coordinate descriptions of them as in the subsection \ref{Charts}.
Therefore $T_{\mA}^o$ is an open subset of a complete intersection.
Noting this, we check the singularities of $T_{\mA}^o$ on the $p_1$- and $p_2$-charts both by the Jacobian criterion and by the method which we will call LPC (we will explain this below); by the Jacobian criterion and investigations of the several base loci as we have mentioned in the item ({\bf Base loci}), it turns out that there remain a finite number of possibilities of the singular points of $T_{\mA}^o$. Then adopting LPC at these points, we see that $T_{\mA}^o$ is actually smooth, i.e., we can show the claim (A) for $T$ on the $p_1$- and $p_2$-charts. Moreover, by LPC, we may also determine the type of singularities of $T$, i.e., we can show the claim (B) for $T$ on the $p_1$- and $p_2$-charts.
We refer to the explanations of  No.360 and No.393 for details. The arguments are similar in other cases so we do not write down them except when additional arguments are needed.

\vspace{3pt}

\noindent {\bf Locus $\{p_1=p_2=0\}$}: 
In each case, we can easily compute $\{p_1=p_2=0\}_{|T}$, which is written down in the item ({\bf Base loci}). Then we observe that
$\{p_1=p_2=0\}_{|T}$ consists of finite number of points. We will adopt LPC at these points to show the claims (A) and (B) for $T$ on $\{p_1=p_2=0\}$ as in the case of the $p_1$- and $p_2$-charts. 
In the case of No.360, No.569, and No.1218, however, we also use the Jacobian criterion on the $u$-chart;
since the $u$-chart of $T_{\mA}^o$ is also an open subset of a complete intersection, we can argue
similarly to the case of the $p_1$- and $p_2$-charts.

\vspace{3pt}

\noindent {\bf Linear Part Computation (LPC)}: This is one of the main methods to show both (A) and (B) at the same time. 
We refer for explicit application examples to the investigation of the $p_2$-chart of $T$ in the item ({\bf 360iv: $p_1$- and $p_2$-charts}), and the item ({\bf 393iii: LPC}).
We adopt this method both for $X$ and $T$. In the item {\bf ($*$: LPC)} in the case $*$, we write the singularity of the point of $T$ for which we apply LPC.
\vspace{3pt}

\noindent {\bf LPC for $T$}: We fix a point ${\sf t} \in T$ for which we shall adopt LPC. We choose a coordinate $x$ of $\mP_T$ which is not zero for $\sf{t}$, and one point ${\sf t'}\in T_{\mA}^x$ corresponding to ${\sf t}$. We localize the equations of $T$ at ${\sf t'}$ setting $x=1$ and transforming the coordinates of $\mA_T^x$ to make the point ${\sf t}'$ the origin. 
For example, in the case of the $u$-point, we have only to set $u=1$ in the equations. 
We compute the degree one parts of the equations and show that they span a four dimensional subspace in the cotangent space of ${\mA}_{T,x}$ at ${\sf t}'$ (the name LPC comes from this computation).
Then we see that there are two local coordinates of $T_{\mA,x}$ at ${\sf t}'$. This implies that $T$ is quasi-smooth at ${\sf t}$. Let $\alpha$ be the order of the stabilizer group at ${\sf t}'$ of the group action induced on $\mA_T^x$ from $\mP_T$. In each case, we see that the weights of the two local coordinates with respect to the $\mZ/\alpha \mZ$-action are $\beta$ and $\alpha-\beta$ with some $\beta\in \mN$ coprime to $\alpha$. Thus $T$ has a $1/\alpha (\beta,\alpha-\beta)$-singularity at ${\sf t}$ which is the desired type in each case.

\vspace{3pt}

\noindent {\bf LPC at the $u$-point}:
Here we do not restrict to the case that $h^0(\sO_{\mP_X}(1))=1$.
We see that $X$ contains the $u$-point except in the cases of No.360 and No.1218.
At the $u$-point, we adopt LPC to answer (A) and (B) for any No.$*$. Explicitly, we refer to the explanation of No.393. 
As for LPC at the $u$-point, we do not need the full equations of $X$ or $T$. So we may use this method both for $X$ and $T$.
We explain this reason for $T$. As for the equations of the sections as in (\ref{sectioneq}),
we have only to present only degree one parts of them. Indeed, the $u$-point is the origin of $\mA_T^u$ if we set $u=1$.
We observe that, in each case, $a_i$ for any $i\geq 1$ is less than $w(u)$, hence $u$ does not appear in the equations of the sections as in
(\ref{sectioneq}). Therefore, by setting $u=1$, these equations of the sections do not change.
The higher degree parts of them do not contribute to the degree one parts of the equations of $T$ with $u=1$, hence we may ignore them when we adopt LPC.

\vspace{3pt}

\vspace{1pt}

\noindent \underline{Case $h^0(\sO_{\mP_X}(1))=2$}:
Now we consider the case that $h^0(\sO_{\mP_X}(1))=2$.
As we have mentioned above in the item {\bf LPC at the $u$-point}, the claims (A) and (B) for $T$ at the $u$-point follow from LPC. 
Outside the $u$-point, the smoothness of $C_{\mA}^o$ implies that of $T_{\mA}^o$. Indeed, this can be shown by noting $h^0(\sO_{\mP_T}(1))=1$ and repeating for $T$ and $C$ what we did for $X$ and $T$ in the item {\bf Reduction of (A) and (B) to $T$} above.
As in the case that $h^0(\sO_{\mP_X}(1))=1$, the smoothness of $C_{\mA}^o$ will be checked separately on the $p_1$-chart, the $p_2$-chart, and the locus $\{p_1=p_2=0\}_{|C}$, and 
we use the Jacobian criterion and LPC. 

\vspace{3pt}

\noindent {\bf LPC for $C$}: LPC works also for $C$. We fix a point ${\sf t} \in C$ for which we shall adopt LPC. We choose a coordinate $x$ of $\mP_C$ which is not zero for $\sf{t}$, and one point ${\sf t'}$ corresponding to ${\sf t}$. 
Exactly in the same way as LPC for $T$,
we see that there is one local coordinate of $C_{\mA}^x$ at ${\sf t}'$. Let $\alpha$ be the order of the stabilizer group at ${\sf t}'$ of the group action induced on $\mA_C^x$ from $\mP_C$. In each case, we see that the weight of the local coordinate with respect to the $\mZ/\alpha \mZ$-action is $\alpha-1$. Then, $T$ has a $1/\alpha\, (1,\alpha-1)$-singularity at ${\sf t}$ and $X$ has a $1/\alpha\, (1,1,\alpha-1)$-singularity at ${\sf t}$.

\vspace{3pt}

\vspace{1pt}

\noindent \underline{Case $h^0(\sO_{\mP_X}(1))\geq 3$}: Finally, in the case that $h^0(\sO_{\mP_X}(1))\geq 3$, the claims (A) and (B) follows from LPC at the $u$-point
by (\ref{eq:SingXT}) and (\ref{eq:BsCase1}).

\vspace{3pt}

\noindent {\bf $1/2(1,1)$-singularities}:
In many cases, $T$ turns out to have $1/2(1,1)$-singularities.
Usually, the coordinates of $1/2(1,1)$-singularities are complicated and hence LPC is not appropriate for them.
It suffices, however, to obtain the number of the candidates of $1/2(1,1)$-singularities and to check it is finite since we can immediately conclude that a point is a $1/2(1,1)$-singularity once we know it is obtained locally by the $\mZ_2$-quotient of a smooth surface with an isolated fixed point.
So LPC is not needed and the Jacobian criterion is sufficient for them.
We refer for an explicit application example to the item ({\bf 642iv: $p_1$-chart}) in the treatment of No.642. 

\vspace{5pt}

\noindent {\bf Claim (C)}:
Assuming that we have checked (A) and (B) for $X$,
we finish our proof of the claim (C) here.

Note that $X$ is normal and irreducible . Indeed, the normality follows from (A) and (B). The connectedness, hence the irreducibility  of $X$ follows since so is $\Sigma$ by Proposition \ref{prop:Tom14} and Lemma \ref{lem:1314} and $X$ is a weighted complete intersection in $\Sigma$ by ample divisors.

Hence, to prove $X\cap \{p_1=0\}$ is irreducible, it suffices to show that it is normal since $X\cap \{p_1=0\}$ is an ample divisor of $X$. Since $X$ has only terminal singularities by the claim (B), $X\cap \{p_1=0\}$ is Cohen-Macaulay (cf.~\cite[Cor.~5.25]{KM}).
Hence, to prove that $X\cap \{p_1=0\}$ is normal, it suffices to show that $\dim \Sing (X\cap \{p_1=0\})\leq 0$.
Since $\SigmaP\cap \{p_1=0\}$ is normal by Proposition \ref{prop:Tom14} (3) and Lemma \ref{lem:1314}, 
we have only to check 
\begin{equation}
\label{0dimT}
\text{$\dim \Sing (T\cap \{p_1=0\})\leq 0$ along $\bigcup_i \Bs |\sO(a_i)|$}
\end{equation}
by a similar discussion to that for (A) and (B).

First we consider the case that $h^0(\sO_{\mP_X}(1))=1$. Recall that in this case, the full equation of $T$ is given (see the item {{\bf Equations of $X$ and $T$}} as above).
Similarly to (A) and (B), we check (\ref{0dimT}) on the $p_2$-chart and the locus $\{p_1=p_2=0\}$ separately using the equation of $T$.
On the $p_2$-chart, the check is done by the Jacobian criterion since they are open subsets of complete intersections.
As for $\{p_1=p_2=0\}$, we are done since the restrictions of several base loci to $\{p_1=p_2=0\}$ consists of a finite number of points as we have mentioned above in the item {\bf Locus $\{p_1=p_2=0\}$}.
In the case that $h^0(\sO_{\mP_X}(1))=2$, 
it is easy to see that $\Bs |\sO_{\mP_{\Sigma}}(a_i)|\cap X\cap \{p_1=0\}=\Bs |\sO_{\mP_{\Sigma}}(a_i)| \cap C\cap \{p_1=0\}$ for any $i$.
We see that $\Bs |\sO_{\mP_{\Sigma}}(a_i)|\cap C\cap \{p_1=0\}$ consists of finite number of points except in the case of No.5859.
So we are done in those cases.
As for No.5859, we can argue similarly to the case that $h^0(\sO_{\mP_X}(1))=1$.
Finally, in the case that $h^0(\sO_{\mP_X}(1))\geq 3$, we are done by (\ref{eq:BsCase1}).$\hfill\square$

\vspace{10pt}


\subsection{Case $h^0(\sO_{\mP_X}(1))=1$}~

\label{Case1}
\vspace{3pt}
\noindent \fbox{\bf No.360}\,

\vspace{3pt}

\noindent {\bf (360i: Equation of $T$)}
{\small{
\begin{center}
$\begin{array}{|c|c|c|c|c|c|c|c|}
\hline
\text{weight} & 1 & 2 & 3  & 4 & 5 &  6 &  8 \\
\hline
\text{equation} & s_{23}=0 & s_{13}=0 & s_{12}=0 & q_3=a_3 t_2 & q_2=a_2 t_1 & q_1=a_1 p_4 & 
u=a_0 p_2+b_0 t_2^2\\
& & s_{22}=0 & t_3=0 & s_{11}=a_{11}t_2 & & & \\
\hline
\end{array}$
\end{center}
}}
In this table, the six parameters $a_3,\dots,b_0\in \mC$ are chosen generally.
Then note that 
\begin{equation}
\label{eq:360Tamb}
T\subset \mP(t_2,t_1,p_4,p_1,r,p_2,p_3)= \mP(4,5,6,7^2,8,9).
\end{equation}

\vspace{3pt}

\noindent {\bf (360ii: Base loci)}\quad 
\begin{center}
$\begin{array}{|c|c|}
\hline
\{p_1=p_2=0\}_{|T} & \Bs |\sO(8)|_{|T}\\
\hline
\text{$p_4$-point}, \text{the locus $(\ref{u-chart360})$} & \text{$p_1$-point, $p_4$-point}\\
\hline
\end{array}$
\end{center}

\vspace{3pt}

\begin{equation}
\label{u-chart360}
\{-a_2^2 t_1^3 + a_3 r t_2^2 = 0,\, 
  r t_1 + a_{11}  a_3  t_2^3 = 0,\, -r^2 - a_{11} a_2^2 t_1^2 t_2 = 0\}
\subset \mP(r,t_1,t_2).
\end{equation}
Note that, for any point of the locus
(\ref{u-chart360}), none of $r,t_1,t_2$ is zero. 

\vspace{3pt}

\noindent {\bf (360iii: LPC)}\quad
\begin{center}
$\begin{array}{|c|c|}
\hline
\text{$p_1$-point} & \text{$p_4$-point} \\
\hline
1/7(2,5) & 1/6(1,5)\\
\hline
\end{array}$
\end{center}

We adopt LPC at the $p_1$- and $p_4$-points by setting $p_1=1$ and $p_4=1$ respectively.
The way to adopt LPC for them is very similar to that for the $u$-point in other cases. We refer to the item ({\bf 393iii}) for the explanations of LPC at the $u$-point. 

\vspace{3pt}

\noindent {\bf (360iv: $p_1$- and $p_2$-charts)}

\vspace{3pt}

\noindent \underline{\bf $p_1$-chart}:
We show that, on this chart, $T$ is quasi-smooth and 
has one $1/7(2,5)$-singularity at the $p_1$-point. 

By the Jacobian criterion, we see $\Sigma(5)_{\mA}^o$ is smooth, and $\Sigma(6)_{\mA}^o$ is smooth along $\Bs |\sO(6)|$, hence
$\Sigma(6)_{\mA}^o$ is smooth. Detailed calculations are omitted. We refer to No.393 for similar calculations. Now note that $T_{\mA}^o=\Sigma(8)^o_{\mA}$.
By the description of $\Bs |\sO(8)|_{|T}$ as above, it consists of only the $p_1$-point in the $p_1$-chart.
We have already determined by LPC that $T$ has a $1/7(2,5)$-singularity at this point. Therefore $T$ is quasi-smooth on this chart.

We show that the $p_1$-point is the unique singularity of $T$ on this chart. 
Since $w(p_1)=7$, the cyclic group $\mZ_7$ acts on $\mA_T^{p_1}$.
By (\ref{eq:360Tamb}), we see that the action on $\mA_T^{p_1}$ is not free only along the $r$-line, which intersects $T_{\mA}^{p_1}$ at the $p_1$-point by the last equation of $\Sigma^{14}_{\mA}$ in (\ref{Key2}). As we have seen above, $T$ has a $1/7(2,5)$-singularity at the $p_1$-point.

Therefore $T$ is quasi-smooth and has one $1/7(2,5)$-singularity on this chart.

\vspace{3pt}

\noindent \underline{\bf $p_2$-chart}:
We show that, on this chart, $T$ is quasi-smooth and 
has two $1/4(1,3)$-singularities. 

We may check that 
$\Sing \Sigma(4)_{\mA}=\{p_1=p_3=p_4=s_{23}=t_2=0\}\cap \Sigma(4)_{\mA}$.
Checking $\Sing \Sigma(5)_{\mA}$ along $\Bs |\sO(5)|\cap \{p_1=0\}$,
we see that 
$\Sing \Sigma(5)_{\mA}=\Sing \Sigma(4)_{\mA}\cap \Sigma(5)_{\mA}$.
Similarly, we see that 
$\Sing \Sigma(6)_{\mA} =\Sing \Sigma(4)_{\mA}\cap \Sigma(6)_{\mA}$.
These are checked by the Jacobian criterion.
We see that $\Sing \Sigma(6)_{\mA}$ will disappear by taking a general octic section. 
We also note that $\Bs |\sO(8)|$ is disjoint from the $p_2$-chart since $w(p_2)=8$. Therefore $T$ is quasi-smooth on this chart.

We determine the singularities of $T$ on this chart. 
Since $w(p_2)=8$, the cyclic group $\mZ_8$ acts on $\mA_T^{p_2}$.
By (\ref{eq:360Tamb}), we see that the action on $\mA_T^{p_2}$ is not free along $\{t_1=p_1=r=p_3=0\}$. This intersects $T_{\mA}^{p_2}$ along
$\{-a_{11} a_3^2t_2^4  + b_0 t_2^2 + a_0 = 0\}$
in the $t_2$-line,
which map to two points on $T$ by the $\mZ_8$-quotient.
We adopt LPC for these two points. For this, we set $t_2=1$ instead of setting $p_2=1$.
In $\mA_T^{t_2}$, the inverse image of these two points on $T$ are 
$T_{\mA}^{t_2}\cap \{-a_{11} a_3^2  + b_0 p_2 + a_0p_2^2 = 0\}$, which consists of two points.
Choose one $A\in \mC$ such that $A^2=b_0^2 + 4 a_0 a_{11} a_3^2$.
Then the $p_2$-coordinates of these two points in $T_{\mA}^{t_2}$ are 
$\frac{-b_0 \pm A}{2 a_0}$.
By the corresponding coordinate change $p_2 = p'_2 + \frac{-b_0 + A}{2 a_0}$, or $p_2 = p'_2 + \frac{-b_0 - A}{2 a_0}$
for each of two points, it becomes the origin of the $\mA_T^{t_2}$.
Then we can apply LPC.
Consequently, we see that $T$ has $1/4 (1,3)$-singularities at these two points. 
Therefore, a general $T$ is quasi-smooth and 
has two $1/4(1,3)$-singularities on this chart.

\vspace{3pt}

\noindent \noindent {\bf (360v: $\{p_1=p_2=0\}$)}
We check that $T_{\mA}^o$ is smooth along $\{p_1=p_2=0\}$ refering to the item {\bf (360ii)}.
As for the $p_4$-point, we are already done by LPC.
As we have seen above, any point $x$ of (\ref{u-chart360}) satisfies $t_2\not =0$, thus we see that $x$ is contained in the $u$-chart of $\Sigma(6)_{\mA}$ by the equation of $T$ of weight eight.
Recall the coordinate description of the $u$-chart of $\Sigma^{14}_{\mA}$ as in the subsection \ref{Charts}.
By a similar method to that we adopted in the case of the $p_1$-chart,
we see that $T_{\mA}^o$ is smooth in the $u$-chart of $\Sigma(6)_{\mA}$ by the Jacobian criterion.

\vspace{5pt}


\noindent \fbox{\bf No.393}

\vspace{3pt}

\noindent {\bf (393i: Equation of $T$)}
{\small{
\begin{center}
$\begin{array}{|c|c|c|c|c|c|c|c|}
\hline
\text{weight} & 1 & 2 & 3  & 4 & 5 &  6 &  7 \\
\hline
\text{equation} & s_{23}=0 & s_{13}=0 & s_{12}=0 & q_3=a_3 t_2 & q_2=a_2 p_4+b_2 t_1 & q_1=a_1 p_1 & 
r=a_0 p_2\\
& & s_{22}=0 & t_3=0 & s_{11}=a_{11}t_2 & & & \\
\hline
\end{array}$
\end{center}
}}
In this table, the six parameters $a_3,\dots, a_0\in \mC$ are chosen generally.
Note that 
\begin{equation}
\label{eq:393X}
T\subset \mP(t_2,p_4,t_1,p_1,p_2,p_3,u)\simeq \mP(4,5^2,6,7,8,9).
\end{equation}

\vspace{3pt}

\noindent {\bf (393ii: Base loci)}
{\small{
\[
\begin{array}{|c|c|}
\hline
\{p_1=p_2=0\}_{|T} = \Bs |\sO(i)|_{|T}\cap \{p_1=p_2=0\} (i=2,3)  & \Bs |\sO(5)|_{|T}\cap \{p_2=0\} \\
\qquad \qquad \quad\,= \Bs |\sO(i)|_{|T}\cap \{p_2=0\} (i=4,6) & \\
\hline
 \text{$u$-point}, \,\text{$p_4$-point},\, \sf{p}_5    & \text{$u$-point}, \, \sf{p}_2\\
\hline
\end{array}
\]
}}
In this table,
\begin{align*}
&{\sf{p}_5}:=\{a_2 p_4+b_2 t_1=0\}\in \mP(p_4,t_1),\\
& {\sf{p}_2}:=\{a_1 p_3-a_{11} a_3 t_2^2=0, a_1^2 p_1^2+a_{11} a_3^2 t_2^3=0\}\in \mP(p_1,p_3,t_2).
\end{align*}

\vspace{3pt}

\noindent {\bf (393iii: LPC)}\quad
\begin{center}
$\begin{array}{|c|c|c|c|}
\hline
\text{$u$-point} & \text{$p_4$-point} & \sf{p}_5 & \sf{p}_2\\
\hline
1/9(4,5) & 1/5(1,4) & 1/5(2,3) & 1/2(1,1)\\
\hline
\end{array}$
\end{center}
\vspace{3pt}

\noindent \underline{\bf LPC at the $u$-point}:
Setting $u=1$, we apply LPC at the $u$-point. Note that there is a unique point ${\sf u}'$ of $\mA_T^u$ corresponding to the $u$-point, and the coordinates of $\mA_T^u$ are $t_2,p_4,t_1,p_1,p_2,p_3$ by (\ref{eq:393X}). We may easily compute that the linear parts of the equations of $T$ with $u=1$ generates the subspace spanned by $p_1,p_2,p_3,p_4$ of the cotangent space of $\mA_T^u$ at ${\sf u}'$. Thus, as a local coordinate of $T_\mA^u$ at the point ${\sf u}'$,
we may take $t_2,t_1$. Noting that 
$w(t_2)=4$, and $w(t_1)=5$, we see that $T$ has a $1/9(4,5)$-singularity at the $u$-point.

\vspace{3pt}
 
\noindent \underline{\bf LPC at the $p_4$-point}: Setting $p_4=1$, we apply LPC at the $p_4$-point in a very similar way to that for the $u$-point. We omit detailed explanations. Consequently, $T$ has a $1/5(1,4)$-singularity at the $p_4$-point.

\vspace{3pt}

\noindent \underline{\bf LPC at the point $\sf{p}_5$}:
Setting $t_1=1$, we apply LPC at ${\sf p}_5$ after the coordinate change $p_4=p'_4-b_2/a_2$.
Detailed explanations are omitted. Consequently, $T$ has a $1/5(2,3)$-singularity at ${\sf p}_5$.

\vspace{3pt}

\noindent \underline{\bf LPC at the point ${\sf p}_2$}:
We set $t_2=1$. There are two points of $\mA_T^{t_2}$ corresponding to the point ${\sf p}_2$.
Choose one $A\in \mC$ such that $A^2=-\frac{a_{11} a_3^2}{a_1^2}$ and let ${\sf p}'_2$ be the point of $\mA_T^{t_2}$ with $p_2=A$
corresponding to the point ${\sf p}_2$. After the coordinate change $p_3=p'_3+\frac{a_{11}a_3}{a_1}$ and
$p_1=p'_1+A$, we apply LPC. Detailed explanations are omitted. Consequently, $T$ has a $1/2(1,1)$-singularity at ${\sf p}_2$.

\vspace{3pt}

\noindent {\bf (393iv: $p_1$- and $p_2$-charts)}

\vspace{3pt}

\noindent \underline{\bf $p_1$-chart}:
We show that, on this chart, $T$ is quasi-smooth, and has a $1/2 (1,1)$-singularity at $\sf{p}_2$.

To show that $T$ is quasi-smooth, we wiil check step by step that $\Sigma(3)_{\mA}^o$, $\Sigma(5)_{\mA}^o$, $\Sigma(6)_{\mA}^o$, and $\Sigma(7)_{\mA}^o=T_{\mA}^o$ are smooth.   

By the equations of $T$ of weight $2$ and $3$ and the coordinate description of the $p_1$-chart as in the subsection \ref{Charts}, we immediately see that $\Sigma(3)_{\mA}^o$ is smooth on this chart.
By the equations of $T$ of weight $2,3,4,5$ and the expression of $s_{11}$ on the $p_1$-chart as in the subsection \ref{Charts},
we see that
$\Sigma_{\mA}(5)^o=
\{a_{11} t_2 +p_3^2  + p_4 t_1 +p_2 p_4 t_2=0\},$
which is easily to be seen smooth.
As for the singularities of $\Sigma(6)_{\mA}^o$,
we have only to check it along $\Bs|\sO(6)|_{\mA}$ since we have already shown that $\Sigma(5)_{\mA}^o$ is smooth. 
Since $w(p_1)=6$, the $p_1$-chart is disjoint from $\Bs |\sO(6)|_{\mA}$.
Therefore $\Sigma_{\mA}(6)^o$ is smooth.
Similarly, as for the singularities of $\Sigma(7)_{\mA}^o$,
we have only to check it along $\Bs|\sO(7)|_{\mA}$ since we have shown that $\Sigma(6)_{\mA}^o$ is smooth on this chart. 
The check can be done again by the Jacobian criterion.

\vspace{3pt}

Finally, we determine the singularities of $T$ on this chart. 
Since $w(p_1)=6$, the cyclic group $\mZ_6$ acts on this chart of $\mA_T^{p_1}$. By (\ref{eq:393X}), we see that the action on $\mA_T^{p_1}$ is not free only along $\{t_2=p_4=t_1=u=0\}\cup \{t_2=p_4=t_1=p_2=p_3=0\}$, which intersects $T_{\mA}$ at the inverse image of the point $\sf{p}_2$ (we see this by the equation of $T$).
We have already shown that $T$ has a $1/2 (1,1)$-singularity at the point $\sf{p}_2$.

\vspace{3pt}

\noindent \underline{\bf $p_2$-chart}:
We show that $T$ is smooth on this chart.
The argument below is very similar to that in the case of the $p_1$-chart.

By the equations of $T$ of weight $2,3,4$ and the expressions of $s_{22}$ on the $p_2$-chart as in the subsection \ref{Charts},
we see that
$\Sigma_{\mA}(4)^o=
\{p_3^2 + p_1 p_4 t_1 + p_4 t_2 +a_{11} p_1^2  t_2=0\}.$
It is easy to check that the singular locus of this is $\Sigma_{\mA}(4)^o\cap \{p_1=p_3=p_4=s_{23}=t_2=0\}$.

By the equations of $T$ of weight $2,3,4,5$ and the expressions of $q_2$ and $s_{22}$ on the $p_2$-chart as in the subsection \ref{Charts},
we see that
$\Sigma_{\mA}(5)^o=\{a_2 p_4  + b_2 t_1 +p_1 q_1 +b_3 p_3 t_2=0,\,
p_3^2  + p_1 p_4 t_1 + p_4 t_2 +a_{11}  p_1^2  t_2=0\}.$
By the generality of the equations of $T$ of weight $2,3,4,5$ and Bertini's singularity theorem,
$\Sigma(5)_{\mA}^o$ is singular along the restriction of $\Sing \Sigma(4)_{\mA}^o$ and possibly along the restriction of 
$\Bs |\sO(5)|_{\mA}$.
Therefore we have only to check the singularlity of 
$\Sigma(5)_{\mA}^o$ along  
$\Bs |\sO(5)|_{\mA}=\{s_{23}=p_4=t_1=0\}$.
Using the above two equations, 
we see that 
this is contained in 
the restriction of $\Sing \Sigma(4)_{\mA}^o$.
Therefore 
$\Sing \Sigma(5)_{\mA}^o$ is equal to $\Sigma(5)_{\mA}^o\cap \{p_1=p_3=p_4=s_{23}=t_2=0\}$.

By similar considerations to 
those for $\Sigma(5)_{\mA}^o$,
we see 
$\Sing \Sigma(6)_{\mA}^o$ is equal to $\Sigma(6)_{\mA}^o\cap \{p_1=p_3=p_4=s_{23}=t_2=0\}$.

Since $w(p_2)=7$, the $p_2$-chart is disjoint from $\Bs |\sO(7)|_{\mA}$.
Therefore 
$\Sing \Sigma(7)_{\mA}^o$ is equal to $\Sigma(7)_{\mA}^o\cap \{p_1=p_3=p_4=s_{23}=t_2=0\}$,
which turns out to be empty on the $p_2$-chart. Consequently, 
$T_{\mA}^o=\Sigma(7)_{\mA}^o$ is smooth on
the $p_2$-chart, hence $T$ is quasi-smooth on this chart.

\vspace{3pt}

Finally, we show that $T$ is smooth on this chart. 
Since $w(p_1)=7$, the cyclic group $\mZ_7$ acts on $\mA_T^{p_1}$.
By (\ref{eq:393X}), we see that the action on $\mA_T^{p_1}$ is not free only at the $p_2$-point. This point is, however, not contained in $T$, thus $T$ is smooth on the $p_2$-chart.

\vspace{5pt}


\noindent \fbox{\bf No.569}\,

\vspace{3pt}
\noindent {\bf (569i: Equation of $T$)}
{\small{
\begin{center}
$\begin{array}{|c|c|c|c|c|c|c|c|}
\hline
\text{weight} & 1 & 2 & 3  & 4 & 5 &  6  \\
\hline
\text{equation} & s_{23}=0 & s_{13}=0 & p_4=a_4 s_{12} &   q_2=a_2 t_3 & q_1=a_1 p_1+b_1 t_2 & 
r=a_0 p_2+b_0 s_{12}^2\\
& & s_{22}=0& q_3=a_3 s_{12} & s_{11}=a_{11}t_3 & &  t_1=c_1 p_2+d_1 s_{12}^2\\
\hline
\end{array}$
\end{center}
}}
In this table, the ten parameters $a_4,\dots,d_1\in \mC$ are chosen generally.
Note that
\begin{equation}
\label{eq:569X}
T\subset \mP(s_{12}, t_3,p_1,t_2,p_2,p_3,u)=\mP(3,4,5^2,6,7,9).
\end{equation}

\vspace{3pt}

\noindent {\bf (569ii: Base loci)}\quad
\begin{center}
$\begin{array}{|c|c|c|}
\hline
\{p_1=p_2=0\}_{|T} & \Bs |\sO(5)|_{|T} & \Bs |\sO(6)|_{|T} \\
\hline
\text{$u$-point}, (\ref{eq:569Bs1})     & \text{$u$-point},  \sf{p}_3, \sf{q}_3 &\text{$u$-point},  \sf{p}_5 \\
\hline
\end{array}$
\end{center}
\begin{itemize}
\item One point with $u\not=0$ constitutes the locus
\begin{align}
\label{eq:569Bs1}
&\{s_{12}=1,\ \ a_3^2 d_1 - a_2 b_1 t_3^3=0,
 a_3^3 - a_3 b_0^2+(2 a_2^2 b_1 - a_{11} a_2^2) a_3 t_3^3 - a_3 a_4 u=0,\\
&\qquad t_2= (a_2 t_3^2) /a_3\}
\subset \mP(s_{12},t_3,t_2,u).\nonumber
\end{align}
\item
The two points $\sf{p}_3,\sf{q}_3$ constitute the locus
\begin{align*}
\label{eq:569index 3}
&\{s_{12}=1,\ \
u=  (a_3^2 - (b_0  + a_0 p_2)^2)/a_4,\\
&\qquad a_3 a_4 d_1 + (a_3 + b_0 + a_3 a_4 c_1) p_2 + a_0 p_2^2=0\}
 \subset \mP(s_{12},p_2,u)\nonumber.
\end{align*}
\item
${\sf{p}_5}:=\{a_1 p_1+b_1 t_2=0\}\in \mP(p_1,t_2).$
\end{itemize}

\vspace{3pt}

\noindent {\bf (569iii: LPC)}\quad 
\begin{center}
$\begin{array}{|c|c|c|}
\hline
\text{$u$-point} & \sf{p}_5 & \sf{p}_3, \sf{q}_3\\
\hline
1/9(4,5) & 1/5(2,3) & 1/3(1,2)\\
\hline
\end{array}$
\end{center}
As for LPC at the points $\sf{p}_3,\sf{q}_3$, it is very similar to LPC at the two points in the $p_2$-chart of $T$ in the case of No.360.

\vspace{3pt}

\noindent \noindent {\bf (569iv: $\{p_1=p_2=0\}$)}
By {\bf (569ii)} and {\bf (569iii)}, to check that $T_{\mA}^o$ is smooth along the locus $\{p_1=p_2=0\}$, we have only to show that
the $u$-chart of  $T_{\mA}^o$ is smooth at the points in (\ref{eq:569Bs1}).
We see immediately see that $\Sigma_{\mA}(2)^o$ is smooth on the $u$-chart by the equations of $T$ and the coordinate descriptions of the $u$-chart of $\SigmaA$ as in the subsection \ref{Charts}.
Note that none of the coordinates $s_{12},t_3,t_2,u$ of these points  in (\ref{eq:569Bs1}) are zero and the weights of them are $3,4,5,9$, respectively.
Therefore these points are not contained in $\Bs |\sO_{\mP_{\Sigma}}(i)|$ with $i=3,4,5,6$.
Hence $T_{\mA}^o$ is smooth at the points  in (\ref{eq:569Bs1}).

Finally, we show that $T$ is singular only at the $u$-point along the locus $\{p_1=p_2=0\}$. 
Since $w(u)=9$, the cyclic group $\mZ_9$ acts on ${\mA}_T^u$.
By (\ref{eq:569X}), it is easy to see that the action on $\mA_T^u$ is not free only at the $u$-point.

\vspace{5pt}

\noindent \fbox{\bf No.574}\,

\vspace{3pt}

\noindent {\bf (574i: Equation of $T$)}
{\small{
\begin{center}
$\begin{array}{|c|c|c|c|c|c|c|}
\hline
\text{weight} & 1 & 2 & 3  & 4 & 5 &  6  \\
\hline
 & s_{23}=0 & s_{13}=0 & q_3=a_3 t_2 &   q_2=a_2 t_1 & q_1=a_1 p_1+b_1 p_4 & 
r=a_0 p_2+b_0 t_2^2\\
\text{equation}& & s_{22}=0& s_{12}=a_{12} t_2 & s_{11}=a_{11}t_1 & &  \\
& & t_3=0 & & & & \\
\hline
\end{array}$
\end{center}
}}
In this table, the parameters $a_3,\dots,b_0\in \mC$ are chosen generally.

Note that 
\[
T \subset \mP(t_2,t_1,p_1,p_4, p_2, p_3,u)=\mP(3,4,5^2,6,7^2).
\]

\vspace{3pt}

\noindent {\bf (574ii: Base loci)}\quad
\begin{center}
$\begin{array}{|c|c|c|}
\hline
\{p_1=p_2=0\}_{|T} & \Bs |\sO(5)|_{|T} & \Bs |\sO(6)|_{|T} \\
\hline
\text{$u$-point}, \text{$p_4$-point}    & \text{$u$-point},  \sf{p}_3 &\text{$u$-point}, \text{$p_4$-point},\sf{p}_5 \\
\hline
\end{array}$
\end{center}
\begin{align*}
&{\sf{p}_3}:=\{a_0p_2+(b_0 +a_{12} a_3) t_2^2=0\}\in  \mP(p_2,t_2).\\
&{\sf{p}_5}:=\{a_1 p_1+b_1 p_4=0\}\in \mP(p_1,p_4).
\end{align*}

\noindent {\bf (574iii: LPC)}\quad 
\begin{center}
$\begin{array}{|c|c|c|c|}
\hline
\text{$u$-point} & \text{$p_4$-point} & \sf{p}_5 & \sf{p}_3\\
\hline
1/7(1,4) & 1/5(1,4) & 1/5(2,3) & 1/3(1,2)\\
\hline
\end{array}$
\end{center}
\vspace{5pt}


\noindent \fbox{\bf No.642}\,

\vspace{3pt}

\noindent {\bf (642i: Equation of $T$)}
{\small{
\begin{center}
$\begin{array}{|c|c|c|c|c|c|c|}
\hline
\text{weight} & 1 & 2 & 3  & 4 & 5 &  6  \\
\hline
\text{equation} & s_{23}=0 & s_{13}=0 & p_4=a_4 t_3 &   s_{11}=a_{11} p_1+b_{11}t_2 & t_1=a_1 p_2 & 
q_1=b_1 p_3+c_1 t_3^2\\
& & s_{22}=0& s_{12}=a_{12} t_3 & q_3=a_3p_1+b_3 t_2 &q_2=a_2 p_2 &  \\
\hline
\end{array}$
\end{center}
}}
In this table, the parameters $a_4,\dots,c_1\in \mC$ are chosen generally. Note that
\begin{equation}
\label{eq:642X}
T\subset \mP(t_3,p_1,t_2,p_2,p_3,r,u)=\mP(3,4^2,5,6,7,11).
\end{equation}

\vspace{3pt}

\noindent {\bf (642ii: Base loci)}\quad 
\begin{center}
$\begin{array}{|c|c|}
\hline
\{p_1=p_2=0\}_{|T} =\Bs |\sO(4)|_{|T} & \Bs |\sO(6)|_{|T} \\
\hline
\text{$u$-point}, \sf{p}_3    & \text{$u$-point},\sf{p}_4 \\
\hline
\end{array}$
\end{center}
\begin{align*}
&{\sf{p}_3}:=\{p_3+a_4 t_3^2=0\}\in \mP(p_3,t_3).\\
&{\sf{p}_4}:=\{a_{11} p_1 + b_{11} t_2 = 0\}\in\mP(p_1,t_2).
\end{align*}

\vspace{3pt}

\noindent {\bf (642iii: LPC)}\quad
\begin{center}
$\begin{array}{|c|c|c|}
\hline
\text{$u$-point} &  \sf{p}_4 & \sf{p}_3\\
\hline
1/11(4,7) & 1/4(1,3) & 1/3(1,2)\\
\hline
\end{array}$
\end{center}
\vspace{3pt}

\noindent {\bf (642iv: $p_1$-chart)}
The smoothness of $T_{\mA}^o$ on the $p_1$-chart follows from the Jacobian criterion and LPC at the point $\sf{p}_4$.

To check the singularities of $T$ on the $p_1$-chart, we note that $w(p_1)=4$.
Therefore the cyclic group $\mZ_4$ acts on $\mA_T^{p_1}$.
We compute the stabilizer groups of points of $\mA_T^{p_1}$ noting (\ref{eq:642X}). Then we see that the singular locus of $T$ on the $p_1$-chart consists of the point $\sf{p}_4$ and
\[
\{p_1=1,
t_2 = -(a_3 + b_1)/b_3, -(-a_3 b_{11} + a_{11} b_3 - b_{11} b_1)/b_3  =  p_3^2\},
\]
which consists of one point $\sf{p}_2$. By LPC, we have already determined the singularity of $T$ at the point ${\sf p}_4$.
We wiil determine the singularity of $T$ at the point ${\sf p}_2$ more directly as follows.
On the affine space $\mA_T^{p_1}$, there are two points corresponding to ${\sf p}_2$. 
Since $p_3\not =0$ at these two points, we see that the stabilizer groups of the $\mZ_4$-action at them are isomorphic to $\mZ_2$. Therefore
$T$ has a $1/2(1,1)$-singularity at $\sf{p}_2$. 
\vspace{5pt}


\noindent \fbox{\bf No.644}

\vspace{3pt}

\noindent {\bf (644i: Equation of $T$)}
{\small{
\begin{center}
$\begin{array}{|c|c|c|c|c|c|c|}
\hline
\text{weight} & 1 & 2 & 3  & 4 & 5 &  6  \\
\hline
 & s_{23}=0 & s_{13}=0 & s_{12}=a_{12} t_2 &  p_4=a_4 p_1+b_4 t_1 & q_2=a_2 p_2 & 
q_1=a_1 p_3+b_1 t_2^2\\
\text{equation} & & s_{22}=0& & q_3=a_3 p_1+b_3 t_1 & &  \\
& & t_3=0 & & s_{11}=a_{11}p_1+b_{11}t_1 & & \\
\hline
\end{array}$
\end{center}
}}
In this table, the parameters $a_{12},\dots,b_1\in \mC$ are chosen generally. Note that 
\[
T\subset \mP(t_2,p_1,t_1,p_2,p_3,r,u)=\mP(3,4^2,5,6,7,10)
\]

\vspace{3pt}

\noindent {\bf (644ii: Base loci)}\quad
\begin{center}
$\begin{array}{|c|c|}
\hline
\{p_1=p_2=0\}_{|T} =\Bs |\sO(4)|_{|T} & \Bs |\sO(6)|_{|T} \\
\hline
\text{$u$-point}    & \text{$u$-point},\sf{p}_4,\sf{q}_4 \\
\hline
\end{array}$
\end{center}
The two points $\sf{p}_4,\sf{q}_4$ constitute the locus
\begin{equation*}
\label{eq:644Bs6}
\{a_{11} p_1^2+(a_4+b_{11}) p_1 t_1+ b_4 t_1^2=0\}\subset \mP(p_1,t_1).
\end{equation*}

\vspace{3pt}

\noindent {\bf (644iii: LPC)}\quad 
\begin{center}
$\begin{array}{|c|c|}
\hline
\text{$u$-point} & \sf{p}_4, \sf{q}_4\\
\hline
1/10(3,7) &1/4(1,3) \\
\hline
\end{array}$
\end{center}
\vspace{3pt}
\noindent {\bf (644iv: $p_1$-chart)}
We may determine the singularities of $T$ on this chart in a very similar way to that for No.642.
The singular locus consists of the points ${\sf p}_4$, ${\sf q}_4$, and one point $\sf{p}_2$, which constitutes
\begin{align*}
&\{p_1=1, t_1 = -(a_1 + a_3)/b_3,\\
&p_3^2=-\frac{a_1^2 b_4 + 2 a_1 a_3 b_4 + a_3^2 b_4 - a_1 a_4 b_3 - a_1 b_{11} b_3 - a_4 a_3 b_3 - 
 b_{11} a_3 b_3 + a_{11} b_3^2}{b_3^2}\}.
\end{align*}
We see that $T$ has a $1/2(1,1)$-singularity at ${\sf p}_2$ noting $p_3\not =0$.
\vspace{5pt}

\noindent \fbox{\bf No.1091}\,

\vspace{3pt}

\noindent {\bf (1091i: Equation of $T$)}
{\small{
\begin{center}
$\begin{array}{|c|c|c|c|c|c|c|}
\hline
\text{weight} & 1 & 2 & 3  & 4 & 5 &  6  \\
\hline
 & s_{23}=0 & s_{22}=a_{22}s_{13} & s_{12}=0 &  q_3=a_3 s_{13}^2& p_4=a_4 t_1 & 
q_1=a_1 p_1+b_1 s_{13}^2\\
\text{equation} & & & t_3=0& s_{11}=a_{11} s_{13}^2 &q_2=b_2 t_1 &  \\
& &  & & t_2=a_2 s_{13}^2 & & \\
\hline
\end{array}$
\end{center}
}}
In this table, the parameters $a_{22},\dots,b_1\in \mC$ are chosen generally.
Note that
\[
T\subset \mP(s_{13},t_1,p_1,r, p_2,p_3,u)=\mP(2,5,6,7^2,8,9).
\]

\noindent {\bf (1091ii: Base loci)}:
\begin{center}
\[
\begin{array}{|c|c|}
\hline
\{p_1=p_2=0\}_{|T} = \Bs |\sO(i)|_{|T}\cap \{p_1=p_2=0\}\, (2\leq i\leq 5) & \Bs |\sO(6)|_{|T} \\
\hline
\text{$u$-point}  & \text{$u$-point},\text{$p_2$-point} \\
\hline
\end{array}
\]
\end{center}
\vspace{3pt}

\noindent {\bf (1091iii:LPC)}\quad
\begin{center}
$\begin{array}{|c|c|c|c|}
\hline
\text{$u$-point} & \text{$p_2$-point}\\
\hline
1/9(2,7) & 1/7(1,6)\\
\hline
\end{array}$
\end{center}
\vspace{3pt}

\noindent {\bf (1091iv: $p_1$-chart)}
We may determine the singularities of $T$ on this chart in a very similar way to that for No.642.
The singular locus is
\begin{align*}
\{s_{13}=1, \ \ (b_1^2 - 2 b_1 a_3 + a_{11} a_3^2) + 2( a_1 b_1 -  a_1 a_3) p_1 + a_1^2 p_1^2=0,\\
p_3= -(1/a_3) (b_1 p_1 + a_1 p_1^2)\},
\end{align*}
which consists of two points $\sf{p}_2,\sf{q}_2$.
We see that $T$ has $1/2(1,1)$-singularities at ${\sf p}_2$ and ${\sf q}_2$ noting $s_{13}\not =0$. 

\vspace{5pt}

\noindent \fbox{\bf No.1181}\,

\vspace{3pt}

\noindent {\bf (1181i: Equation of $T$)}
{\small{
\begin{center}
$\begin{array}{|c|c|c|c|c|c|c|}
\hline
\text{weight} & 1 & 2 & 3  & 4 & 5 &  6  \\
\hline
 & s_{23}=0 & p_4=a_4 s_{13} & t_3=a_3 p_1 &  s_{11}=a_{11}p_2& q_2=c_2 p_3+d_2 t_1 &  
q_1=a_1 p_2 s_{13}\\
\text{equation} & &s_{22}=a_{22} s_{13} & s_{12}=a_{12} p_1& q_3=b_3 p_2+c_3 s_{13}^2 & \quad \quad \quad +e_2 s_{13}p_1& \quad +b_{1} s_{13}^2 \\
& &  & & t_2=a_2 p_2+b_2 s_{13}^2 & & \\
\hline
\end{array}$
\end{center}
}}
In the table, the parameters $a_4,\dots,b_1\in \mC$ are chosen generally.
Note that 
\[
T \subset \mP(s_{13},p_1,p_2,p_3,t_1,r,u)=\mP(2,3,4,5^2,7,12).
\]

\vspace{3pt}

\noindent {\bf (1181ii: Base loci)}
\begin{center}
\[
\begin{array}{|c|c|c|}
\hline
\{p_1=p_2=0\}_{|T} = \Bs |\sO(i)|_{|T}\cap \{p_1=p_2=0\}\, (i=2,3)  &  \Bs |\sO(5)|_{|T} & \Bs |\sO(6)|_{|T} \\
\,=\Bs |\sO(4)|_{|T}  \quad \quad \quad \quad \quad & &\\
\hline
\text{$u$-point}   & \text{$u$-point},\sf{p}_4,\sf{p}_2 & \text{$u$-point},\sf{p}_4 \\
\hline
\end{array}
\]
\end{center}
\begin{itemize}
\item
${\sf{p}_4}:=\{u=a_{11} a_2 b_3^2 p_2^3\}\in \mP(p_2,u).$ 
\item The point $\sf{p}_2$ satisfies $p_2\not =0$ and $s_{13}\not =0$. Detailed coordinates of ${\sf p}_2$ are complicated so are omitted.
\end{itemize}
\vspace{3pt}

\noindent {\bf (1181iii: LPC)}\quad
\begin{center}
$\begin{array}{|c|c|c|c|}
\hline
\text{$u$-point} & \sf{p}_4\\
\hline
1/12(5,7) & 1/4(1,3) \\
\hline
\end{array}$
\end{center}
\vspace{3pt}

\noindent {\bf (1181iv: $p_2$-chart)}
We may determine the singularities of $T$ on this chart in a very similar way to that for the $p_1$-chart of $T$ of No.642.
The singular locus consists of 
the points $\sf{p}_4$ and $\sf{p}_2$. The singularity of $T$ at ${\sf p}_4$ is determined by LPC, and 
$T$ has a $1/2 (1,1)$-singularity at $\sf{p}_2$ since $s_{13}\not=0$.

\vspace{5pt}

\noindent \fbox{\bf No.1185}\,

\vspace{3pt}

\noindent {\bf (1185i: Equation of $T$)}
{\small{
\begin{center}
$\begin{array}{|c|c|c|c|c|c|c|}
\hline
\text{weight} & 1 & 2 & 3  & 4 & 5 &  6  \\
\hline
 & s_{23}=0 & p_4=a_4 s_{13} & q_2=a_2 s_{12} &  q_1=a_1 p_1+b_1 s_{13}^2& r=a_0 p_2+b_0 t_2 &  
t_1=c_1 p_3 +d_1 s_{12}^2\\
\text{equation} & &q_3=a_3 s_{13} && s_{11}=a_{11} p_1+b_{11} s_{13}^2 & \quad \quad \quad +c_0 s_{13} s_{12}& \quad+e_1 s_{13}^2+f_1 p_1 s_{13} \\
& & s_{22} = a_{22} s_{13} & & t_3=c_3 p_1+d_3 s_{13}^2 & & \\
\hline
\end{array}$
\end{center}
}}
In the table, the parameters $a_4,\dots,f_1\in \mC$ are chosen generally.
Note that 
\[
T\subset \mP(s_{13}, s_{12}, p_1,p_2,t_2,p_3, u)
=\mP(2,3,4,5^2,6,8).
\]

\vspace{3pt}

\noindent {\bf (1185ii: Base loci)}
\begin{center}
\small{
\[
\begin{array}{|c|c|c|c|}
\hline
\{p_1=p_2=0\}_{|T} & \Bs |\sO(i)|_{|T}\cap \{p_1=p_2=0\}\, (i=2,3)  &  \Bs |\sO(4)|_{|T} =  \Bs |\sO(6)|_{|T} & \Bs |\sO(5)|_{|T} \\
\hline
\text{$u$-point}, {\sf v} & \text{$u$-point}&    \text{$u$-point},\sf{p}_5&  \text{$u$-point},\sf{p}_2,\sf{q}_2,\sf{r}_2 \\
\hline
\end{array}
\]
}
\end{center}
\begin{itemize}
\item
${\sf{p}_5}:=\{a_0 p_2+b_0 t_2=0\}\in \mP(p_2,t_2).$
\item The three points $\sf{p}_2,\sf{q}_2,\sf{r}_2$ satisfy $p_1\not =0$ and $s_{13}\not =0$. Detailed coordinates of them are complicated so are omitted.
\item
The point ${\sf v}$ satisfies $s_{13}\not = 0, s_{12}\not = 0, t_2\not = 0$, and $u\not =0$. Here we show that $T$ is smooth at ${\sf v}$.  
We see that the $u$-chart of $\Sigma(1)_{\mA}^o$ is smooth by the coordinate descriptions. 
Since $w(u)=8$, $w(t_2)=5$ and $t_2\not =0$ for ${\sf v}$, the $\mZ_8$-action on $\mA_{\Sigma(1)}^u$ is free along the inverse image of ${\sf v}$. Therefore $\Sigma(1)$ is smooth at ${\sf v}$.
Since ${\sf v}$ satisfies $s_{13}\not = 0, s_{12}\not = 0$, and $t_2\not = 0$, we see that ${\sf v}\not \in \Bs |\sO(i)|$ $(i=2,3,4,5,6)$. Therefore $T$ is smooth at ${\sf v}$.
\end{itemize}
\vspace{3pt}

\noindent {\bf (1185iii: LPC)}\quad 
\begin{center}
$\begin{array}{|c|c|c|c|}
\hline
\text{$u$-point} &  \sf{p}_5 \\
\hline
1/8(3,5) & 1/5(1,4)\\
\hline
\end{array}$
\end{center}
\vspace{3pt}

\noindent {\bf (1185iv: $p_1$-chart)}
We may determine the singularities of $T$ on this chart in a very similar way to that for No.642.
The singular locus consists of $\sf{p}_2,\sf{q}_2,\sf{r}_2$. We see that $T$ has $1/2(1,1)$-singularities at them noting $s_{13}\not =0$.

\vspace{5pt}

\noindent \fbox{\bf No.1186}\,

\vspace{3pt}

\noindent {\bf (1186i: Equation of $T$)}
{\small{
\begin{center}
$\begin{array}{|c|c|c|c|c|c|}
\hline
\text{weight} & 1 & 2 & 3  & 4 & 5  \\
\hline
 & s_{23}=0 & q_3=a_3 s_{13} & p_4=a_4 t_3 &  q_1=a_1 p_1+b_1 s_{13}^2& r=a_0 p_2+b_0 t_1  \\
\text{equation} & &s_{22}=a_{22} s_{13} &q_2 =a_2 t_3& s_{11}=a_{11} p_1+b_{11} s_{13}^2 & \quad \quad \quad +c_0 s_{13} t_3 \\
& &  & s_{12}=a_{12} t_3& t_2=b_2 p_1+c_2 s_{13}^2 & \\
\hline
\end{array}$
\end{center}
}}
In the table, the parameters $a_3,\dots,c_0\in \mC$ are chosen generally.
Note that 
\[
T\subset \mP(s_{13}, t_3,p_1,p_2,t_1,p_3,u)=\mP(2,3,4,5^2,6,7).
\]

\vspace{3pt}

\noindent {\bf (1186ii: Base loci)}
\begin{center}
{\small{
\[
\begin{array}{|c|c|c|c|c|}
\hline
\{p_1=p_2=0\}_{|T} & \Bs |\sO(2)|_{|T}\cap \{p_1=p_2=0\} & \Bs |\sO(3)|_{|T}\cap \{p_1=p_2=0\}  &  \Bs |\sO(4)|_{|T} &  \Bs |\sO(5)|_{|T} \\
\hline
\text{$u$-point}, \sf{p}_3, {\sf v} & \text{$u$-point}, \sf{p}_3 &    \text{$u$-point}&   \text{$u$-point},\sf{p}_3,\sf{p}_5& \text{$u$-point},\sf{p}_3,\sf{p}_2, \sf{q}_2 \\
\hline
\end{array}
\]
}} 
\end{center}
\begin{itemize}
\item ${\sf{p}_3}:=\{p_3+a_4 t_3^2=0\} \in \mP(t_3,p_3).$
\item ${\sf{p}_5}:=\{a_0 p_2+b_0 t_1=0\}\in \mP(p_2,t_1).$
\item The two points $\sf{p}_2,\sf{q}_2$ satisfy $p_1\not =0$ and $s_{13}\not =0$.
\item
The point ${\sf v}$ satisfies $s_{13}\not = 0, t_1\not = 0, t_3\not = 0$, and $u\not=0$.
We can show that $T$ is smooth at ${\sf v}$ as we have shown the same claim for No.1185.
\end{itemize}
\vspace{3pt}

\noindent {\bf (1186iii: LPC)}\quad
\begin{center}
$\begin{array}{|c|c|c|c|}
\hline
\text{$u$-point} & \sf{p}_5 & \sf{p}_3\\
\hline
1/8(3,5) & 1/5(1,4)  & 1/3(1,2)\\
\hline
\end{array}$
\end{center}
\vspace{3pt}

\noindent {\bf (1186iv: $p_1$-chart)}
We may determine the singularities of $T$ on this chart in a very similar way to that for No.642.
The singular locus consists of $\sf{p}_2,\sf{q}_2$. We see that $T$ has $1/2(1,1)$-singularities at them noting $s_{13}\not =0$. 

\vspace{5pt}

\noindent \fbox{\bf No.1218}\,

\vspace{3pt}

\noindent {\bf (1218i: Equation of $T$)}
{\small{
\begin{center}
$\begin{array}{|c|c|c|c|c|c|}
\hline
\text{weight} & 1 & 2 & 3  & 4 & 5  \\
\hline
 & s_{23}=0 & q_3=a_3 t_2 & q_2=a_2 t_1 &  q_1=a_1 p_1+b_1 t_2^2& p_4=a_4 p_2+b_4 r  \\
\text{equation} &t_3=0 &s_{13}=a_{13} t_2 &s_{12} =a_{12} t_1& s_{11}=a_{11} p_1+b_{11} t_2^2 & \quad +c_4 u+d_4 t_2 t_1 \\
& & s_{22}=a_{22}t_2 & &  & \\
\hline
\end{array}$
\end{center}
}}
In the table, the parameters $a_3,\dots,d_4\in \mC$ are chosen generally. Note that 
\[
T\subset \mP(t_2,t_1,p_1,p_2,r,u,p_3)=\mP(2,3,4,5^3,6).
\]

\vspace{3pt}

\noindent {\bf (1218ii: Base loci)}
\begin{center}
\[
\begin{array}{|c|c|c|}
\hline
\{p_1=p_2=0\}_{|T} = \Bs |\sO(i)|_{|T}\cap \{p_1=p_2=0\}\,(i=2,3) &  \Bs |\sO(4)|_{|T} &  \Bs |\sO(5)|_{|T} \\
\hline
\sf{p}_5, \sf{q}_5, {\sf v}, {\sf w}  & \sf{p}_5, \sf{q}_5,  \text{$p_2$-point}&  \sf{p}_2, \sf{q}_2 \\
\hline
\end{array}
\]
\end{center}
\begin{itemize}
\item The two points $\sf{p}_5,\sf{q}_5$ constitute the locus 
\begin{equation}
\label{eq:p5q5}
\{-r^2-b_4 ru-c_4 u^2=0\}\subset \mP(r,u).
\end{equation}
\item 
The two points ${\sf v}, {\sf w}$ satisfies $t_1\not = 0, t_2\not = 0$, and $u\not =0$. 

Here we show that $T$ is smooth at ${\sf v}$ and ${\sf w}$.  We see that the $u$-chart of $\Sigma(3)_{\mA}$ is smooth by the coordinate descriptions of the $u$-chart of $\Sigma_{\mA}$ as in the subsection \ref{Charts} and {\bf (1218i)}. 
Since $w(u)=5$, $w(t_2)=2$ and $t_2\not =0$ for ${\sf v}$ and ${\sf w}$, the $\mZ_5$-action on $\mA_{\Sigma(3)}^u$ is free along the inverse images of ${\sf v}$ and ${\sf w}$. Therefore $\Sigma(3)$ is smooth at ${\sf v}$ and ${\sf w}$.
Since ${\sf v}$ and ${\sf w}$ satisfies $t_1\not = 0, t_2\not = 0$, and $w(t_1)=3$, $w(t_2)=2$, we see that ${\sf v}, {\sf w}\not \in \Bs |\sO(i)|$ $(i=4,5)$. Therefore $T$ is smooth at ${\sf v}$ and ${\sf w}$.
\item
 The two points $\sf{p}_2$, $\sf{q}_2$ in 
$\mP(t_2,p_1,p_3)$ satisfies $t_2\not =0$.
\end{itemize}
\vspace{3pt}

\noindent {\bf (1218iii: LPC)}\quad
\begin{center}
$\begin{array}{|c|c|c|c|}
\hline
\text{$p_2$-point}  & \sf{p}_5, \sf{q}_5 \\
\hline
1/5(1,4) & 1/5(2,3) \\
\hline
\end{array}$
\end{center}
\vspace{3pt}

\noindent {\bf (1218iv: $p_1$-chart)}
We may determine the singularities of $T$ on this chart in a very similar way to that for No.642.
The singular locus consists of $\sf{p}_2,\sf{q}_2$. We see that $T$ has $1/2(1,1)$-singularities at them noting 
$t_2\not=0$.

\vspace{5pt}


\noindent \fbox{\bf No.1253}\,

\vspace{3pt}

\noindent {\bf (1253i: Equation of $T$)}
{\small{
\begin{center}
$\begin{array}{|c|c|c|c|c|}
\hline
\text{weight} & 1 & 2 & 3  & 4  \\
\hline
 & s_{13}=0 & s_{11}=a_{11} t_3 & p_4=a_4 t_2 &  q_1=b_1 p_1+c_1 p_2+d_1 t_3^2  \\
\text{equation} &s_{23}=0 &s_{12}=a_{12} t_3 &q_3=a_3 t_2 & q_2=a_2 p_1+b_2 p_2+c_2 t_3^2  \\
& & s_{22}=a_{22}t_3 & t_1=a_1 t_2 &  \\
\hline
\end{array}$
\end{center}
}}
In this table, the parameters $a_{11},\dots,c_2\in \mC$ are chosen generally.
Note that 
\[
T\subset \mP(t_3,t_2,p_1,p_2,p_3,r,u)=\mP(2,3,4^2,5^2,7).
\]

\vspace{3pt}

\noindent {\bf (1253ii: Base loci)}
\begin{center}
\[
\begin{array}{|c|}
\hline
\{p_1=p_2=0\}_{|T} = \Bs |\sO(i)|_{|T}\cap \{p_1=p_2=0\}\, (i=2,3)   \\
\,=  \Bs |\sO(4)|_{|T} \quad \quad \quad \quad \quad \\
\hline
 \text{$u$-point}     \\
\hline
\end{array}
\]
\end{center}
\vspace{3pt}

\noindent {\bf (1253iii: LPC)}\quad
\begin{center}
$\begin{array}{|c|}
\hline
\text{$u$-point}\\
\hline
1/7(2,5) \\
\hline
\end{array}$
\end{center}

\vspace{3pt}

\noindent {\bf (1253iv: $p_1$-chart)}
We may determine the singularities of $T$ on this chart in a very similar way to that for No.642.
The singular locus consists of the following:

\begin{itemize}
\item 
$\{b_1p_1^2 + (a_2+c_1) p_1 p_2 + b_2 p_2^2 = 0\}\subset \mP(p_1,p_2)$,
which consists of two points $\sf{p}_4,\sf{q}_4$.

By LPC at these two points, we see that $T$ has $1/4(1,3)$-singularities there;
\item
two points $\sf{p}_2,\sf{q}_2$ (the coordinates are complicated so are omitted). 

We see that $T$ has $1/2(1,1)$-singularities at them noting 
$t_3\not =0$. 
\end{itemize}

\vspace{5pt}

\noindent \fbox{\bf No.1413}\,

\vspace{3pt}

\noindent {\bf (1413i: Equation of $T$)}
{\small{
\begin{center}
$\begin{array}{|c|c|c|c|c|}
\hline
\text{weight} & 1 & 2 & 3  & 4  \\
\hline
\text{equation} & s_{13}=0 & q_3=a_3 t_3 & p_4=a_4 t_1+b_4 t_2 &  r=a_0 p_1+b_0 p_2+c_0 t_3^2  \\
 &s_{23}=0 &s_{11}=a_{11} t_3 &q_1=a_1 t_1+b_1 t_2 &   \\
& & s_{12}=a_{12}t_3 & q_2=a_2 t_1+b_2 t_2 &  \\
& & s_{22}=a_{22} t_3 & &\\
\hline
\end{array}$
\end{center}
}}
In this table, the parameters $a_3,\dots, c_0\in \mC$ are chosen generally. Note that 
\[
T\subset \mP(t_3,t_1,t_2,p_1,p_2,p_3,u)=\mP(2,3^2,4^2,5^2).
\]

\vspace{3pt}

\noindent {\bf (1413ii: Base loci)}
\begin{center}
\[
\begin{array}{|c|c|c|}
\hline
\{p_1=p_2=0\}_{|T} = \Bs |\sO(2)|_{|T}\cap \{p_1=p_2=0\} & \Bs |\sO(3)|_{|T} &  \Bs |\sO(4)|_{|T} \\
\hline
\text{$u$-point}, \sf{p}_3, \sf{q}_3  & \text{$u$-point}, \sf{p}_4, \sf{p}_2, \sf{q}_2  &   \text{$u$-point},  \sf{p}_3, \sf{q}_3 \\
\hline
\end{array}
\]
\end{center} 
\begin{itemize}
\item The two points $\sf{p}_3,\sf{q}_3$ constitute the locus
\begin{equation*}
\label{eq:1413index 3}
\{a_2 t_1^2 - a_1 t_1 t_2 + b_2 t_1 t_2 - b_1 t_2^2\}\subset \mP(t_1,t_2).
\end{equation*}
\item 
${\sf{p}_4}:=\{a_0 p_1+b_0 p_2=0\}\in \mP(p_1,p_2).$
\item
The two points $\sf{p}_2,\sf{q}_2$ satisfy $p_1\not =0$ and $t_3\not =0$.
\end{itemize}
\vspace{3pt}

\noindent {\bf (1413iii: LPC)}\quad
\begin{center}
$\begin{array}{|c|c|c|c|}
\hline
\text{$u$-point}  & \sf{p}_4 & \sf{p}_3,\sf{q}_3\\
\hline
1/5(2,3) & 1/4(1,3) & 1/3(1,2)\\
\hline
\end{array}$
\end{center} 
\vspace{3pt}

\noindent {\bf (1413iv: $p_1$-chart)}
We may determine the singularities of $T$ on this chart in a very similar way to that for No.642.
The singular locus consists of $\sf{p}_2,\sf{q}_2$. We see that $T$ has $1/2(1,1)$-singularities at them noting $t_3\not =0$.

\vspace{5pt}

\noindent \fbox{\bf No.2422}\,

\vspace{3pt}

\noindent {\bf (2422i: Equation of $T$)} 
{\small{
\begin{center}
$\begin{array}{|c|c|c|c|c|}
\hline
\text{weight} & 1 & 2 & 3  & 4  \\
\hline
\text{equation} & t_3=0 & q_3=a_3 p_1+b_3 t_2 & p_4=a_4 p_2+b_4 t_1 &  q_1=a_1 p_3+b_1 p_1^2+c_1p_1t_2+d_1 t_2^2  \\
& &s_{13}=a_{13} p_1+b_{13} t_2 &q_2=a_2 p_2+b_2 t_1 &  s_{11}=a_{11} p_3+b_{11} p_1^2+c_{11}p_1t_2+d_{11} t_2^2\\
& & s_{22}=a_{22}p_1+b_{22} t_2& s_{12}=a_{12} p_2+b_{12} t_1 &  \\
\hline
\end{array}$
\end{center}
}}
In this table, the parameters $a_3,\dots, d_{11}\in \mC$ are chosen generally. Note that
\[
T\subset \mP(p_1,t_2,p_2,t_1,p_3,r,u)=\mP(2^2,3^2,4,5,7).
\]

\vspace{3pt}

\noindent {\bf (2422ii: Base loci)}\quad
\begin{center}
$\begin{array}{|c|c|c|}
\hline
\{p_1=p_2=0\}_{|T} & \Bs |\sO(2)|_{|T} &  \Bs |\sO(3)|_{|T} \\
\hline
\text{$u$-point} &  \text{$u$-point}, \sf{p}_3  &   \text{$u$-point},  \sf{p}_2, \sf{q}_2, \sf{r}_2,\sf{s}_2 \\
\hline
\end{array}$ 
\end{center}
\begin{itemize}
\item
${\sf{p}_3}:=\{a_2 p_2+b_2 t_1=0\}\in \mP(p_2,t_1).$
\item
The four points $\sf{p}_2,\sf{q}_2,\sf{r}_2,\sf{s}_2$ are contained in the $p_1$-chart.
\end{itemize}
\vspace{3pt}

\noindent {\bf (2422iii: LPC)}\quad
\begin{center}
$\begin{array}{|c|c|}
\hline
\text{$u$-point} &  \sf{p}_3\\
\hline
1/7(2,5) & 1/3(1,2)\\
\hline
\end{array}$
\end{center}
\vspace{3pt}

\noindent {\bf (2422iv: $p_1$-chart)}
We may determine the singularities of $T$ on this chart in a very similar way to that for No.642.
The singular locus consists of $\sf{p}_2,\sf{q}_2,{\sf r}_2,{\sf s}_2$. We see that $T$ has $1/2(1,1)$-singularities at them noting $w(p_1)=2$.
\vspace{5pt}


\subsection{Case $h^0(\sO_{\mP_X}(1))=2$}~
\label{Case2}

$\empty$

\vspace{5pt}

\noindent \fbox{\bf No.4850}\,

\vspace{3pt}

\noindent {\bf (4850i: Equation of $C$)}
{\small{
\begin{center}
$\begin{array}{|c|c|c|c|c|c|c|c|}
\hline
\text{weight} & 1 & 2 & 3  & 4 & 5 &  6 &  7 \\
\hline
\text{equation} & p_4=0 & s_{13}=0 & s_{12}=0 & s_{11}=a_{11} p_1 & q_2=a_2 p_2 & q_1=a_1 p_3+b_1 t_2 & 
r=a_0 t_1\\
&s_{23}=0 & s_{22}=0 & & q_3=a_3 p_1 &t_3=b_3 p_2 & & \\
\hline
\end{array}$
\end{center}
}}
In this table, the parameters $a_{11},\dots,a_0\in \mC$ are chosen generally.
Note that \[
C\subset \mP(p_1,p_2,p_3,t_2,t_1,u)=\mP(4,5,6^2,7,13).
\]

\vspace{3pt}

\noindent {\bf (4850ii: Base loci)}\quad 
\begin{center}
$\begin{array}{|c|c|}
\hline
\{p_1=0\}_{|C} =\Bs |\sO(i)|_{|C}\, (i=4,6) &  \Bs |\sO(i)|_{|C}\, (i=5,7) \\
\hline
\text{$u$-point}      &  \text{$u$-point}, \sf{p}_2\\
\hline
\end{array}$ 
\end{center}
\[
{\sf{p}_2}:=\{p_1=1, t_2=-1/b_1 (a_1+a_3)p_3, a_{11} +p_3^2=0\}\subset \mP(p_1,p_3,t_2).
\]
\vspace{3pt}

\noindent {\bf (4850iii: LPC)}\quad
\begin{center}
$\begin{array}{|c|}
\hline
\text{$u$-point} \\
\hline
1/13(6,7)  \\
\hline
\end{array}$
\end{center}
\vspace{3pt}

\noindent {\bf (4850iv: $p_1$-chart)}
We may determine the singularities of $T$ on this chart in a very similar way to that for No.642.
The singular locus consists of $\sf{p}_2$. We see that $T$ has a $1/2(1,1)$-singularity at ${\sf p}_2$ noting $p_3\not=0$.

\vspace{5pt}

\noindent \fbox{\bf No.4938}\,

\vspace{3pt}

\noindent {\bf (4938i: Equation of $C$)}
{\small{
\begin{center}
$\begin{array}{|c|c|c|c|c|c|c|}
\hline
\text{weight} & 1 & 2 & 3  & 4 & 5 &  6 \\
\hline
\text{equation} & p_4=0 & s_{13}=0 & s_{12}=a_{12}p_1 & q_2=a_2 p_2 & q_1=a_1 p_3+b_1 t_2 & r=a_0 t_1+b_0 p_1^2 \\
&s_{23}=0 & s_{22}=0 & q_3=a_3 p_1 &s_{11}=a_{11}p_2 & & \\
& & & &  t_3=b_3 p_2 & &\\
\hline
\end{array}$
\end{center}
}}
In this table, the parameters $a_{12},\dots,b_0\in \mC$ are chosen generally.
Note that 
\[
C\subset \mP(p_1,p_2,p_3,t_2,t_1,u)=\mP(3,4,5^2,6,11).
\]

\vspace{3pt}

\noindent {\bf (4938ii: Base loci)}\quad 
\begin{center}
$\begin{array}{|c|c|}
\hline
\{p_1=0\}_{|C} = \Bs |\sO(i)|_{|C} (i=3,6) &  \Bs |\sO(i)|_{|C} (i=4,5)\\
\hline
\text{$u$-point}  &  \text{$u$-point}, \sf{p}_3 \\
\hline
\end{array}$ 
\end{center}
\[
{\sf{p}_3}:=
\{t_1=(a_{12} a_3-b_0)/a_0\, p_1^2\}\in \mP(p_1,t_1).
\]
\vspace{3pt}

\noindent {\bf (4938iii: LPC)}\quad
\begin{center}
$\begin{array}{|c|c|}
\hline
\text{$u$-point} &  \sf{p}_3\\
\hline
1/11(5,6)  & 1/3(1,2)\\
\hline
\end{array}$
\end{center}
\vspace{5pt}

\noindent \fbox{\bf No.5202}\,

\vspace{3pt}

\noindent {\bf (5202i: Equation of $C$)}
{\small{
\begin{center}
$\begin{array}{|c|c|c|c|c|c|}
\hline
\text{weight} & 1 & 2 & 3  & 4 & 5 \\
\hline
\text{equation} & p_4=0 & q_3=a_3 p_1 & q_2=a_2p_2 & q_1=a_1 p_3+b_1 t_2+c_1 p_1^2 & r=a_0 t_1+b_0 p_1 p_2\\
&s_{23}=0 & s_{22}=a_{22}p_1 & s_{12}=a_{12} p_2 &s_{11}=a_{11}p_3+b_{11}t_2+c_{11} p_1^2 & \\
& & s_{13}=a_{13} p_1& t_3=a_0 p_2&   &\\
\hline
\end{array}$
\end{center}
}}
In this table, the parameters $a_3,\dots,b_0\in \mC$ are chosen generally.
Note that 
\[
C\subset \mP(p_1,p_2,p_3,t_2,t_1,u)=\mP(2,3,4^2,5,9).
\]

\vspace{3pt}

\noindent {\bf (5202ii: Base loci)}\quad
\begin{center}
$\begin{array}{|c|c|}
\hline
\{p_1=0\}_{|C} =
\Bs |\sO(i)|_{|C}\, (i=2,4)&  \Bs |\sO(i)|_{|C}\, (i=3,5) \\
\hline
\text{$u$-point} & \text{$u$-point}, \sf{p}_2, \sf{q}_2 \\
\hline
\end{array}$
\end{center}
The two points $\sf{p}_2,\sf{q}_2$ are contained in the $p_1$-chart.

\vspace{3pt}

\noindent {\bf (5202iii: LPC)}\quad
\begin{center}
$\begin{array}{|c|}
\hline
\text{$u$-point} \\
\hline
1/9(4,5)  \\
\hline
\end{array}$
\end{center}
\vspace{3pt}

\noindent {\bf (5202iv: $p_1$-chart)}
We may determine the singularities of $T$ on this chart in a very similar way to that for No.642.
The singular locus consists of $\sf{p}_2,\sf{q}_2$. We see that $T$ has $1/2(1,1)$-singularities at them noting $w(p_1)=2$.

\vspace{5pt}


\noindent \fbox{\bf No.5859}

\vspace{3pt}

\noindent {\bf (5859i: Equation of $C$)}
{\small{
\begin{center}
$\begin{array}{|c|c|c|c|c|}
\hline
\text{weight} & 1 & 2 & 3  & 4 \\
\hline
\text{equation} & t_3=0 & p_4=a_4 p_2+b_4 t_2 & q_2=a_2p_3+b_2 t_1 & q_1=a_1 p_2^2+b_1 p_2t_2+c_1 t_2^2 \\
&p_1=0 & q_3=a_3 p_2+b_3 t_2 & s_{12}=a_{12} p_3+b_{12} t_1 &s_{11}=a_{11} p_2^2+b_{11} p_2t_2+c_{11} t_2^2 \\
& s_{23}=0 & s_{13}=a_{13} p_2+b_{13}t_2&    &\\
& & s_{22}=a_{22} p_2+b_{22} t_2 & &\\
\hline
\end{array}$
\end{center}
}}
In this table, the parameters $a_3,\dots,b_0\in \mC$ are chosen generally.
Note that 
\[
C\subset \mP(p_2,t_2,p_3,t_1,r,u)=\mP(2^2,3^2,5,8).
\]

\vspace{3pt}

\noindent {\bf (5859ii: Base loci)}\quad
\begin{center}
$\begin{array}{|c|c|}
\hline
\{p_2=0\}_{|C} = \Bs |\sO(i)|_{|C}\, (i=2,4)&  \Bs |\sO(3)|_{|C} \\
\hline
\text{$u$-point} &    \text{$u$-point}, \sf{p}_2, \sf{q}_2 \\
\hline
\end{array}$
\end{center}
The two point $\sf{p}_2,\sf{q}_2$ are contained in the $p_2$-chart.

\vspace{3pt}

\noindent {\bf (5859iii: LPC)}\quad 
\begin{center}
$\begin{array}{|c|}
\hline
\text{$u$-point} \\
\hline
1/8(3,5)  \\
\hline
\end{array}$
\end{center}
\vspace{3pt}

\noindent {\bf (5859iv: $p_2$-chart)}
We may determine the singularities of $T$ on this chart in a very similar way to that for No.642.
The singular locus consists of $\sf{p}_2,\sf{q}_2$. We see that $T$ has $1/2(1,1)$-singularities at them noting $w(p_2)=2$.

\vspace{5pt}

\noindent \fbox{\bf No.5866}

\vspace{3pt}

\noindent {\bf (5866i: Equation of $C$)}
{\small{
\begin{center}
$\begin{array}{|c|c|c|c|}
\hline
\text{weight} & 1 & 2 & 3  \\
\hline
 & p_4=0 & q_3=a_3 p_1+b_3 p_2 & q_1=a_1p_3+b_1 t_2  \\
&s_{13}=0 & s_{11}=a_{11} p_1+b_{11} p_2 & q_2=a_2 p_3+b_2 t_2  \\
\text{equation}& s_{23}=0 & s_{12}=a_{12} p_1+b_{12}p_2&t_1=c_1 p_3+d_1 t_2  \\
& & s_{22}=a_{22} p_1+b_{22} p_2 & \\
& & t_3=a_0 p_1+b_0 p_2&\\
\hline
\end{array}$
\end{center}
}}
In this table, the parameters $a_3,\dots,d_1\in \mC$ are chosen generally.
Note that 
\[
C\subset \mP(p_1,p_2,p_3,t_2,r,u)=\mP(2^2,3^2,4,7).
\]

\vspace{3pt}

\noindent {\bf (5866ii: Base loci)}\quad
\begin{center}
$\begin{array}{|c|c|}
\hline
\{p_1=p_2=0\}_{|C} = \Bs |\sO(2)|_{|C} &  \Bs |\sO(3)|_{|C} \\
\hline
\text{$u$-point} &   \text{$u$-point}, \sf{p}_2, \sf{q}_2,\sf{r}_2 \\
\hline
\end{array}$
\end{center}
The three points $\sf{p}_2,\sf{q}_2,\sf{r}_2$ are contained in both of the $p_1$- and $p_2$-charts.

$\{p_1=0\}_{|C}$ consists of a finite number of points.
 
\vspace{3pt}

\noindent {\bf (5866iii: LPC)}\quad 
\begin{center}
$\begin{array}{|c|}
\hline
\text{$u$-point} \\
\hline
1/7(3,4)  \\
\hline
\end{array}$
\end{center}
\vspace{5pt}

\noindent {\bf (5866iv: $p_1$-chart)}
We may determine the singularities of $T$ on this chart in a very similar way to that for No.642.
The singular locus consists of $\sf{p}_2,\sf{q}_2,\sf{r}_2$. We see that $T$ has $1/2(1,1)$-singularities at them noting $w(p_1)=2$.

\vspace{5pt}


\noindent \fbox{\bf No.6860}

\vspace{3pt}

\noindent {\bf (6860i: Equation of $C$)}
{\small{
\begin{center}
$\begin{array}{|c|c|c|c|}
\hline
\text{weight} & 1 & 2 & 3  \\
\hline
 & p_4=0 & s_{11}=a_{11}p_1+b_{11}p_2+c_{11}t_3& t_1=a_1p_3+b_1 t_2  \\
&q_3=0 & s_{12}=a_{12} p_1+b_{12} p_2+c_{12}t_3 & r=a_0 p_3+b_0 t_2  \\
\text{equation}& s_{13}=0 & s_{22}=a_{22} p_1+b_{22}p_2+c_{22}t_3&  \\
&s_{23}=0 & q_1=c_1 p_1+d_1 p_2+e_1 t_3 & \\
& & q_2=a_2 p_1+b_2 p_2+c_2 t_3&\\
\hline
\end{array}$
\end{center}
}}
In this table, the parameters $a_{11},\dots,b_0\in \mC$ are chosen generally.
Note that 
\[
C\subset \mP(p_1,p_2,t_3,p_3,t_2,u)=\mP(2^3,3^2,5).
\]

\vspace{3pt}

\noindent {\bf (6860ii: Base loci)}\quad
\begin{center}
$\begin{array}{|c|c|}
\hline
\{p_1=p_2=0\}_{|C} = \Bs |\sO(2)|_{|C} &  \Bs |\sO(3)|_{|C} \\
\hline
\text{$u$-point} &   \text{$u$-point}, \sf{p}_2, \sf{q}_2,\sf{r}_2,\sf{s}_2 \\
\hline
\end{array}$ 
\end{center}
The four points $\sf{p}_2,\sf{q}_2,\sf{r}_2,\sf{s}_2$ are contained in both of the $p_1$- and $p_2$-charts.

$\{p_1=0\}_{|C}$ consists of a finite number of points.

\vspace{3pt}

\noindent {\bf (6860iii: LPC)}\quad 
\begin{center}
$\begin{array}{|c|}
\hline
\text{$u$-point} \\
\hline
1/5(2,3)  \\
\hline
\end{array}$
\end{center}
\vspace{3pt}

\noindent {\bf (6860iv: $p_1$-chart)}
We may determine the singularities of $T$ on this chart in a very similar way to that for No.642.
The singular locus consists of $\sf{p}_2,\sf{q}_2,\sf{r}_2,\sf{s}_2$. We see that $T$ has $1/2(1,1)$-singularities at them noting $w(p_1)=2$.

\vspace{5pt}

\noindent \fbox{\bf No.6865}\,

\vspace{3pt}

\noindent {\bf (6865i: Equation of $C$)}
{\small{
\begin{center}
$\begin{array}{|c|c|c|}
\hline
\text{weight} & 1 & 2   \\
\hline
 & q_3=0 & s_{11}=a_{11}p_1+b_{11}p_2+c_{11}p_4  \\
&t_3=0 & s_{12}=a_{12} p_1+b_{12} p_2+c_{12}p_4  \\
& s_{13}=0 & s_{22}=a_{22} p_1+b_{22}p_2+c_{22}p_4 \\
\text{equation}&s_{23}=0 & t_1=a_1 p_1+b_1 p_2+c_1 p_4  \\
& & t_2=a_2 p_1+b_2 p_2+c_2 p_4\\
& & q_1=d_1 p_1+e_1 p_2+f_1 p_4\\
& & q_2=d_2 p_1+e_2 p_2+f_2 p_4\\
\hline
\end{array}$
\end{center}
}}
In this table, the parameters $a_{11},\dots,f_2\in \mC$ are chosen generally.
Note that 
\[
C\subset \mP(p_1,p_2,p_4,p_3,r,u)=\mP(2^3,3^2,4).
\]
\vspace{3pt}

\noindent {\bf (6865ii: Base loci)}\quad
\begin{center}
$\begin{array}{|c|}
\hline
\{p_1=p_2=0\}_{|C} = \Bs |\sO(2)|_{|C} \\
\hline
\text{$u$-point}   \\
\hline
\end{array}$
\end{center}
\vspace{3pt}

$\{p_1=0\}_{|C}$ consists of a finite number of points.

\vspace{3pt} 

\noindent {\bf (6865iii: LPC)}\quad 
\begin{center}
$\begin{array}{|c|}
\hline
\text{$u$-point} \\
\hline
1/4(1,3)  \\
\hline
\end{array}$
\end{center}
\vspace{3pt}

\noindent {\bf (68651iv: $p_1$-, $p_2$-chart)}
We may determine the singularities of $T$ on this chart in a very similar way to that for No.642.
The singular locus consists of five points $\sf{p}_2,\sf{q}_2,\sf{r}_2,\sf{s}_2,\sf{t}_2$. We see that $T$ has $1/2(1,1)$-singularities at them noting 
$w(p_1)=2$.

\vspace{10pt}

\subsection{Case $h^0(\sO_{\mP_X}(1))\geq 3$}~
\label{Case3}

\begin{center}
\[
\begin{array}{|c|c|c|}
\hline
\text{No.} & \text{Base loci} & \text{$u$-point}\\
\hline
11004 & \Bs |\sO(i)|_{|T}\, (i=1,2,3,4)=\text{the $u$-point} & 1/7(3,4)\\
\hline
16227 & \Bs |\sO(i)|_{|T}\, (i=1,2)=\text{the $u$-point} & 1/5(2,3)\\
\hline
\end{array}
\]
\end{center}

\section{\bf Projections of prime $\mQ$-Fano threefolds}
\label{sec:proj}
It is an interesting problem to study various projections of $X$ from the view point of $\Sigma^{12}_{\mP}$ or $\Sigma^{13}_{\mP}$.
Here we study only a few of them to complete our proof of the main theorem \ref{thm:main}.

\subsection{Type I Tom projection}
In this subsection, we treat $\Sigma^{13}_{\mA}$; the treatment of $\Sigma^{14}_{\mA}$ is similar, so we leave it to the readers.
We investigate the projection $\Sigma^{13}_{\mA}\to \mA(\bm{p},p_4,\bm{q},r,S',\bm{t})$ (or the elimination of the coordinate $u$) and show Theorem \ref{thm:main} (1-2) and (1-3).

By the study of the $u$-chart as in the subsection \ref{Charts}, $\bm{q},r,S',\bm{t}$ are local orbifold coordinates of $\Sigma^{12}_{\mP}$ at the $u$-point. For any class of  $\Sigma^{12}_{\mP}$ as in Theorem \ref{thm:main}, 
the $u$-point is a Type I center since it satisfies the condition of Definition \ref{defn:typeI} (1).
We may check that the closure of the image of the projection is equal to the affine scheme $\Sigma_{\rm{Tom}}$ defined by
the five $4\times 4$ Pfaffians of the skew-symmetric matrix (\ref{eq:Pf}), which is codimension three by Lemma \ref{lem:Tom} (2).
Therefore the projection $\Sigma^{12}_{\mP}\dashrightarrow \Sigma_{\rm{Tom}\mP}$ is a Type I Tom projection since 
$\Sigma_{\rm{Tom}\mP}$ satisfies the condition of Definition \ref{defn:typeI} (2), where $\Sigma_{\rm{Tom}\mP}$ is the weighted projectivization of $\Sigma_{\rm{Tom}}$ for which the weights of $\bm{q},r,S',\bm{t}$ is the same as those for $\Sigma^{12}_{\mP}$.
Moreover, $\Sigma_{\rm{Tom}}$ contains $\{p_1=p_2=p_3=p_4=0\}=\mP(\bm{q},r,S',\bm{t})$, which is the image of the exceptional divisor of the weighted blow-up of $\Sigma^{12}_{\mP}$ at the $u$-point with the weights of local orbinates being given by the weights of global coordinates.

In the same way as above, we see that, for a $\mQ$-Fano threefold $X$ in any class as in Theorem \ref{thm:main} except No.360, No.1218, and No.24078, the $u$-point is a Type I center and the projection from the $u$-point is a Type I Tom projection. By a numerical reason, this is the unique Type I Tom projection for $X$ as indicated by \cite{grdb}.

Finally we consider the key variety $\Sigma_{\mP}^{12}$ of No.1218.
Let's us take a constant $\alpha\in \mC$. Then, by elementary calculations, we may check that
the automorphism of $\mA^{18}$ defined by  
\[
\bm{p}\mapsto \bm{p}+\alpha A\bm{t}, p_4\mapsto p_4-2\alpha r-\alpha^2 u, r\mapsto r+\alpha u
\]
induces an automorphism of the affine variety $\Sigma^{14}_{\mA}$. Moreover, in case of No.1218, this is graded since it holds that $w(p_4)=w(r)=w(u)$ in this case. For No.1218, we set $\alpha=-(-b_4\pm \sqrt{b_4^2-4c_4})/2$ for the points ${\sf p}_5, {\sf q}_5$, where these two points are $1/5 (1,2,3)$-singularities of $X$ (see (\ref{eq:p5q5}) in the item (1218ii: Base loci)). Then we may assume that ${\sf p}_5$ or ${\sf q}_5$ is the $u$-point of $\Sigma^{12}_{\mP}$.
Therefore, the projections from ${\sf p}_5$ and ${\sf q}_5$ of $\SigmaP$ and $X$ are Type I Tom projections as we have seen above.

\subsection{On Type II$_1$ projection}
\label{sub:TypeII1}
In this subsection, we investigate the projection $\Sigma^{13}_{\mA}\to \mA(p_2,p_4,\bm{q},r,u,S,\bm{t})$ (or the elimination of the coordinates $p_1$ and $p_3$) and show Theorem \ref{thm:main} (1-4).

We see that the closure of the image of this projection is equal to 
\[
\Xi:=\{up_2=F_2, up_4=F_4\}.
\]
In a weighted projective setting with a set of weights of the coordinates, 
the center of the corresponding projection $\Sigma^{12}_{\mP}\dashrightarrow \mP(p_2,p_4,\bm{q},r,u,S,\bm{t})$ is the $p_1$-point since $s_{33}=1$ for $\Sigma^{13}_{\mA}$.

For a $\mQ$-Fano threefold $X$ of No.360, the weights of the equations $up_2=F_2,\, up_4=F_4$ are $16,\,14$ respectively. Therefore
we have shown Theorem \ref{thm:main} (1-3).

\section{{\bf Affine variety $\Upsilon^{14}_{\mA}$}}
\label{sec:Up}
The projection as in the subsection \ref{sub:TypeII1} should be a Type II$_1$ projection by a numerical reason as is indicated in the database \cite{grdb}. Therefore following the general theory of Type II$_1$ projection as in \cite{II1}, we investigate this projection. It turns out, however, that this projection does not coincide with a Type II$_1$ projection though there are similarities with it (we may say that it depends on the definition of Type II$_1$ projection). This projection is close to \cite[Example 9.13]{kino}. 

Investigating the projection as in the subsection \ref{sub:TypeII1}, we arrived at another affine variety $\Upsilon^{14}_{\mA}$ in the first version of this paper but we did not study properties of $\Upsilon^{14}_{\mA}$ so much in detail. After writing up the first version of this paper, we knew that R.~Taylor, in her phD thesis, University of Warwick \cite[\S 3.1.2]{Tay}, arrived at the construction of $\Upsilon^{14}_{\mA}$ independently and investigated it more in detail. She calls the projection as in the subsection \ref{sub:TypeII1} {\it Type II$^{(2,1)}_1$ projection}, so we follow her convention hereafter. We remark that Type II$^{(2,1)}_1$ projection is one special case of Type II$^{(n,m)}_1$ projection defined in \cite[\S 2]{Tay}. As for the following construction of $\Upsilon^{14}_{\mA}$, it is possible only to quot her result but we keep it to explain our way to derive it in our context although we update our result by using \cite{Tay}. 

\subsection{Construction of $\Upsilon^{14}_{\mA}$}\label{sub:constUp}
Consider the inverse of the projection $\Xi\dashrightarrow \Sigma^{13}_{\mA}$ as in the subsection \ref{sub:TypeII1}. This is actually defined as a rational map since if a point of $\Xi$ satisfies $u\not =0$, then the image of the inverse of the projection can be defined by the equations $up_1=F_1$ and $up_3=F_3$ in (\ref{Key2}), hence $\Xi\dashrightarrow \Sigma^{13}_{\mA}$ can be defined as a morphism on the open subset $\{u\not=0\}\cap \Xi$. This fact also implies that the indeterminacy of the inverse of the projection is contained in $\{u=0\}\cap \Xi$.
By computations, we see that $\{u=0\}\cap \Xi$ has two irreducible components; one is $D':=\{u=F_1=F_2=F_3=F_4=0\}$, and another is 
\[
D=\{u=0, \, \rank \begin{spmatrix}   q'_1 & r' & s'_{11} q_3 & s'_{11} q_2\\
q_3 & q_2 & q'_1& r' \end{spmatrix}\leq 1\},
\]
where we set $r':=r-\begin{svmatrix} s_{12} & s_{23} \\ q_1 & q_3 \end{svmatrix}$, $q'_1:=q_1-s_{13} q_3$, and $s'_{11}:=-s_{11}+s_{13}^2$.
Note that $\{u=0\}\cap \Sigma^{13}_{\mA}$ is contained in $\{F_1=F_2=F_3=F_4=0\}$. By computations, we see that the divisor $\{u=0\}\cap \Sigma^{13}_{\mA}$
dominates $D'$. Therefore, $D$ is the divisor to be unprojected.
We may easily verify that the equations $up_2=F_2$ and $up_4=F_4$ are written as follows:
\begin{equation}
\label{eq:1218codim2}
\begin{cases}
up_2=t'_1\Delta_{12}+t_2\Delta_{13}-t_3\Delta_{23},\\
up_4=s'_{12}\Delta_{12}-s_{22}\Delta_{13}+2s_{23}\Delta_{23}-\Delta_{24}, 
\end{cases}
\end{equation}
where $\Delta_{ij}$ is the $(i,j)$ minor of the matrix 
$\begin{spmatrix}   q'_1 & r' & s'_{11} q_3 & s'_{11} q_2\\
q_3 & q_2 & q'_1& r' \end{spmatrix}$, and we set $t'_1:=s_{13}t_3-t_1$ and $s'_{12}:=2(s_{12}-s_{13}s_{23})$.
It is easy to see the coordinate change $r\mapsto r', q_1\mapsto q'_1, s_{11}\mapsto s'_{11}, t_1\mapsto t'_1, s_{12}\mapsto s'_{12}$ as above is graded with respect to the weights such that the equations in (\ref{Key1}) and (\ref{Key2}) are homogeneous.

\vspace{3pt}

Slightly more generally, we consider  
\[
D:=\left\{u=0,\, \rank \begin{spmatrix} y_1 & y_2 & zx_1 & zx_2 \\
x_1 & x_2 & y_1 & y_2 \end{spmatrix}\leq 1\right\}
\]
in the following codimension two complete intersection:
\begin{equation}
\label{eq:codim2}
X:=
\begin{cases}
up_2=t_1\Delta_{12}+t_2\Delta_{13}+t_3\Delta_{23}+t_4\Delta_{24},\\
up_4=s_1\Delta_{12}+s_2\Delta_{13}+s_3\Delta_{23}+s_4\Delta_{24},
\end{cases}
\end{equation}
where $\Delta_{ij}$ is the $(i,j)$ minor of the matrix 
$\begin{spmatrix} y_1 & y_2 & zx_1 & zx_2 \\
x_1 & x_2 & y_1 & y_2 \end{spmatrix}$, and 
$p_2,p_4, t_i, s_j$ are variables.
There are obvious correspondences between variables for the equations of $\Sigma^{13}_{\mA}$ and the general format considered here; the coordinates $t_1$, $t_2$, $t_3$, $t_4$; $s_1$, $s_2$, $s_3$, $s_4$ in (\ref{eq:codim2}) correspond to 
$t'_1$, $t_2$, $-t_3$, $0$; $s'_{12}$, $-s_{22}$, $2s_{23}$, $-1$ in  (\ref{eq:1218codim2}) respectively. 
 
By (\ref{eq:codim2}), we obtain
\[
0=A_{12}(-\Delta_{12})+B_{11}\Delta_{13}+2B_{12}\Delta_{23}+B_{22}\Delta_{24},
\]
where, imitating p.2206 of \cite{II1}, we set 
$A_{12}=p_4t_1-p_2s_1, B_{11}=p_2s_2-p_4t_2, B_{12}=1/2 (p_2s_3-p_4t_3), B_{22}=p_2s_4-p_4t_4.$
Then we see that it is possible to choose two unprojection variables $p_1$ and $p_3$ such that the equations of the unprojection consists of the following nine equations:
\begin{itemize}
\item
the two equations (\ref{eq:codim2}), 
\item
five $4\times 4$ Pfaffians of the $5\times 5$ skew-symmetric matrix
\begin{equation}
\label{eq:5rel}
\begin{spmatrix}
0 & x_1 & x_2 & y_1  & y_2 \\
 & 0 & -p_1 & -B_{22} & p_3+B_{12}\\
& & 0 & -p_3+B_{12} & -B_{11}\\
& & & 0 & -zp_1-A_{12}
\end{spmatrix}
\end{equation}
(these five correspond to $l_1,\dots,l_4,q$ as in p.2206 of \cite{II1}), and 
\item
two new relations
\begin{equation}
\label{eq:2rel}
\begin{cases}
p_1u+\begin{spmatrix} t_1 & t_2 & t_3 & t_4 \end{spmatrix}M_1
\begin{spmatrix} s_1\\ s_2 \\s_3 \\s_4 \end{spmatrix}=0,\\
p_3u+\begin{spmatrix} t_1 & t_2 & t_3 & t_4 \end{spmatrix}M_3
\begin{spmatrix} s_1\\ s_2 \\s_3 \\s_4 \end{spmatrix}=0,
\end{cases}
\end{equation}
where
\begin{align*}
&M_1:=\begin{spmatrix} 0 & -x_1^2 & -x_1 x_2 & -x_2^2 \\
& 0 & -x_1 y_1 & -x_1 y_2-x_2 y_1\\
& & 0 & -x_2 y_2 \\
& & & 0\end{spmatrix},\\
&M_3:=\begin{spmatrix} 0 & x_1 y_1 & 1/2 (x_1y_2+x_2y_1) & x_2 y_2 \\
& 0 & 1/2(y_1^2+zx_1^2) & y_1 y_2+zx_1x_2\\
& & 0 & 1/2(y_2^2+zx_2^2)\\
& & & 0
\end{spmatrix}
\end{align*}
are skew-symmetric matrices.
\end{itemize}

\begin{defn}
\label{defn:Upsilon}
We denote by $\Upsilon^{14}_{\mA}$ the affine variety defined by these nine equations.
\end{defn}

We can easily verify that $\Upsilon^{14}_{\mA}$ is isomorphic to the affine variety defined by the nine equations as in \cite[Thm.3.1.2 and Prop.~3.1.1]{Tay}. Therefore, by \cite[Thm.~2.2.1 and Cor.~2.7.1]{Tay}, it holds that $\Upsilon^{14}_{\mA}$ is Gorenstein.

\subsection{Relation between $\Sigma^{13}_{\mA}$ and $\Upsilon^{14}_{\mA}$}

A close relationship between $\Sigma^{13}_{\mA}$ and $\Upsilon^{14}_{\mA}$ is indicated in the subsection \ref{sub:constUp}. Actually, we have the following:
\begin{prop}
\label{Sigma13Upsilon14}
There exists an isomorphism from $\Sigma^{13}_{\mA}\times \mA^{1*}$ to $\Upsilon^{14}_{\mA}\cap \{s_4\not =0\}$.
\end{prop}
\begin{proof}
By the obvious correspondence between the equations (\ref{eq:1218codim2}) and (\ref{eq:codim2}), we can easily verify that
there exists an isomorphism from $\Sigma^{13}_{\mA}$ to $\Upsilon^{14}_{\mA}\cap \{t_4=0, s_4 =-1\}$. 
Indeed, with the coordinates $p_1,p_2,p_4,u,t_2$ being unchanged, the following equalities defines this isomorphism, where the l.h.s. are the cooridinates of $\Upsilon^{14}_{\mA}$ and the r.h.s.~are polynomials of the cooridinates of $\Sigma^{13}_{\mA}$, and, to avoid confusion, we denote the coordinates $t_1,t_3,p_3$ of $\Upsilon^{14}_{\mA}$ by the capital letters $T_1,T_3,P_3$:
\begin{align*}
&T_1 = s_{13} t_3 - t_1, T_3 = -t_3, s_1 = 2 (s_{12} - s_{13} s_{23}), s_2 = -s_{22}, s_3 = 2 s_{23},\\
 &x_1 = q_3, x_2 = q_2, y_1 =
  q_1 - s_{13} q_3, y_2 = r - (s_{12} q_3 - q_1 s_{23}),\\
  & z = -s_{11} + s_{13}^2, P_3 =  p_3 +  p_1 s_{13} +  p_2 s_{23} +1/2 p_4 t_3.
  \end{align*}
On $\Upsilon^{14}_{\mA}\cap\{s_4\not=0\}$, we can replace the two equations in (\ref{eq:codim2}) by 
{\small{
\begin{equation*}
\label{eq:codim2v2}
\begin{cases}
u(p_2-(t_4/s_4)p_4)=(t_1-(t_4/s_4) s_1)\Delta_{12}+(t_2-(t_4/s_4) s_2)\Delta_{13}+(t_3-(t_4/s_4) s_3)\Delta_{23},\\
u(-p_4/s_4)=(-s_1/s_4)\Delta_{12}-(s_2/s_4)\Delta_{13}-(s_3/s_4)\Delta_{23}-\Delta_{24},
\end{cases}
\end{equation*}}}
where the first equation is equal to the first equation of  (\ref{eq:codim2})$-t_4/s_4\times$ the second equation of  (\ref{eq:codim2}), and the second equation is equal to $-1/s_4\times$ the second equation of  (\ref{eq:codim2}). From this, we see that  $\Upsilon^{14}_{\mA}\cap \{t_4=0, s_4 =-1\}\times \mA^{1*}$ is ismorphic to $\Upsilon^{14}_{\mA}\cap \{s_4\not =0\}$, where $s_4$ is the coordinate of $\mA^{1*}$.
 \end{proof}
In \cite[Chap.4]{Tay}, candidates of prime $\mQ$-Fano threefolds are constructed by using $\Upsilon^{14}_{\mA}$ (though the Picard numbers are not calculated) for several classes including the classes considered in this paper except the 8 classes No.393, 642, 644, 4850, 5202, 11004, 16227, and 24078, for which correspoinding $\mQ$-Fano threefolds have no Type II$^{(2,1)}_1$ projection. By Proposition \ref{Sigma13Upsilon14}, however, examples of prime $\mQ$-Fano threefolds can be constructed by using $\Upsilon^{14}_{\mA}$ also for all the classes considered in this paper except one class No.24078 since they are constructed by using $\Sigma^{13}_{\mA}$ in this paper. As for No.24078, it seems that no examples of prime $\mQ$-Fano threefolds can be constructed by using $\Upsilon^{14}_{\mA}$ though we do not search the weights of coordinates exhaustively. We will give one evidence for this inference in Proposition \ref{nonisoSIUp}.

In a further work, we show that candidates of prime $\mQ$-Fano threefolds constructed in \cite[Chap.4]{Tay} but not constructed in this paper are essentially the same as the corresponding examples of prime $\mQ$-Fano threefolds constructed in \cite{CD} by clarifying a relationship between $\Upsilon^{14}_{\mA}$ and the cluster variety of type $C_2$.

\section{\bf More on geometry of $\Sigma^{14}_{\mA}$}
\label{More14}

\subsection{$\GL_3$- and $(\mC^*)^6$-action}
\label{sub:sl3}
We define a $\GL_3$- and a $(\mC^*)^6$-actions on $\Sigma_{\mA}^{14}$.

By the conditions on the weights of coordinates (\ref{eq:wsij})--(\ref{eq:wti}), we have a $(\mC^*)^6$-action on $\Sigma_{\mA}^{14}$ since we may choose $d_0$, $w(p_1)$, $w(p_2)$, $w(p_3)$, $w(r)$, and $w(u)$ freely.

We define a $\GL_3$-action on $\Sigma_{\mA}^{14}$. We define a $\GL_3$-action on the ambient affine space $\mA^{18}$ of
$\Sigma_{\mA}^{14}$ by setting
\begin{align*}
&\bm{p}\mapsto g\bm{p},\, \bm{q}\mapsto ({\empty^t}\!{g})^{\dagger}\bm{q},
S\mapsto  ({\empty^t}\!{g})^{\dagger}Sg^{\dagger},\, \bm{t}\mapsto  ({\empty^t}\!{g})^{\dagger}\bm{t},\\
&p_4\mapsto (\det g) p_4,\, r\mapsto (\det g)^2 r,\, u\mapsto (\det g)^3 u
\end{align*}
for  any element $g$ of $\GL_3$.
Then we can easily check that
\begin{align*}
&A\mapsto gA ({\empty^t\!g}),\, S^{\dagger}\mapsto (\det g)^2 gS^{\dagger} ({\empty^t\!g}),\\
&\begin{spmatrix} F_1 \\ F_2 \\F_3 \end{spmatrix}\mapsto (\det g)^3 g\begin{spmatrix} F_1 \\ F_2 \\F_3 \end{spmatrix},\,F_4\mapsto (\det g)^4 F_4,
\end{align*}
and hence a $\GL_3$-action on $\Sigma_{\mA}^{14}$ is certainly defined by the equations (\ref{Key1}) and (\ref{Key2}).

\subsection{$\mP^2\times \mP^2$-fibration related with $\Sigma^{14}_{\mA}$}
\label{fib13}
Note that 
any equation of $\Sigma^{14}_{\mA}$ is of degree two if 
we regard the entries of $S$ and $\bm{t}$ as constants.
Therefore, considering 
the variables of the equations of $\Sigma_{\mA}^{14}$ except the entries of $S$ and $\bm{t}$ as projective coordinates of weight one, we obtain a $13$-dimensional variety in $\mA(S,\bm{t})\times \mP(\bm{p},p_4,\bm{q},r,u)$, where $\mA(S,\bm{t})$ is the affine space with coordinates $S,\bm{t}$, and $\mP(\bm{p},p_4,\bm{q},r,u)$ is the projective space with the coordinates $\bm{p},p_4,\bm{q},r,u$. We denote this variety by $\widehat{\Sigma}^{13}$, and by $\rho_{\Sigma}\colon \widehat{\Sigma}^{13}\to \mA(S,\bm{t})$ the natural projection. 
Note that the $\GL_3$- and $(\mC^*)^6$-actions on $\Sigma_{\mA}^{14}$ defined as in the subsection \ref{sub:sl3}
induces those on $\widehat{\Sigma}^{13}$ and those on $\mA(S,\bm{t})$, and
the natural projection $\rho_{\Sigma}\colon \widehat{\Sigma}^{13}\to \mA(S,\bm{t})$
is equivariant with respect to these actions.

In Proposition \ref{prop:pifib} below, we shall show that a $\rho_{\Sigma}$-general fiber is isomorphic to $\mP^2\times \mP^2$ with more precise descriptions.

As a preparation, let \[
M:=\begin{spmatrix}
m_{11}& m_{12} & m_{13}\\
m_{12} & m_{22} & m_{23}\\
m_{13} & m_{23} & m_{33}
\end{spmatrix},\, \bm{\sigma}=\begin{spmatrix} \sigma_1 \\ \sigma_2 \\ \sigma_3\end{spmatrix},
\]
where we note that $M$ is a symmetric matrix.
Consider the quasi-projective variety ${\tt \Sigma}$ defined by the following equations: 
\begin{equation}
\label{eq:doublesym}
M^{\dagger}=\lambda\bm{\sigma} ({{\empty}^t\!\bm{\sigma}}),\, M\bm{\sigma}=\bf{0}
\end{equation}
in $\mP(M)\times \mA(\lambda)\simeq \mP^8\times \mA^1$, 
where $M^{\dagger}$ is the adjoint matrix of $M$, $\mP(M)$ is the projective space whose coordinates are the entries of $M$,
and $\mA(\lambda)$ is the affine line with the coordinate $\lambda$.
Let $\pi_{\tt \Sigma}\colon {\tt \Sigma}\to \mA(\lambda)$ be the natural projection.
The $\pi_{\tt \Sigma}$-fiber over $0$ is defined by 
 \begin{equation}
M^{\dagger}=O,\, M\bm{\sigma}=\bf{0}.
\end{equation}
This is nothing but the singular sextic del Pezzo $4$-fold $\mP^{2,2}$ introduced in \cite{P22M} and revisited in \cite{Fu}.
The equation of the $\pi_{\tt \Sigma}$-fiber over a point $\lambda\not =0$ turns out to be equal to 
\begin{equation}
\label{eq:P2P2}
\rank \begin{spmatrix}
m_{11}& m_{12}-\sqrt{-\lambda}\sigma_3 & m_{13}+\sqrt{-\lambda}\sigma_2\\
m_{12}+\sqrt{-\lambda}\sigma_3 & m_{22} & m_{23}-\sqrt{-\lambda}\sigma_1\\
m_{13}-\sqrt{-\lambda}\sigma_2 & m_{23}+\sqrt{-\lambda}\sigma_1 & m_{33}
\end{spmatrix}\leq 1.
\end{equation}
Therefore the $\pi_{\tt \Sigma}$-fiber over a point $\lambda\not =0$ is isomorphic to the Segre embedded $\mP^2\times \mP^2$.
\begin{rem}
In a forthcoming paper, we show that the cluster variety of type $C_2$ with the coordinates $A_1$ and $A_3$ being nonzero constants is isomorphic to the product of $\mA^1$ and the affine variety $\tt{\Sigma}_{\mA}$ defined by the equation (\ref{eq:doublesym}),
where we refer to \cite{CD} for the notation $A_1$ and $A_3$. 
\end{rem}

\begin{prop}
\label{prop:pifib}
Descriptions of $\rho_{\Sigma}$-fibers are given as follows, where we set $B:=\begin{spmatrix} 0 & -t_3 & t_2\\
                   t_3 & 0 & -t_1\\
                  -t_2 & t_1 & 0
\end{spmatrix}$, $Q^n$ denotes an $n$-dimensional quadric, and $(\mP^1\times \mP^3)'$ denotes the cone over a hyperplane section of the Segre embedded $\mP^1\times \mP^3$:
\begin{itemize}
\item
The $\rho_{\Sigma}$-fiber over a point with $S\not=O$ and $\bm{t}\not=\bm{0}$ is described as in Table \ref{Tablefib1}.
{\small{
\begin{table}[hp]
\begin{center}
\caption{$\rho_{\Sigma}$-fibers in case of $S\not=O$ and $\bm{t}\not=\bm{0}$}
\label{Tablefib1}
\begin{tabular}{|c||c|c||c|c|c||c|c|} \hline
$\rank S$ & \multicolumn{2}{c||}{$3$} & \multicolumn{3}{c||}{$2$} &  \multicolumn{2}{c|}{$1$} \\ \hline
${{\empty}^t\bm{t}}S^{\dagger} \bm{t}$ & $\not =0$ & $0$ & $\not =0$ & \multicolumn{2}{c||}{$0$} & \multicolumn{2}{c|}{$0$}\\ \hline
$\rank BSB$ & $2$ & $1$ & $2$ & $1$ & $0$ & $1$ & $0$\\ \hline
$\rho_{\Sigma}$-fiber & $\mP^2\times \mP^2$ & $\mP^{2,2}$ & $\mP^2\times \mP^2$ & $\mP^{2,2}$ & $\mP^1\times \mP^3\cup$ $Q^4$& $\mP^{2,2}$ & $(\mP^1\times \mP^3)'\cup$ $Q^4$ \\ \hline
\end{tabular}
\end{center}
\end{table}
}}
\item
The $\rho_{\Sigma}$-fiber over a point with $S\not=O$ and $\bm{t}=\bm{0}$ is described as in Table \ref{Tablefib2}.
\begin{table}[!hbt]
\begin{center}
\caption{$\rho_{\Sigma}$-fibers in case of $S\not =O$ and $\bm{t}=\bm{0}$}
\label{Tablefib2}
\begin{tabular}{|c||c|c|c|} \hline
$\rank S$ & $3$ & $2$ & $1$\\ \hline
$\rho_{\Sigma}$-fiber & $(\mP^1\times \mP^3)'\cup$ $Q^4$ & $\mP^4\cup \mP^4\cup\mP^4\cup Q^5\cup Q^5$ & $Q^4\cup$ double $Q^4$\\ \hline
\end{tabular}
\end{center}
\end{table}
\item
The $\rho_{\Sigma}$-fiber over a point with $S=O$ and $\bm{t}\not=\bm{0}$ is the union of $(\mP^1\times \mP^3)'$ and a smooth $Q^4$.
\item
The $\rho_{\Sigma}$-fiber over the point with $S=O$ and $\bm{t}=\bm{0}$ is the union of the $Q^5=\{r=u=p_1 q_1+p_2 q_2+p_3 q_3=0\}$ and the $Q^4=\{p_1=p_2=p_3=r^2+p_4 u=0\}$.

\end{itemize}
In particular, the $\rho_{\Sigma}$-fiber over a point $(S,\bm{t})$ is $\mP^2\times \mP^2$ (resp.~$\mP^{2,2}$) if and only if $(S,\bm{t})$ belongs to the open subset $\{{{\empty}^t\bm{t}}S^{\dagger} \bm{t}\not =0\}$ (resp.~the locally closed subset $\{{{\empty}^t\bm{t}}S^{\dagger} \bm{t} =0, BSB\not=O\}$).
\end{prop}

\begin{proof}
To derive the descriptions of  $\rho_{\Sigma}$-fibers,
it is important to note that the condition that ${{\empty}^t\bm{t}}S^{\dagger} \bm{t}=0$ or not, and $\rank BSB$ are invariant under the $\GL_3$- and $(\mC^*)^6$-actions on $\mA(S,\bm{t})$. Therefore we may describe the $\rho_{\Sigma}$-fiber over a point $(S,\bm{t})$ transforming $(S,\bm{t})$ to a point with a simpler form by the $\GL_3$- and $(\mC^*)^6$-actions.

Assume that $\rank S=3$ and $\bm{t}\not=\bm{0}$. Then, by the $\GL_3$- and $(\mC^*)^6$-actions, we may assume that
$S=\begin{spmatrix} 1 & 0 & 0\\ 0& 1 & 0 \\ 0 & 0 & 1 \end{spmatrix}$ and $t_2^2+t_3^2=1$. Then setting 
\[
M=\begin{spmatrix} t_1 p_1+t_2 p_2+t_3 p_3 & -t_3 p_2+t_2 p_3 & q_1-t_1 t_2 q_2-t_1 t_3 q_3\\
-t_3 p_2+t_2 p_3 & -p_4+t_1 p_1-t_2 p_2-t_3 p_3 & r+t_1 t_3 q_2-t_1 t_2 q_3\\
q_1-t_1 t_2 q_2-t_1 t_3 q_3& r+t_1 t_3 q_2-t_1 t_2 q_3 & u \end{spmatrix},
\]
$\bm{\sigma}=\begin{spmatrix} t_2 q_2+t_3 q_3 \\ t_2 q_3-t_3 q_2 \\ p_1 \end{spmatrix}$, and $\lambda=t_1^2+1$,
we may check the $\rho_{\Sigma}$-fiber is defined by the format (\ref{eq:doublesym}). Since ${{\empty}^t\bm{t}}S^{\dagger} \bm{t}=t_1^2+1$,
the $\rho_{\Sigma}$-fiber is $\mP^2\times \mP^2$ or $\mP^{2,2}$ if ${{\empty}^t\bm{t}}S^{\dagger} \bm{t}\not =0$ or $=0$ respectively.
Since $BSB=\begin{spmatrix} -1& t_1 t_2& t_1 t_3 \\ t_1 t_2&  -t_1^2 - t_3^2 & t_2 t_3\\ t_1 t_3 & 
  t_2 t_3 &-t_1^2 - t_2^2\end{spmatrix}$, we may check that if ${{\empty}^t\bm{t}}S^{\dagger} \bm{t}\not =0$ or $=0$, then 
$\rank BSB=2$ or $1$ respectively.

Assume that $\rank S=2$ and $\bm{t}\not=\bm{0}$. Then, by the $\GL_3$- and $(\mC^*)^6$-actions, we may assume that
$S=\begin{spmatrix} 0 & 1 & 0\\ 1 & 0 & 0 \\ 0 & 0 & 0 \end{spmatrix}$ and $\bm{t}=\begin{spmatrix} 0 \\ 0 \\ 1\end{spmatrix}$,
$\begin{spmatrix} 1 \\ 1 \\ 0\end{spmatrix}$, or $\begin{spmatrix} 1 \\ 0 \\ 0\end{spmatrix}$ considering the possibilities of three lines in $\mP^2$. Since ${{\empty}^t\bm{t}}S^{\dagger} \bm{t}=-t_3^2$ and $BSB=\begin{spmatrix} 0&  t_3^2& -t_2 t_3\\
 t_3^2& 0& -t_1 t_3\\ -t_2 t_3& -t_1 t_3& 2 t_1 t_2\end{spmatrix}$ for $S=\begin{spmatrix} 0 & 1 & 0\\ 1 & 0 & 0 \\ 0 & 0 & 0 \end{spmatrix}$, the three possibilities of $\bm{t}$ corresponds to the three cases that
${{\empty}^t\bm{t}}S^{\dagger} \bm{t}\not =0$ and $\rank BSB=2$, ${{\empty}^t\bm{t}}S^{\dagger} \bm{t} =0$ and $\rank BSB=1$,
or ${{\empty}^t\bm{t}}S^{\dagger} \bm{t} =0$ and $BSB=O$ respectively. 
In the case that $\bm{t}=\begin{spmatrix} 0 \\ 0 \\ 1\end{spmatrix}$ or 
$\begin{spmatrix} 1 \\ 1 \\ 0\end{spmatrix}$, setting $M$ and $\bm{\sigma}$ as follows and $\lambda:={{\empty}^t\bm{t}}S^{\dagger} \bm{t}=-t_3^2$,
we may check the $\rho_{\Sigma}$-fiber is defined by the format (\ref{eq:doublesym}).
In the case that $\bm{t}=\begin{spmatrix} 	1 \\ 0 \\ 0\end{spmatrix}$,
the $\rho_{\Sigma}$-fiber is described as follows.
\begin{itemize}
\item $\bm{t}=\begin{spmatrix} 0 \\ 0 \\ 1\end{spmatrix}$:
\begin{align*}
&M=\begin{spmatrix} p_3 & 1/\sqrt{2}\, (p_1+p_2) & -1/\sqrt{2}\, (q_1-q_2)\\
1/\sqrt{2}\, (p_1+p_2) & -p_4 & r \\
-1/\sqrt{2}\, (q_1-q_2) & r & u
\end{spmatrix},
\bm{\sigma}=\begin{spmatrix}
q_3 \\ 1/\sqrt{2}\, (q_1+q_2) \\ 1/\sqrt{2}\, (-p_1+p_2)
\end{spmatrix}.
\end{align*}
\item
$\bm{t}=\begin{spmatrix} 1\\ 1 \\ 0\end{spmatrix}$:
\begin{align*}
&M=\begin{spmatrix} p_1+p_2 & -i/\sqrt{2}\, (p_1-p_2) & \sqrt{2}iq_3\\
-i/\sqrt{2}\, (p_1-p_2) & -p_4-1/2\, (p_1+p_2) & r \\
\sqrt{2}i\, q_3 & r & u
\end{spmatrix},
\bm{\sigma}=\begin{spmatrix}
-1/2\, (q_1+q_2) \\ -i/\sqrt{2}\, (q_1-q_2) \\ 1/\sqrt{2}\, i p_3
\end{spmatrix}.
\end{align*}
\item
$\bm{t}=\begin{spmatrix} 1\\ 0 \\ 0\end{spmatrix}$:
\begin{align*}
&\{p_1=0,\rank \begin{spmatrix} p_2 & p_3 & q_3-r & p_4 \\
-q_3 & q_2 & u & q_3+r
\end{spmatrix}\leq 1\}\displaystyle\cup\\
&\qquad\{q_3 =r, u=0, p_4= -2 p_2, p_1 q_1 + p_2 q_2 + p_3 q_3=0\}.
\end{align*}
\end{itemize}
We can describe the $\rho_{\Sigma}$-fibers similarly in the other cases, hence we omit them.
\end{proof}

\subsection{$\mP^2\times \mP^2$-fibration related with $\Upsilon^{14}_{\mA}$}
\label{sub:P2P2Up}
Note that 
any equation of $\Upsilon^{14}_{\mA}$ is of degree two if 
we regard the coordinates $z$ and $s_i, t_j$ $(1\leq i,j\leq 4)$ as constants.
Therefore, similarly to the definition of $\widehat{\Sigma}^{13}$, a $13$-dimensional variety $\widehat{\Upsilon}^{13}$ in $\mA(z,s_i,t_j)\times \mP(x_1,x_2,y_1,y_2,p_1,p_2,p_3,p_4,u)$ is defined by $\Upsilon^{14}_{\mA}$.
Let $\rho_{\Upsilon}\colon \widehat{\Upsilon}^{13}\to \mA(z,s_i,t_j)$ be the natural projection.
Note that the isomorphism from $\Sigma^{13}_{\mA}\times \mA^{1*}$ to $\Upsilon^{14}_{\mA}|_{s_4\not =0}$ as in Proposition \ref{Sigma13Upsilon14} is linear with respect to the coordinates vertical for the projections to $\mA(S',\bm{t})\times \mA^{1*}$ and $\mA(z,s_i,t_j)|_{s_4\not =0}$ respectively. Thus this induces an isomorphism $(\widehat{\Sigma}^{13}|_{s_{33}=1})\times \mA^{1*}\to \widehat{\Upsilon}^{13}|_{s_4\not =0}$. Moreover this induces an isomorphism $\mA(S',\bm{t})\times \mA^{1*}\to \mA(z,s_i,t_j)|_{s_4\not =0}$ such that the following diagram commutes:
\[
\xymatrix{(\widehat{\Sigma}^{13}|_{s_{33}=1})\times \mA^{1*}\ar[r]\ar[d] & \widehat{\Upsilon}^{13}|_{s_4\not =0}\ar[d]\\
\mA(S',\bm{t})\times \mA^{1*}\ar[r] & \mA(z,s_i,t_j)|_{s_4\not =0}.}
\]
Therefore, by Proposition \ref{prop:pifib}, we have
\begin{prop}
A general fiber of $\rho_{\Upsilon}\colon \widehat{\Upsilon}^{13}\to \mA(z,s_i,t_j)$ is isomorphic to $\mP^2\times \mP^2$.
\end{prop}

Although we avoid to determine all the $\rho_{\Upsilon}$-fibers, we see that the $\rho_{\Upsilon}$-fiber over the origin of $\mA(z,s_i,t_j)$ is 
\[
\mP(x_1,x_2,y_1,y_2,u)\cup \mP(x_1,x_2,x_3,x_4,p_2,p_4)\cup \left\{\rank \begin{pmatrix}
y_1& y_2 &  p_3&  p_1z\\
 x_1& x_2 & - p_1 & p_3 \end{pmatrix}\leq 1\right\}.
\]
By Proposition \ref{prop:pifib}, there is no $\rho_{\Sigma}$-fiber isomorphic to the $\rho_{\Upsilon}$-fiber over the origin of $\mA(z,s_i,t_j)$. Therefore we have the following:
\begin{prop}
\label{nonisoSIUp}
There is no isomorphism $\widehat{\Sigma}^{13}\to \widehat{\Upsilon}^{13}$ such that the following diagram is commutative with some isomorphism $\mA(S,\bm{t})\to \mA(z,s_i,t_j)$:
\[
\xymatrix{\widehat{\Sigma}^{13}\ar[r]\ar[d]_{\rho_{\Sigma}} & \widehat{\Upsilon}^{13}\ar[d]^{\rho_{\Upsilon}}\\
\mA(S,\bm{t})\ar[r] & \mA(z,s_i,t_j)}.
\]
\end{prop}

\subsection{Duality of conics}
\label{sub:dual}
In this subsection, we point out the classical duality of conics is hidden in projective geometry of the $\mP^2\times\mP^2$-fibration $\rho_{\Sigma}\colon \colon \widehat{\Sigma}^{13}\to \mA(S,\bm{t})$.

Denote by $N:=\begin{spmatrix} n_{11} & n_{12} & n_{13} \\ n_{21} & n_{22} & n_{23} \\ n_{31} & n_{32} & n_{33} \end{spmatrix}$
the matrix as in (\ref{eq:P2P2}). Consider the projection of the $\mP^2\times \mP^2$ defined by $\rank N\leq 1$
from the $\mP^1\times \mP^1$ defined by the condition that $\det \begin{spmatrix} n_{11} & n_{12} \\ n_{21} & n_{22}  \end{spmatrix}=0$.
It is classically known (and it is easy to figure out) that 
\begin{itemize}
\item
this projection is birational and the target is the $4$-dimensional projective space $\mP(n_{13},n_{23},n_{31},n_{32},n_{33})$,
\item
the exceptional locus consists of the two divisors $E_1$ and $E_2$ defined by
\[
\rank \begin{spmatrix} n_{11} & n_{12} & n_{13} \\ n_{21} & n_{22} & n_{23}  \end{spmatrix}\leq 1,\ \text{and}\
\rank
\begin{spmatrix} n_{11} & n_{12} \\ n_{21} & n_{22}  \\ n_{31} & n_{32} \end{spmatrix}\leq 1,
\]
respectively.
\item
The images of $E_1$ and $E_2$ by the projection are
the skew lines $\{(n_{13}:n_{23})\}$ and $\{(n_{31}:n_{32})\}$ respectively in $\mP(n_{13},n_{23},n_{31},n_{32})\simeq \mP^3$.
\end{itemize} 

Now we consider the situation of Proposition \ref{prop:pifib}; the $\rho_{\Sigma}$-fiber over a point of $\mA(S,\bm{t})$ is $\mP^2\times \mP^2$ if and only if 
$\rank S=2, 3$ and  ${{\empty}^t\bm{t}}S^{\dagger} \bm{t}\not =0$. Hereafter in this subsection, we fix a pair of $S$ and $\bm{t}$ with this conditions. 
In this situation, we can check that the quadric $\det \begin{spmatrix} n_{11} & n_{12} \\ n_{21} & n_{22}  \end{spmatrix}=0$ is equal to
$\{{\empty^t \bm{p}}S\bm{p}+p_4 {\empty^t \bm{p}}\bm{t}=0\}$ as in (\ref{Key1}) (by the $\GL_3$- and $(\mC^*)^6$-actions on $\mA(S,\bm{t})$, we may assume that $S$ and $\bm{t}$ are of the form as in the proof of Proposition \ref{prop:pifib}. Then the check becomes easier). This observation leads us to consider the relative projection of $\SigmaA$ over $\mA(S,\bm{t})$ from the family of quadrics  $\{{\empty^t \bm{p}}S\bm{p}+p_4 {\empty^t \bm{p}}\bm{t}=0\}$. Then we see that the target of the projection is the $\mP^4$-bundle
$\mP(\bm{q},r,u)\times \mA(S,\bm{t})$ and the projection is birational.
We may also check that, in the case that $\rank S=2, 3$ and  ${{\empty}^t\bm{t}}S^{\dagger} \bm{t}\not =0$, the two skew lines $\{(n_{13}:n_{23})\}$ and $\{(n_{31}:n_{32})\}$ as above is defined by $\{u=F_1=F_2=F_3=F_4=0\}$ in $\mP(\bm{q},r,u)\times \mA(S,\bm{t})$.

We observe that the images of these skew lines by the projection to $\mP(\bm{q})=\mP^2$ have a nice projective geometric meaning. Assume that $\rank S=3$ and  ${{\empty}^t\bm{t}}S^{\dagger} \bm{t}\not =0$. By the $\GL_3$- and $(\mC^*)^6$-actions on $\mA(S,\bm{t})$, we may assume that $S$ and $\bm{t}$ are of the form as in the proof of Proposition \ref{prop:pifib}. Then we may read off that the the images of these skew lines by the projection to $\mP(\bm{q})$ are the two lines $\{q_1-t_1 t_2 q_2-t_1 t_3 q_3+\sqrt{1-t_1^2}(t_2 q_3-t_3 q_2)=0\}$ and $\{q_1-t_1 t_2 q_2-t_1 t_3 q_3-\sqrt{1-t_1^2}(t_2 q_3-t_3 q_2)=0\}$. These two lines are nothing but the two tangent lines to the conic $\{{\empty^t \bm{q}}S^{\dagger}\bm{q}=0\}$ drawn from the point $[\bm{t}]$. We can check this last assertion similarly also in the case that $\rank S=2$ and  ${{\empty}^t\bm{t}}S^{\dagger} \bm{t}\not =0$.

On the other hand, the conic $\{{\empty^t \bm{p}}S\bm{p}=0\}$ dual to the conic $\{{\empty^t \bm{q}}S^{\dagger}\bm{q}=0\}$ appears in the quadric $\{{\empty^t \bm{p}}S\bm{p}+p_4 {\empty^t \bm{p}}\bm{t}=0\}$. Consider the projection of this quadric to $\mP(\bm{p})\simeq \mP^2$ from the $p_4$-point. Then it is easy to check that the projection is isomorphic outside the two points in 
$\{{\empty^t \bm{p}}S\bm{p}=0\}\cap \{{\empty^t \bm{t}}\bm{p}=0\}$ (note that this actually consists of two points by the condition
that ${{\empty}^t\bm{t}}S^{\dagger} \bm{t}\not =0$). Therefore the two points dual to the two tangent lines of $\{{\empty^t \bm{q}}S^{\dagger}\bm{q}=0\}$ appears in this situation.

\subsection{More on singularities}
\label{sub:More}
In this subsection, we use the notation of the subsection \ref{Charts} and the section \ref{13} freely.
We have the following descriptions of the singularities of $\Sigma^{14}_{\mA}$: 
\begin{prop}
\label{prop:SigmaTerm}
The variety $\Sigma^{14}_{\mA}$ has only terminal singularities with the following descriptions:
\begin{enumerate}[(1)]
\item
$\Sigma^{14}_{\mA}$ has $c(G(2,5))$-singularities along the $7$-dimensional loci ${\sf S}_1$--${\sf S}_4$, where we call a singularity isomorphic to the vertex of the cone over $G(2,5)$ a $c(G(2,5))$-singularity.

\item There exists a primitive $K$-negative divisorial extraction
$f\colon \widetilde{\Sigma}\to \Sigma^{14}_{\mA}$
such that 
\begin{enumerate}[(i)]
\item
singularities of $\widetilde{\Sigma}$ are only $c(G(2,5))$-singularities along  the strict transforms of $\mA(p_4, S)$ and $\overline{\sf S}$, and 
\item
for the $f$-exceptional divisor $E_{\Sigma}$, the morphism $f|_{E_{\Sigma}}$ can be identified with $\rho_{\Sigma}\colon \widehat{\Sigma}^{13}\to \mA(S,\bm{t})$ as in the subsection \ref{fib13}.
\end{enumerate}
\end{enumerate}
\end{prop}

\begin{proof}
The descriptions of the singularities of $\Sigma^{14}_{\mA}$ outside $\mA(S,\bm{t})$ are already given in the subsection \ref{Charts}.

Here we check
the singularities of $\Sigma^{14}_{\mA}$ along $\mA(S,\bm{t})$.
Note that, by Proposition \ref{prop:SinglociSigma},
$\mA(S,\bm{t})$ is contained in $\Sing \Sigma^{14}_{\mA}$.

\vspace{3pt}

\noindent {\bf Step 1.} Let $\widetilde{\Sigma}$ be the variety obtained by
blowing up $\Sigma^{14}_{\mA}$ along $\mA(S,\bm{t})$ and taking the reduced
structure of the blow-up. We denote by $f\colon \widetilde{\Sigma}\to \Sigma^{14}_{\mA}$ the natural induced morphism.
We show that 
$\widetilde{\Sigma}$ has only $c(G(2,5))$-singularities as stated in (2) (i).  

A crutial fact is that any equation of $\Sigma^{14}_{\mA}$ is of degree two if 
we regard entries of $S,\bm{t}$ as constants. Therefore, 
$\widetilde{\Sigma}$ has the same singularities as the $p_i$-, $q_j$-, $u$-, $r$-charts $(1\leq i, j\leq 3)$ of $\Sigma^{14}_{\mA}$, and then, by the subsection \ref{Charts},
$\widetilde{\Sigma}$ has also only $c(G(2,5))$-singularities, which are Gorenstein terminal singularities.
Since $\widetilde{\Sigma}$ has only $c(G(2,5))$-singularities, the singular locus of  $\widetilde{\Sigma}$ is a smooth $7$-dimensional closed subset. Moreover, it contains the strict transforms of ${\sf S}_1$--${\sf S}_4$ and, by the descriptions of the charts of $\widetilde{\Sigma}$, the restriction of the singular locus of $\widetilde{\Sigma}$ to the $f$-exceptional divisor is $6$-dimensional. Therefore, the singular locus of $\widetilde{\Sigma}$ is the closure of the strict transforms of ${\sf S}_1$--${\sf S}_4$, and it is nothing but the closure of the strict transoforms of $\mA(p_4,S)$ and $\overline{\sf S}$.

\vspace{3pt}
 
\noindent {\bf Step 2.}
We show that $\Sigma^{14}_{\mA}$ has only Gorenstein terminal singularities also along $\mA(S,\bm{t})$.

Let $E$ be the exceptional divisor of the blow-up $\widetilde{\mA}\to \mA^{18}$ along $\mA(S,\bm{t})$.
Let $P_1,\dots,R, U$ be the projective coordinates of $\widetilde{\mA}$ corresponding to $p_1,\dots,r, u$.
We see that the $f$-exceptional divisor
$E_{\Sigma}=E\cap \widetilde{\Sigma}$ is 
defined with the equation of 
$\Sigma^{14}_{\mA}$
by replacing 
$p_1,\dots, r,u$
with $P_1,\dots,R,U$. Hence $E_{\Sigma}\to \mA(S,\bm{t})$ can be identified with the restriction of $\widehat{\Sigma}^{13}\to \mA(S,\bm{t})$ over $\mA(S,\bm{t})$.
Therefore $E_{\Sigma}$ is irreducible and reduced since so is $\Sigma^{14}_{\mA}$ by Proposition \ref{prop:Tom14} (2).
By Proposition \ref{prop:Tom14} (1) and (3),
the variety 
$\Sigma^{14}_{\mA}$ is normal and Gorenstein, and by Step 1, so is $\widetilde{\Sigma}$.
Thus
we may write 
\begin{equation}
\label{eq:SigmaAdj}
K_{\widetilde{\Sigma}}=f^*K_{\Sigma^{14}_{\mA}}+aE_{\Sigma}
\end{equation}
with some integer $a$.
Note that a general fiber of $E_{\Sigma}\to \mA(S,\bm{t})$ is $\mP^2\times \mP^2$ by the results in the subsection \ref{fib13}.
Therefore we see that $a=2$ restricting (\ref{eq:SigmaAdj}) to a general fiber.
This implies that $\Sigma^{14}_{\mA}$ has only Gorenstein terminal singularities also along $\mA(S,\bm{t})$ since so does $\widetilde{\Sigma}$ by Step 1.

\noindent {\bf Step 3.} Finally we show that $f$ is a $K$-negative primitive divisorial extraction.

Note that $-E_{\Sigma}$ is $f$-ample by a general property of blow-up.
Thus $f$ is $K$-negative by (\ref{eq:SigmaAdj}) with $a=2$.
Since $\Sigma^{14}_{\mA}$ is quasi-affine and $-E_{\Sigma}$ is $f$-ample,
we may find an effective divisor $D$ such that $D$ is linearly equivalent to $n(-E_{\Sigma})$ for sufficiently big $n\in \mN$ and 
$(\widetilde{\Sigma},1/n D)$ is a klt pair since $\widetilde{\Sigma}$ has only terminal singularities. Note that $\widetilde{\Sigma}$ is locally factorial since so is a $c(G(2,5))$-singularity. Therefore, by \cite[Cor.~1.4.1 and 1.4.2]{bchm}, we may run the $K$-MMP over $\Sigma^{14}_{\mA}$ with scale $1/n D$.
By the choice of the scale, the strict transforms of $E_{\Sigma}$ are always negative for the contractions when we run the $K$-MMP.
Therefore the final contraction of the $K$-MMP is the contraction of the strict transforms of $E_{\Sigma}$. We denote by $f'$ this contraction, by ${\Sigma}'$ the source of $f'$ and by ${E}'_{\Sigma}$ the strict transforms of $E_{\Sigma}$ on $\Sigma'$. Since $R_{\Sigma}$ is a UFD by Proposition \ref{prop:Tom14} (3), we see that the target of $f'$ is equal to $\Sigma^{14}_{\mA}$ and 
$f'({E}'_{\Sigma})=\mA(S,\bm{t})$.
Since the $K$-MMP induces a birational map $\widetilde{\Sigma}\dashrightarrow \Sigma'$ which is isomorphism in codimension one, $f$ and $f'$ have the common target $\Sigma^{14}_{\mA}$, and both $f$ and $f'$ are $K$-negative contractions,
we conclude that the induced map $\widetilde{\Sigma}\dashrightarrow \Sigma'$ is actually an isomorphism by \cite[Lem.~5.5]{JAGFano} for example.
\end{proof}

By Lemma \ref{lem:1314} and Proposition \ref{prop:SigmaTerm},
We immediately have the following descriptions of the singularities of 
$\Sigma^{13}_{\mA}$: 
\begin{prop}
\label{prop:SigmaTerm13}
The variety $\Sigma^{13}_{\mA}$ has only terminal singularities with the following descriptions:
\begin{enumerate}[$(1)$]
\item
$\Sigma^{13}_{\mA}$ has $c(G(2,5))$-singularities along the $6$-dimensional loci ${{\sf S}'_3}$ and ${{\sf S}'_4}$.

\item There exists a primitive $K$-negative divisorial extraction
$f'\colon {\widetilde{\Sigma}}^{'}\to \Sigma^{13}_{\mA}$
such that 
\begin{enumerate}[$(i)$]
\item
singularities of $\widetilde{\Sigma}^{'}$ are only $c(G(2,5))$-singularities along the strict transforms of 
$\mA(p_4, S')$ and ${\overline{{\sf S}}'}$, and 
\item
for the $f'$-exceptional divisor $E'_{\Sigma}$, the morphism $f'|_{E'_{\Sigma}}$ can be identified with $\rho_{\Sigma}\colon \widehat{\Sigma}^{13}\to \mA(S,\bm{t})$ over $\mA(S',\bm{t})$.
\end{enumerate}
\end{enumerate}
\end{prop}

\end{document}